\documentclass[a4paper, 11pt]{article}
\usepackage{amsmath}
\numberwithin{equation}{section}
\usepackage{amssymb,esint,hyperref}
\usepackage{amscd}
\usepackage{xspace}
\usepackage{fancyhdr}
\usepackage{color}
\usepackage{verbatim}
\usepackage{cite}
\usepackage{stmaryrd}
\usepackage{bbm}
\usepackage{mathrsfs}
\usepackage{srcltx} %jump between tex and dvi, used for kile
\usepackage{marginnote}
\let\oldmarginpar\marginpar
\renewcommand\marginpar[1]{\-\oldmarginpar[\raggedleft\footnotesize #1]%
	{\raggedright\footnotesize\color{red} #1}} % 注释文字用红色footnote 大小
\newtheorem{theorem}{Theorem}[section]

\newtheorem{lemma}[theorem]{Lemma}

\newtheorem{remark}[theorem]{Remark}

\newtheorem{thm}{Theorem}[section]

\newenvironment{proof}[1][Proof]{\textbf{#1.} }{\hfill\rule{0.5em}{0.5em}}
{\catcode`\@=11\global\let\AddToReset=\@addtoreset
	\AddToReset{equation}{section}
	
	\AddToReset{theorem}{section}

	\title{Local well-posedness of the 1d compressible Navier–Stokes system with rough data}
	\author{
		{\bf Ke Chen,\thanks{E-mail address: kchen18@fudan.edu.cn, Fudan University, 220 Handan Road, Yangpu, Shanghai, 200433, China.}
			~~Ruilin Hu,\thanks {E-mail address: huruilin16@mails.ucas.ac.cn,  Academy of Mathematics and Systems Science, Chinese Academy of Sciences, Beijing, 100190, China.}
			~~Quoc-Hung Nguyen\thanks{E-mail address: qhnguyen@amss.ac.cn, Academy of Mathematics and Systems Science, Chinese Academy of Sciences, Beijing, 100190, China.}}}}
\begin{document}
	\maketitle
	\begin{abstract}
		This paper presents a new approach to the local well-posedness of the 1d compressible Navier–Stokes systems with rough initial data. Our approach is based on establishing some smoothing and Lipschitz-type estimates for the 1d parabolic equation with piecewise continuous coefficients.
	\end{abstract}
	
	\section{Introduction}
	In this paper, we study the compressible Navier--Stokes equations in Lagrangian coordinates, which can be written as (see \cite{Smoller}) 
	\begin{equation}\label{eqcompre}
		\left\{
		\begin{aligned}
			&	v_t-u_x=0,\\
			&	u_t+p_x=\left(\frac{\mu u_x}{v}\right)_x,\\
			&	(e+\frac{1}{2}u^2)_t+({p}u)_x=\left(\frac{\kappa}{v}\theta_x+\frac{\mu}{v}uu_x\right)_x.
		\end{aligned}
		\right.
	\end{equation}
	Here we denote $v$ the specific volume, $u$ the velocity, $p$ the pressure, $e$ the specific internal energy,  $\theta $ the absolute temperature. And $\mu, \kappa>0$ are   viscosity and heat conductivity coefficients. The above equations describe the conservation of mass, momentum and energy, respectively. 
	
	We will consider two systems. The first system is the so-called $p-$system (see \cite{Smoller}), which is a general model of isentropic gas dynamics(when $p(v)=Av^{-\gamma}$). In the isentropic case, the temperature is held constant hence energy must be added to the system, whence the conservation 
	of energy is absent. The $p-$system can be written as  
	\begin{align}\label{inns}
		\left\{
		\begin{aligned}
			&	v_t-u_x=0,\\
			&	u_t+(p(v))_x=(\frac{\mu u_x}{v})_x.
		\end{aligned}
		\right.
	\end{align}
	In this paper we consider a more general system which only requires $p\in W^{2,\infty}$.
	
	Another system studied in this paper is \eqref{eqcompre} for the polytropic ideal gas which has the constitutive relations
	$$
	p(v,\theta)=\frac{K\theta}{v},\ \ \ e=\mathbf{c}\theta,
	$$
	where $K$ and heat capacity $\mathbf{c}$ are both positive constants. The system can be written as 
	\begin{equation}\label{cpns}
		\left\{
		\begin{aligned}
			&	v_t-u_x=0,\\
			&	u_t+(p(v,\theta))_x=\left(\frac{\mu u_x}{v}\right)_x,\\
			&	\theta_t+\frac{p(v,\theta)}{\mathbf{c}}u_x-\frac{\mu}{\mathbf{c}v}(u_x)^2=\left(\frac{\kappa}{\mathbf{c}v}\theta_x\right)_x.
		\end{aligned}
		\right.
	\end{equation}

	Let us give a review of classical results. The study about compressible Navier-Stokes equations in gas dynamics started with Nash in \cite{Nash1958}, where he considered general elliptic and parabolic equations. Then Itaya \cite{Itaya71,Itaya75} solved the compressible Navier-Stokes equations for initial data in H\"{o}lder space. Kazhikhov \cite{Ka77} established a priori estimates and proved the existence of weak and classical solutions. For the equation of 3-d ideal gas, Nishida and Matsumura in \cite{MN1980} applied the energy method to give the existence of a global solution with data in H\"{o}lder space. The result of large-time behavior of global solution was provided by Kawashima \cite{Kawa87}. Using the energy method, Hoff \cite{Hoff92} proved the existence of a global solution for discontinuous data. The results in the book of  Lions \cite{Lions} and the book of Feireisl \cite{Fe} are applicable to a larger space of initial data. 
	
	Recently, Liu and Yu \cite{LiuCAPM} studied the system \eqref{inns} with $BV\cap L^1$ data, under the assumption $p'(v)<0$. Based on exact analysis of the associated  linear equations and their Green’s functions, they initiated a new method, which converts the differential equations into integral equations, and proved that for initial data 
	$v_0,u_0$ satisfying
	\begin{align*}
		\|v_0-1\|_{L^1},\|v_0\|_{BV},\|u_0\|_{L^1},\|u_0\|_{BV}<\delta\ll 1,
	\end{align*}
	there exist $t_\sharp>0$ and $C_\sharp>0$ such that 
	the system \eqref{inns} admits a unique weak solution $(v,u)$ for $t\in(0,t_\sharp)$, and
	\begin{align*}
		\|v(t,\cdot)-1\|_{L^1}+\|v(t,\cdot)\|_{BV}+\|\sqrt{t}u_x(t,\cdot)\|_{L^\infty}\leq 2C_\sharp\delta,\ \ \ \forall t\in(0,t_\sharp).
	\end{align*}
	Moreover, they also established global result for polytropic gases $p(v)=Av^{-\gamma}$ with $1\leq \gamma<e$. 
	Their result was extended to the full compressible Navier--Stokes system \eqref{cpns} by Wang, Yu, and Zhang \cite{Wang}. 
	
	In this paper, we are interested in the local well-posedness of  \eqref{inns} and \eqref{cpns} with rough data, which allows $v_0$ to have a finite jump. The idea to study the elliptic problem with jump coefficients comes from \cite{N1,N2}. More precisely, without loss of generality, let $x=0$ be the jump point. Suppose there exists $\varepsilon>0$ such that 
	\begin{align}\label{jump}
		|v_0(x)-v_0(y)|\leq \delta,\ \ \ \text{if} \ |x-y|\leq \varepsilon \ \text{and}\ xy>0,
	\end{align} 
	for some $\delta>0$ that will be fixed later.
	
	Throughout the paper, we fix two constants $0<\alpha<\gamma\ll1$. Define the norms
	\begin{align*}
		&	\|f\|_{X_T^{\sigma,p}(E)}=\sup_{0<s<T}s^{\sigma}\|f(s)\|_{L^p(E)}+\sup_{0<s<t<T}s^{\sigma+\alpha}\frac{\|f(t)-f(s)\|_{L^p(E)}}{(t-s)^\alpha},\\
		&	\|f\|_{Y_T(E)}=\sup_{0<s<T}\|f(s)\|_{L^\infty(E)}+\sup_{0<s<t<T}s^{\alpha}\frac{\|f(t)-f(s)\|_{L^\infty(E)}}{(t-s)^\alpha},\\
		&\|f\|_{L^p_{T}(E)}=\|f\|_{L^p_{T}(E)}=\left(\int_0^T \int_E |f(t,x)|^p dxdt \right)^\frac{1}{p},\ \ \ \ \ \ \ \text{for any measurable set}\ E\subseteq\mathbb{R},
	\end{align*}
	and we omit the domain if $E=\mathbb{R}$. We write several norms together to denote the sum of them. Moreover, we denote $A\lesssim B$ if there exists a universal constant $C$ such that $A\leq CB$. 
	We will solve the system (\ref{inns}) and \eqref{cpns} by the fixed point theorem in the Banach space equipped with the norms above. We note that the norms $\|\cdot\|_{X_T^{\sigma,p}}$ and $\|\cdot\|_{Y_T}$ contain  H\"{o}lder derivatives in time.
	
	For the system \eqref{inns}, we impose the following conditions for initial data
	\begin{align}\label{lowbd}
		v_0\in L^\infty,\  \inf_xv_0\geq\lambda_0>0
		,\ u_0=\partial_x\bar{u}_0\ \operatorname{with}\ \bar{u}_0\in \dot{C}^{2\gamma}.
	\end{align}
	Simplified versions of our main results are stated in the following theorems.
	\begin{theorem}\label{maininns} Assume $p\in W^{2,\infty}(\mathbb{R})$.
		There exists $\delta_0>0$ such that if initial data $(v_0,u_0)$ satisfies the conditions (\ref{jump}) and (\ref{lowbd}) with $\delta=\delta_0$, then  the system \eqref{inns} admits a unique local solution $(v,u)$ in $[0,T]$ satisfying 
		\begin{align*}
			\inf_{t\in[0,T]}\inf_xv(t,x)\geq \frac{\lambda_0}{2},\ \ \ \|v\|_{Y_T}\leq 2\|v_0\|_{L^\infty},\ \ \ \ \|\partial_xu\|_{X_T^{1-\gamma,\infty}}\leq M,	
		\end{align*}
		for some $T,M>0$.
	\end{theorem} 
	%\begin{remark}
	%	By Theorem \ref{maininns}, we can see that to prove global wellposedness for  initial data $(v_0,u_0)\in L^\infty\times \dot B^{-1+2\gamma}_{\infty,\infty}$, it suffices to prove for initial data in $L^\infty\times(  \dot W^{1,\infty}\cap \dot B^{-1+2\gamma}_{\infty,\infty})$. 
	%\end{remark}

	For the full N--S system \eqref{cpns}, we  use $L^p$ setting, and set the following norm and corresponding Banach space
	\begin{align*}
		\|(w,\vartheta)\|_{Z_T}:=&\sum_{\star\in\{L^2_{T},X_T^{\frac{1}{2},2},X_T^{\frac{3}{4},\infty}\}}\|\partial_x w\|_{\star} +\sum_{\star\in\{L^2_{T},X_T^{\frac{1}{2},2}\}}\|\vartheta\|_{\star}+\sum_{\star\in\{L^\frac{6}{5}_{T},X_T^{\frac{5}{6},\frac{6}{5}}\}}\|\partial_x \vartheta\|_{\star}.
	\end{align*}
	The above norms have the same scaling in the parabolic setting. The norms $\|\partial_x w\|_{L^2_{T}}$, $\|\vartheta\|_{L^2_{T}}$ and $\|\partial_x \vartheta\|_{L^\frac{6}{5}_{T}}$ are imposed in view of the parabolic structure of the equation of $(u,\theta)$. Other norms $\|\partial_x w\|_{X_T^{\frac{1}{2},2}}, \|\partial_x w\|_{X_T^{\frac{3}{4},\infty}}, \|\vartheta\|_{X_T^{\frac{1}{2},2}}, \|\partial_x \vartheta\|_{X_T^{\frac{5}{6},\frac{6}{5}}}$ are necessary to obtain local estimates of $v$.
	
	We assume the initial data satisfies that
	\begin{align}\label{idcns}
		\inf_xv_0(x)\geq \lambda_0>0,\ \ \	\|v_0\|_{L^\infty}<\infty,\ \ \|u_0\|_{L^2}<\infty,\ \ \|\theta_0\|_{\dot W^{-\frac{2}{3},\frac{6}{5}}}<\infty.
	\end{align}
	The following result is the local well-posedness of system \eqref{cpns}.
	\begin{theorem}\label{maincp}
		There exists $\delta_0>0$ such that if initial data $(v_0,u_0,\theta_0)$ satisfies \eqref{jump} and \eqref{idcns} with $\delta=\delta_0$, then  the system \eqref{cpns} has a unique local solution $(v,u)$ in $[0,T]$ satisfying 
		\begin{align*}
			\inf_{t\in[0,T]}\inf_xv(t,x)\geq \frac{\lambda_0}{2},\ \ \ \|v\|_{Y_T}\leq 2\|v_0\|_{L^\infty},\ \ \ \ \|(u,\theta)\|_{ Z_T}\leq B.	
		\end{align*}
		for some $T, B>0$. 
	\end{theorem}
	Note that we do not need $u_0,\theta_0$ decay at infinity. Hence the condition \eqref{idcns} can be relaxed to local $L^p$ norms(see Theorem \ref{thmloc}):
	\begin{align}\label{loccon}
		\inf_xv_0(x)\geq \lambda_0>0,\ \ \	\|v_0\|_{L^\infty}<\infty,\ \ \sup_{z\in\mathbb{R}}\|u_0\chi_z\|_{L^2}<\infty,\ \ \sup_{z\in\mathbb{R}}\|\theta_0\chi_z\|_{\dot W^{-\frac{2}{3},\frac{6}{5}}}<\infty.
	\end{align}
	Here $\chi_z$ is a smooth cutoff function satisfying $\mathbf{1}_{[z-1,z+1]}\leq \chi_z\leq \mathbf{1}_{[z-2,z+2]}$, where $\mathbf{1}$ is the indicator function. 
	
	Our method is based on analysis of heat equation with jump coefficient $\phi(x)$ satisfying \eqref{conphi},
	\begin{equation*}
		\begin{aligned}
			&\partial_t f(t,x) -\partial_x(\phi(x)\partial_xf(t,x))=\partial_x F(t,x)+R(t,x),\\
			&f(0,x)=f_0(x).
		\end{aligned}	
	\end{equation*}
	We convert the differential equations into integral equations and give two different formulas of solution. One formula \eqref{soforjum}  is suitable for estimates near jump points, and another one \eqref{sofor1} behaves well for estimates far from jump points. Our main results rely on Lemma \ref{lemma}.
	
	We make some remarks about the choice of initial data. For the isentropic system \eqref{inns}, we requires $u_0\in \dot B^{-1+2\gamma}_{\infty,\infty}$. Generally in Schauder theory of parabolic/elliptic equations, to control the Lipschitz norm of solution, it requires the coefficient to be at least Dini (see \cite{Jin}). Due to the roughness of coefficient, the space $\dot B^{-1+2\gamma}_{\infty,\infty}$ is optimal in the sense that we can take $\gamma>0$ arbitrarily close to 0. For the full N--S system \eqref{inns}, we require $u_0\in L^2$ and $\theta_0\in \dot W^{-\frac{2}{3},\frac{6}{5}}$. Taking a glance at the equation of $\theta$, we need $(u_x)^2, \theta u_x\in L^1_{T}$ to make $\theta$ well-defined. By classical theory of heat equation, $u_x\in L^2_{T}$ requires $u_0\in L^2$, and $\theta\in L^2_{T}$ requires $\theta_0\in \dot H^{-1}$. Taking force terms into consideration, we finally choose $\theta_0\in\dot W^{-\frac{2}{3},\frac{6}{5}}\subset \dot H^{-1}$.

	% We define the following local norms
	%\begin{align*}
	%	&	\|h\|_{\tilde L^p}:=\sup_{z}\|h\|_{L^p([z,z+1])},\ \ \ 
	%	\|h\|_{\tilde W ^{s,p} }:=\sup_{z}\|h\chi_z\|_{\dot W ^{s,p}},
	%\end{align*}
	%for $h:\mathbb{R}\to \mathbb{R}$, $f:\mathbb{R}^+\times\mathbb{R}\to\mathbb{R}$.  Here $\chi_z$ is a smooth cutoff function satisfying $\mathbf{1}_{[z-1,z+1]}\leq \chi_z\leq \mathbf{1}_{[z-2,z+2]}$. 
	%Moreover, we denote $\|f\|_{\tilde L^p_{T}}=\|f\|_{L^p_t\tilde L^p_x}$.\\
	%Following is the last main result in this paper. 
	%\begin{theorem}{(Corollary of Theorem \ref{thmloc})}\label{mainloc}
	%	There exists $\delta_0>0$ such that for any initial data $(v_0,u_0,\theta_0)$ satisfying  \eqref{jump}and \eqref{idcns} , there exists $B,T>0$ such that 
	%	\begin{align*}
	%		\inf_{t\in[0,T]}\inf_xv(t,x)\geq \frac{\lambda_0}{2},\ \ \ \|v\|_{Y_T}\leq 2\|v_0\|_{L^\infty},\ \ \ \ \|(u,\theta)\|_{\tilde  Z_T}\leq B.	
	%	\end{align*}
	%\end{theorem}
	The rest of this paper is organized as follows. We introduce some primary properties for heat equations and heat kernel in Section \ref{secpre}. We establish the main lemma \ref{lemma} in Section \ref{secmain}. In Section 4 we apply Lemma \ref{lemma} to obtain local well-posedness for system (\ref{inns}) and prove Theorem \ref{maininns}. Section 5 is devoted to prove local wellposedness for  system (\ref{cpns}) and prove Theorem \ref{maincp}. 
	
	\section{Preliminaries}\label{secpre}
	We introduce the following estimates for the standard heat kernel $\mathbf{K}(t,x)=(4\pi t)^{-\frac{1}{2}}e^{-\frac{x^2}{4t}}$ in $\mathbb{R}$ which satisfies 
	\begin{align*}
		&\partial_t \mathbf{K}(t,x)-\partial_{xx}\mathbf{K}(t,x)=0,\\
		&\lim_{t\to 0^+}\mathbf{K}(t,x)=\mathbf{Dirac}(x).
	\end{align*}
	The following estimates for  heat kernel will be used frequently in our proof.
	\begin{lemma}\label{lemheat}
		There holds 
		%	\begin{align*}
		%		& \int|\partial_t K(t,x)||x|^\sigma dx\lesssim   \frac{1}{t^{1-\sigma}},\ \forall \sigma\in[0,1].\\
		%		&\int|\partial_t K(t+a,x)-\partial_t K(t,x)|dx\lesssim \frac{1}{t}\min\{1,\frac{a}{t}\},\ \ \forall a,t\geq 0.
		%	\end{align*}
		\begin{align*}
			&|\partial_t^j\partial_x^m \mathbf{K}(t,x)|\lesssim \frac{1}{(t^\frac{1}{2}+|x|)^{1+2j+m}},\\
			&\left(\int_{0}^{\infty}|\partial_x^m\mathbf{K}(t,x)|^pdt\right)^{\frac{1}{p}}\lesssim |x|^{\frac{2}{p}-1-m},\ \ \ \  \left(\int_{\mathbb{R}}|\partial_x^m\mathbf{K}(t,x)|^pdx\right)^{\frac{1}{p}}\lesssim t^{\frac{1}{2}(\frac{1}{p}-1-m)},\\
			&\left(\int_{\mathbb{R}}|\partial_t^j\partial_x^m \mathbf{K}(t,x)|^p|x|^{\sigma p} dx\right)^\frac{1}{p}\lesssim   \frac{1}{t^{j+{\frac{1}{2}(m-\frac{1}{p}+1-\sigma)}}},\\
			&\left(\int_{\mathbb{R}}|\partial_t^j \partial_x^m\mathbf{K}(t+a,x)-\partial_t^j\partial_x^m \mathbf{K}(t,x)|^p|x|^{\sigma p} dx\right)^\frac{1}{p}\lesssim \frac{1}{t^{j+{\frac{1}{2}(m-\frac{1}{p}+1-\sigma)}}}\min\{1,\frac{a}{t}\},
		\end{align*}
		for any $m,j=0,1,2$, any $\sigma,a,t\geq 0$ and $p\in[1,+\infty]$.
	\end{lemma}
	It is easy to check the above estimates by the definition of heat kernel and the fact that $b^me^{-b}\lesssim_m 1$, $\forall b>0,m\in\mathbb{N}^+$. We omit details here.
	
	\begin{lemma}\label{lemsob}
		Let $h\in \dot W^{\sigma,p}(\mathbb{R})$ for some $\sigma\in(0,1), p\in[1,\infty)$. Define $\mathbf{h}(x)=h(x)\mathbf{1}_{x\geq 0}+h(-x)\mathbf{1}_{x<0}$ and $\bar{ \mathbf{h}}(x)=(h(x)-h(-x))\mathbf{1}_{x\leq 0}$. Then we have 
		\begin{align*}
			\int_{\mathbb{R}} \int_0^\infty |\mathbf{h}(x)-\mathbf{h}(x-z)|^p\frac{dxdz}{|z|^{1+\sigma p}}+	\int_{\mathbb{R}} \int_0^\infty |\bar{ \mathbf{h}}(x)-\bar{ \mathbf{h}}(x-z)|^p\frac{dxdz}{|z|^{1+\sigma p}}\lesssim \|h\|_{\dot W^{\sigma,p}}^p.
		\end{align*}
	\end{lemma}
	\begin{proof}
		By definition, one can check that 
		\begin{align*}
			\frac{|\mathbf{h}(x)-\mathbf{h}(x-z)|^p}{|z|^{1+\sigma p}}&=\mathbf{1}_{x\geq z}	\frac{|{h}(x)-{h}(x-z)|^p}{|z|^{1+\sigma p}}+\mathbf{1}_{x< z}	\frac{|{h}(x)-{h}(z-x)|^p}{|z|^{1+\sigma p}}\\
			&\leq	\frac{|{h}(x)-{h}(x-z)|^p}{|z|^{1+\sigma p}}+	\frac{|{h}(x)-{h}(z-x)|^p}{|2x-z|^{1+\sigma p}},\\
			\frac{|\bar {\mathbf{h}}(x)-\bar {\mathbf{h}}(x-z)|^p}{|z|^{1+\sigma p}}&=\mathbf{1}_{x< z}	\frac{|{h}(x-z)-{h}(z-x)|^p}{|z|^{1+\sigma p}}\\
			&\leq \frac{|{h}(x)-{h}(x-z)|^p}{|z|^{1+\sigma p}}+	\frac{|{h}(x)-{h}(z-x)|^p}{|2x-z|^{1+\sigma p}}.
		\end{align*}
		The result is standard in view of the fact that $$\|h\|_{\dot W^{\sigma,p}}^p= \iint_{\mathbb{R}^2} \frac{|h(x)-h(y)|^p}{|x-y|^{1+\sigma p}}dxdy.
		$$
		This completes the proof.
	\end{proof}

	We introduce the following Schauder type lemma, the idea can be found in  \cite{Peskin}.
	\begin{lemma}\label{mainlem}
		Let $\mathbf{K}$ be the heat kernel, and let
		\begin{align*}
			g(t,x)=\int_0^t\int_{\mathbb{R}}\partial_t \mathbf{K}(t-\tau,x-y) f(\tau,y)dyd\tau.
		\end{align*}
		Then we have
		\begin{align}\label{re}
			\|g\|_{X_T^{\sigma,p}}\lesssim \|f\|_{X_T^{\sigma,p}},\ \ \ \forall T>0,\  \sigma\in(0,1-\alpha),\  p\in[1,\infty].
		\end{align}
	\end{lemma}
	\begin{proof}
		One has	\begin{align*}
			g(t,x)&=\int_0^t\int_{\mathbb{R}}\partial_t \mathbf{K}(t-\tau,x-y) (f(\tau,y)-f(t,y))dyd\tau+\int_0^t\int_{\mathbb{R}}\partial_t \mathbf{K}(t-\tau,x-y) f(t,y)dyd\tau\\
			&:=g_1(t,x)+g_2(t,x).
		\end{align*}
		By Lemma \ref{lemheat} we get 
		\begin{align*}
			\|g_1(t)\|_{L^p}\lesssim \int_0^t \|\partial_t \mathbf{K}(t-\tau)\|_{L^1}\|f(\tau)-f(t)\|_{L^p}d\tau\lesssim \int_0^t  (t-\tau)^{\alpha-1}\tau^{-\sigma-\alpha}d\tau\|f\|_{X_T^{\sigma,p}}\lesssim t^{-\sigma}\|f\|_{X_T^{\sigma,p}}.
		\end{align*}
		Moreover, observe that $\int_0^t \partial_t \mathbf{K}(t-\tau,x-y) d\tau=K(t,x-y)-\mathbf{Dirac}(x-y)$, hence $$g_2(t,x)=\int_{\mathbb{R}} \mathbf{K}(t,x-y)f(t,y)dy-f(t,x).$$
		Then 
		\begin{align*}
			\|g_2(t)\|_{L^p}\lesssim (1+\|\mathbf{K}(t)\|_{L^1})\|f(t)\|_{L^p}\lesssim \|f(t)\|_{L^p}\lesssim t^{-\sigma}\|f\|_{X_T^{\sigma,p}}.
		\end{align*} 
		Thus, we obtain that 
		\begin{align}\label{par1}
			\sup_{s\in[0,T]}s^{\sigma}\|g(s)\|_{L^p}\lesssim \|f\|_{X_T^{\sigma,p}}.
		\end{align}
		In the following we will denote $a=t-s>0$ for convenience, and we denote $\delta_\beta f(s,x)=f(s+\beta,x)-f(s,x)$. We write
		\begin{align*}
			g(t)-	g(s)=&\int_0^s\delta_a\partial_t \mathbf{K}(s-\tau)\ast f(\tau)d\tau+\int_s^t\partial_t \mathbf{K}(t-\tau)\ast f(\tau)d\tau\\
			=&\int_0^s\delta_a\partial_t \mathbf{K}(s-\tau)\ast (f(\tau)-f(s))d\tau+\int_s^t\partial_t \mathbf{K}(t-\tau)\ast (f(\tau)-f(t))d\tau\\
			&+\big(\int_0^s\delta_a\partial_t \mathbf{K}(s-\tau)\ast f(s)d\tau+\int_s^t\partial_t \mathbf{K}(t-\tau)\ast f(t)d\tau\big)\\
			:=&I_1+I_2+I_3.
		\end{align*}
		Applying Lemma \ref{lemheat}, we have 
		\begin{align*}
			\|I_1\|_{L^p_x}&\lesssim \int_0^s \min\{1,\frac{a}{s-\tau}\}\frac{1}{(s-\tau)^{1-\alpha}\tau^{\sigma+\alpha}} d\tau \|f\|_{X_T^{\sigma,p}}\\
			&\lesssim a^\alpha s^{-\sigma-\alpha}\|f\|_{X_T^{\sigma,p}},
		\end{align*}
		where we use the fact that 
		\begin{align}
			&\int_0^s \min\{1,\frac{a}{s-\tau}\}\frac{1}{(s-\tau)^{1-\alpha}\tau^{\sigma+\alpha}} d\tau\nonumber\\
			&\lesssim a^\alpha\int_0^\frac{s}{2}\frac{1}{(s-\tau)\tau^{\sigma+\alpha}}d\tau+s^{-\sigma-\alpha}\int_\frac{s}{2}^s\min\{1,\frac{a}{s-\tau}\}\frac{1}{(s-\tau)^{1-\alpha}}d\tau\lesssim a^\alpha s^{-\sigma-\alpha}.\label{intt}
		\end{align}
		For $I_2$, we have 
		\begin{align*}
			\|I_2\|_{L^p}\lesssim \int_s^t \frac{1}{(t-\tau)^{1-\alpha}}\frac{1}{\tau^{\sigma+\alpha}}d\tau \|f\|_{X_T^{\sigma,p}}\lesssim a^\alpha s^{-\sigma-\alpha}\|f\|_{X_T^{\sigma,p}}.
		\end{align*}
		For $I_3$, we have
		\begin{align*}
			\|I_3\|_{L^p}\lesssim &\|f(s)\ast \mathbf{K}(t-s)-f(s)\ast \mathbf{K}(t)-f(s)+f(s)\ast \mathbf{K}(s)+f(t)-f(t)\ast \mathbf{K}(t-s)\|_{L^p}\\
			\lesssim &\|f(t)-f(s)\|_{L^p}+\|f(s)\ast(\mathbf{K}(t)- \mathbf{K}(s))\|_{L^p}
			\lesssim s^{-\sigma-\alpha}(t-s)^\alpha\|f\|_{X_T^{\sigma,p}},
		\end{align*}
		where we applied Lemma \ref{lemheat} to get $\|\mathbf{K}(t)- \mathbf{K}(s)\|_{L^1}\lesssim \frac{(t-s)^\alpha}{s^\alpha}$ in the last inequality.
		Then we obtain that 
		\begin{align}
			\label{par2}
			\sup_{0<s<t<T}s^{1-\gamma+\alpha}\frac{\|g(t)-g(s)\|_{L^p}}{(t-s)^\alpha}\lesssim \|f\|_{X_T^{\sigma,p}}.
		\end{align}
		By \eqref{par1} and \eqref{par2}, we get \eqref{re}. The proof is complete.
	\end{proof} \\ 
	\begin{remark}\label{rem2.2}
		The result in Lemma \ref{mainlem} is still true for $g(t,x)=\int_0^t\int_{\mathbb{R}}\partial_t\mathbf{G}(t-\tau,x-y,x)f(\tau,y)dy$, where the kernel $\mathbf{G}(t,z,x)$ satisfies
		\begin{align*}
			&\sup_z|\partial_t\mathbf{G}(t,z,x)|\lesssim |\partial_t\mathbf{K}(\zeta t,x)|,\\
			&\sup_z|\partial_t\mathbf{G}(t,z,x)-\partial_t\mathbf{G}(s,z,x)|\lesssim |\partial_t\mathbf{K}(\zeta s,x)-\partial_t\mathbf{K}(\zeta t,x)|,
		\end{align*}
		for a universal constant $\zeta>0$.
	\end{remark}
	\begin{lemma}\label{lemlow}
		Let  $0\leq\beta< 1-\alpha$; and  $\tilde{\sigma}<1$, $1\leq q\leq p\leq\infty$, $1+\frac{1}{p}=\frac{1}{r}+\frac{1}{q}$, $1+\sigma-\tilde{\sigma}-\beta\geq 0$. Suppose that  $G(t,y,x)$ satisfies \begin{align}\label{conG}
			\sup_x\|G(t,\cdot,x)\|_{L^r}\lesssim t^{-\beta},\ \ \  	\sup_x\|G(t,\cdot,x)-G(s,\cdot,x)\|_{L^r}\lesssim s^{-\beta}\min\{1,\frac{t-s}{s}\},
		\end{align}
		for any $0<s<t<T$.
		Then for $g(t,x)=\int_0^t\int G(t-\tau,x-y,x) f(\tau,y)dyd\tau$, we have 
		$$
		\|g\|_{X_T^{\sigma,p}}\lesssim T^{1+\sigma-\tilde{\sigma}-\beta }\|f\|_{X_T^{\tilde{\sigma},q}}.
		$$
	\end{lemma}
	\begin{proof}
		By Young's inequality, 
		\begin{align*}
			\|g(t)\|_{L^p}&\lesssim \int_0^t 	\sup_x\|G(t-\tau,\cdot,x)\|_{L^r}\|f(\tau)\|_{L^q}d\tau\\
			&\lesssim \int_0^t \frac{1}{(t-\tau)^{\beta}}\frac{1}{\tau^{\tilde{\sigma}}}d\tau \|f\|_{X_T^{\tilde{\sigma},q}}\lesssim t^{1-\beta-\tilde{\sigma}} \|f\|_{X_T^{\tilde{\sigma},q}}.
		\end{align*}
		Hence 
		\begin{align}\label{ggg}
			\sup_{t\in[0,T]}t^{\sigma}\|g(t)\|_{L^p}\lesssim T^{1-\beta+\sigma-\tilde{\sigma} }\|f\|_{X_T^{\tilde{\sigma},q}}.
		\end{align}
		Moreover, for any $0<s<t<T$, denote $a=t-s$,
		\begin{align*}
			g(t)-	g(s)=&\int_0^s\int_{\mathbb{R}}\delta_aG(s-\tau,x-y,x) f(\tau,y)dyd\tau+\int_s^t\int_{\mathbb{R}} G(t-\tau,x-y,x) f(\tau,y)dyd\tau.
		\end{align*}
		By \eqref{conG} and Young's inequality, one has 
		\begin{align*}
			\|	g(t)-	g(s)\|_{L^p}&\lesssim \|f\|_{X_T^{\tilde{\sigma},q}}\int_0^s\frac{\min\{1,\frac{a}{s-\tau}\}}{(s-\tau)^{\beta}}\frac{1}{\tau^{\tilde{\sigma}}}d\tau+\|f\|_{X_T^{\tilde{\sigma},q}}\int_s^t\frac{1}{(t-\tau)^{\beta}}\frac{1}{\tau^{\tilde{\sigma}}}d\tau\\
			&\overset{\eqref{intt}}\lesssim a^\alpha s^{-\sigma-\alpha} T^{1-\beta+\sigma-\tilde{\sigma} } \|f\|_{X_T^{\tilde{\sigma},q}}.
		\end{align*}
		Combining this with \eqref{ggg}, we obtain the result.
	\end{proof} 
	\begin{lemma}\label{rem}
		Suppose 
		$$
		g(t,x)=\int_0^t\int_{\mathbb{R}}\mathbf{K}(t-\tau,x-y)R(\tau,y)dyd\tau.
		$$
		Let $l=0,1$. Then for any  $\frac{1}{1+2\alpha}<p<\frac{1}{2\alpha}$ and $\sigma\geq \frac{1+l}{2}-\frac{1}{2p}$,  there holds 
		$$
		\|\partial_x^lg\|_{X_T^{\sigma,p}}\lesssim T^{\sigma-\frac{1+l}{2}+\frac{1}{2p}}(\|R\|_{L_{T}^1}+\|R\|_{X_T^{1,1}}).
		$$
	\end{lemma}
	\begin{proof}
		Note that 
		\begin{align*}
			\partial_x^lg(t,x)=\int_0^t\int_{\mathbb{R}}\partial_x^l\mathbf{K}(t-\tau,x-y)R(\tau,y)dyd\tau.
		\end{align*}
		By Young's inequality and Lemma \ref{lemheat},
		\begin{align*}
			\|\partial_x^lg(t)\|_{L^p}\lesssim \int_0^t \|\partial_x^l \mathbf{K}(t-\tau)\|_{L^p}\|R(\tau)\|_{L^1}d\tau\lesssim \int_0^t (t-\tau)^{\frac{1}{2p}-\frac{1+l}{2}}\|R(\tau)\|_{L^1}d\tau.
		\end{align*}
		We divide the integral into two parts. When $\tau\in[0,{t/2}]$, we have $(t-\tau)^{\frac{1}{2p}-\frac{1+l}{2}}\sim t^{\frac{1}{2p}-\frac{1+l}{2}}$. And when $\tau\in[{t/2},t]$, one has $\|R(\tau)\|_{L^1}\lesssim t^{-1}\|R\|_{X_T^{1,1}}$. This yields that 
		\begin{align}\label{in}
			\|\partial_x^lg(t)\|_{L^p}\lesssim t^{\frac{1}{2p}-\frac{1+l}{2}}(\|R\|_{L^1_{T}}+\|R\|_{X_T^{1,1}}).
		\end{align}
		It remains to check the H\"{o}lder norms. Let $0<s<t<T$ and $a=t-s$, one has
		\begin{align*}
			&\partial_x^l g(t,x)-\partial_x^lg(s,x)\\
			&=\int_0^s\int_{\mathbb{R}}\delta_a\partial_x^l\mathbf{K}(s-\tau,x-y)R(\tau,y)dyd\tau+\int_s^t\int_{\mathbb{R}} \partial_x^l\mathbf{K}(t-\tau,x-y)R(\tau,y)dyd\tau:=\mathcal{R}_1+\mathcal{R}_2.
		\end{align*} 
		By Young's inequality and Lemma \ref{lemheat},
		\begin{align*}
			\|\mathcal{R}_1\|_{L^p}\lesssim \int_0^s\|\delta_a \partial_x^l\mathbf{K}(s-\tau)\|_{L^p}\|R(\tau)\|_{L^1}d\tau\lesssim \int_0^s(s-\tau)^{\frac{1}{2p}-\frac{1+l}{2}}\min\{1,\frac{a}{s-\tau}\}\|R(\tau)\|_{L^1}d\tau.
		\end{align*}
		We discuss $\tau\in[0,s/2]$ and $\tau\in[s/2,s]$ seperately and get that 
		\begin{align*}
			\|\mathcal{R}_1\|_{L^p}\lesssim a^\alpha s^{\frac{1}{2p}-\frac{1+l}{2}-\alpha}\int_0^{s/2}\|R(\tau)\|_{L^1}d\tau+s^{-1}\|R\|_{L^1_{T}}\int_{s/2}^s (s-\tau)^{\frac{1}{2p}-\frac{1+l}{2}} \min\{1,\frac{a}{s-\tau}\}d\tau.
		\end{align*}
		By a change of variable, we can write 
		\begin{align*}
			&\int_{s/2}^s (s-\tau)^{\frac{1}{2p}-\frac{1+l}{2}} \min\{1,\frac{a}{s-\tau}\}d\tau=\int_0^{s/2}\tau^{\frac{1}{2p}-\frac{1+l}{2}} \min\{1,\frac{a}{\tau}\}d\tau\\
			&\lesssim \mathbf{1}_{a\leq  s/2}\left(\int_0^a\tau^{\frac{1}{2p}-\frac{1+l}{2}} d\tau+a\int_a^{s/2}\tau^{\frac{1}{2p}-\frac{1+l}{2}-1} d\tau\right)+\mathbf{1}_{a\geq s/2}\int_0^s\tau^{\frac{1}{2p}-\frac{1+l}{2}} d\tau\\
			&\lesssim \mathbf{1}_{a\leq  s/2} a^{\frac{1}{2p}+\frac{1-l}{2}}+\mathbf{1}_{a\geq s/2}s^{\frac{1}{2p}+\frac{1-l}{2}}\lesssim a^\alpha s^{\frac{1}{2p}+\frac{1-l}{2}-\alpha},
		\end{align*}
		provided $\frac{1}{2p}+\frac{1-l}{2}-\alpha\geq0$. Hence 
		\begin{align}\label{r1}
			\|\mathcal{R}_1\|_{L^p}\lesssim a^\alpha s^{\frac{1}{2p}-\frac{1+l}{2}-\alpha}(\|R\|_{L^1_{T}}+\|R\|_{X_T^{1,1}}).
		\end{align}
		Then we deal with $\mathcal{R}_2$. By Young's inequality,
		$$
		\|\mathcal{R}_2\|_{L^p}\lesssim\int_s^t \| \partial_x^l\mathbf{K}(t-\tau)\|_{L^p}\|R(\tau)\|_{L^1}d\tau\lesssim \int_s^t (t-\tau)^{\frac{1}{2p}-\frac{1+l}{2}}\|R(\tau)\|_{L^1}d\tau.
		$$
		We consider two cases. If $s<t/2$, then one has $a\sim t$. Hence $$
		\int_s^t (t-\tau)^{\frac{1}{2p}-\frac{1+l}{2}}\|R(\tau)\|_{L^1}d\tau\lesssim t^{\frac{1}{2p}-\frac{1+l}{2}}(\|R\|_{L^1_{T}}+\|R\|_{X_T^{1,1}})
		\lesssim a ^{\alpha}s^{\frac{1}{2p}-\frac{1+l}{2}-\alpha}(\|R\|_{L^1_{T}}+\|R\|_{X_T^{1,1}}),
		$$
		provided $\frac{1}{2p}-\frac{1+l}{2}-\alpha<0$.\\
		If $s\geq t/2$, we can write 
		\begin{align*}
			\int_s^t (t-\tau)^{\frac{1}{2p}-\frac{1+l}{2}}\|R(\tau)\|_{L^1}d\tau&\lesssim 	\int_s^t (t-\tau)^{\frac{1}{2p}-\frac{1+l}{2}}d\tau t^{-1}\|R\|_{X_T^{1,1}}\lesssim a^{\frac{1}{2p}+\frac{1-l}{2}}t^{-1}\|R\|_{X_T^{1,1}}\\
			&\lesssim a ^{\alpha}s^{\frac{1}{2p}-\frac{1+l}{2}-\alpha}\|R\|_{X_T^{1,1}}.
		\end{align*}
		We conclude that 
		$$
		\|\mathcal{R}_2\|_{L^p}\lesssim a ^{\alpha}s^{\frac{1}{2p}-\frac{1+l}{2}-\alpha}(\|R\|_{L^1_{T}}+\|R\|_{X_T^{1,1}}).
		$$
		Combining this with \eqref{r1} yields that 
		\begin{align}\label{in2}
			\|\partial_x^l g(t)-\partial_x^lg(s)\|_{L^p}\lesssim (t-s) ^{\alpha}s^{\frac{1}{2p}-\frac{1+l}{2}-\alpha}(\|R\|_{L^1_{T}}+\|R\|_{X_T^{1,1}}).
		\end{align}
		Then we complete the proof in view of \eqref{in} and \eqref{in2}.
	\end{proof}
	\begin{remark}
		The result in Lemma \ref{rem} is still true for $g(t,x)=\int_0^t\int \mathbf{G}(t-\tau,x,x-y)R(\tau,y)dyd\tau$, where $\mathbf{G}(t,x,y)$ satisfies that
		\begin{align*}
			&\sup_{z}\|\partial^l_2\mathbf{G}(t,z,\cdot)\|_{L^p}\lesssim t^{-\frac{1+l}{2}+\frac{1}{2p}},\\
			&\sup_{z}\|\partial^l_2\mathbf{G}(t,z,\cdot)-\partial^l_2\mathbf{G}(s,z,\cdot)\|_{L^p}\lesssim s^{-\frac{1+l}{2}+\frac{1}{2p}}\min\{1,\frac{t-s}{s}\}, \ \text{for}\ l=0,1,\ p\in[1,\infty].
		\end{align*}
	\end{remark}

	\section{Main lemma and its proof}\label{secmain}
	We consider the system \eqref{inns} and the system \eqref{cpns} with jump initial data $v_0$. Observe that the equation for $v$ is hyperbolic, $v(t)$ will have jump at some points for any $t>0$. To deal with $u$ in \eqref{inns} and $(u,\theta)$ in \eqref{cpns}, we need to 
	consider the following parabolic equation with jump coefficient.
	\begin{equation}\label{eqpara}
		\begin{aligned}
			&\partial_t f(t,x) -\partial_x(\phi(x)\partial_xf(t,x))=\partial_x F(t,x)+R(t,x),\\
			&f(0,x)=f_0(x).
		\end{aligned}	
	\end{equation}
	Here the coefficient function satisfies $0<c_0\leq \phi(x)\leq c_0^{-1}, \forall x\in \mathbb{R}$, and there exist $\varepsilon>0$, $\{a_n\}_{n=1}^N\subset\mathbb{R}$ satisfying 
	\begin{equation}\label{conphi}
		\begin{aligned}
			&\min_{i\neq j}|a_i-a_j|\geq 10\varepsilon, \ \ \ \ \ \|\phi'\|_{L^\infty(\mathbb{R}\backslash\mathbb{I}_\varepsilon)}\leq C_\phi,\\
			& \phi(x)= \begin{cases}
				c_n^+,\ \ x\in [a_n,a_n+\varepsilon],\\
				c_n^-,\ \ x\in [a_n-\varepsilon,a_n),
			\end{cases}	\ \ n=1,\cdots,N,
		\end{aligned}
	\end{equation}
	for some $\{c_n^\pm\}_{n=1}^N\subset\mathbb{R}^+$. Here we denote  $\mathbb{I}_\varepsilon=\cup_{n=1}^N[a_n-\varepsilon,a_n+\varepsilon]$.
	\begin{lemma}\label{lemma}Let $f $ be a solution to \eqref{eqpara} with initial data $f_0=\partial_x\bar f_0$. There exists  $T^*=T^*(c_0,C_\phi,\varepsilon)>0$ such that for any $0<T<T^*$,
		\begin{align*}
			\sum_{\star\in\{L^2_{T},X_T^{{1}/{2},2}\}}\|f\|_{\star}+\sum_{\star\in\{L^\frac{6}{5}_{T},X_T^{\frac{5}{6},\frac{6}{5}}\}}	\|\partial_x f\|_{\star}\lesssim \|\bar f_0\|_{\dot W^{\frac{1}{3},\frac{6}{5}}}+ \sum_{\star\in\{L^\frac{6}{5}_{T},X_T^{\frac{5}{6},\frac{6}{5}}\}}\| F\|_{\star}+T^{\frac{1}{4}}\sum_{\star\in\{{L^1_{T}},{X_T^{1,1}}\}}\|R\|_{\star}.
		\end{align*}
		Moreover, if $R=0$, we have 
		\begin{align*}
			&\quad\quad\quad\quad\quad\|\partial_x f\|_{X_T^{1-\gamma,\infty}}\lesssim \|\bar f_0\|_{\dot C^{2\gamma}}+\|F\|_{X_T^{1-\gamma,\infty}},\\
			&\|\partial_xf\|_{L^2_{T}}+\|\partial_xf\|_{X_T^{\frac{1}{2},2}}+\|\partial_x f\|_{X_T^{\frac{3}{4},\infty}}\lesssim \|f_0\|_{L^2}+\|F\|_{L^2_{T}}+\|F\|_{X_T^{\frac{1}{2},2}}+\|F\|_{X_T^{\frac{3}{4},\infty}}.
		\end{align*}
	\end{lemma}
	\begin{proof}
		Applying  Lemma \ref{lemXTjum}-Lemma \ref{lemLPint} to obtain that there exists $T_0=T_0(C_\phi,\varepsilon)>0$ such that	for any $0<T\leq T_0$, 
		\begin{align*}
			&\sum_{\star\in\{L^2_{T},X_T^{\frac{1}{2},2}\}}\|f\|_{\star}+\sum_{\star\in\{L^\frac{6}{5}_{T},X_T^{\frac{5}{6},\frac{6}{5}}\}}		\|\partial_x f\|_{\star}\\
			&\ \leq \mathbf{C} \|\bar f_0\|_{\dot W^{\frac{1}{3},\frac{6}{5}}}+ \mathbf{C}\sum_{\star\in\{L^\frac{6}{5}_{T},X_T^{\frac{5}{6},\frac{6}{5}}\}}	\| F\|_{\star}+\mathbf{C}T^{\frac{1}{4}}\sum_{\star\in\{{L^1_{T}},{X_T^{1,1}}\}}\|R\|_{\star}
			+\frac{1}{5}(\|f\|_{L^2_{T}}+\sum_{\star\in\{L^\frac{6}{5}_{T},X_T^{\frac{5}{6},\frac{6}{5}}\}}	\|\partial_x f\|_{\star}),
		\end{align*}
		and, if $R=0$,
		\begin{align*}
			&\| \partial_x f\|_{X_T^{1-\gamma,\infty}}\leq  \mathbf{C}(\|\bar f_0\|_{\dot C^{2\gamma}}+\|F\|_{X_T^{1-\gamma,\infty}}+T^\frac{1}{3}\|\partial_x f\|_{X_T^{1-\gamma,\infty}}),\\
			&\|f\|_{L^6_{T}}+\sum_{\star\in\{L^2_{T},X_T^{\frac{1}{2},2},X_T^{\frac{3}{4},\infty}\}}\|\partial_xf\|_{\star}\leq \mathbf{C}\|f_0\|_{L^2}+\mathbf{C}\sum_{\star\in\{L^2_{T},X_T^{\frac{1}{2},2},X_T^{\frac{3}{4},\infty}\}}\|F\|_{\star}\\
			&\quad\quad\quad\quad\quad\quad\quad\quad\quad\quad\quad\quad\quad\quad\quad\quad\quad\quad\quad+\frac{1}{10}(\|f\|_{L^6_{T}}+\sum_{\star\in\{L^2_{T},X_T^{\frac{1}{2},2},X_T^{\frac{3}{4},\infty}\}}\|\partial_x f\|_{\star}).
		\end{align*}
		where a constant $\mathbf{C}>0$ depends only on $c_0$, $C_\phi$ and $\varepsilon$. 
		This  implies the results by taking $0<T^*=\min \{\frac{1}{(\mathbf{C}+1)^{10}},T_0\}$.
		Then we complete the proof.
	\end{proof}
	
	We first introduce some elementary lemmas that will be used frequently in our proof. 
	\begin{lemma}\label{lemXTl}
		Let $g(t,x)=\int_{\mathbb{R}} \partial_x\mathbf{K}(t,x-y) \bar g_0(y) dy$. Then for any $T>0$,
		\begin{equation*}
			\begin{aligned}
				&\|\partial_x g\|_{X_T^{1-\gamma,\infty}}\lesssim \|\bar g_0\|_{\dot C^{2\gamma}}, \\
				&\|\partial_x g\|_{X_T^{\frac{1}{2},2}}+\|\partial_x g\|_{X_T^{\frac{3}{4},\infty}}\lesssim \|\partial_x \bar g_0\|_{L^2},\\
				&\|g\|_{X_T^{\frac{1}{2},2}}\lesssim \|\bar g_0\|_{L^2},\ \ \ \ \|\partial_x g\|_{X_T^{\frac{5}{6},\frac{6}{5}}}\lesssim \|\bar g_0\|_{\dot W^{\frac{1}{3},\frac{6}{5}}}.
			\end{aligned} 
		\end{equation*}
	\end{lemma}
	\begin{proof}
		Observe that $\int \partial_x^2\mathbf{K}(t,x)dx=0$. Hence 
		\begin{align*}
			\partial_x	g(t,x)=\int_{\mathbb{R}} \partial_x^2\mathbf{K}(t,x-y) (\bar g_0(y)-\bar g_0(x)) dy=-\int_{\mathbb{R}} \partial_x^2\mathbf{K}(t,z) \Delta_z\bar g_0(x) dy,
		\end{align*}
		where we denote $\Delta_z\bar g_0(x)=\bar g_0(x)-\bar g_0(x-z)$.
		Note that $|\Delta_z\bar g_0(x)|\leq |z|^{2\gamma}\|\bar g_0\|_{\dot C^{2\gamma}}$. Hence, by Lemma \ref{lemheat}, we have 
		\begin{align*}
			&\|	\partial_x	g(t)\|_{L^\infty}\lesssim \int_{\mathbb{R}}|\partial_x^2\mathbf{K}(t,x)||x|^{2\gamma}dx\|\bar g_0\|_{\dot C^{2\gamma}}\lesssim t^{\gamma-1}\|\bar g_0\|_{\dot C^{2\gamma}}.
		\end{align*}
		Moreover, for any $0<s<t<T$, 
		\begin{align*}
			\|\partial_x	g(t)-\partial_x	g(s)\|_{L^\infty}&\lesssim \int_{\mathbb{R}} |\partial_{x}^2\mathbf{K}(t,x)-\partial_{x}^2\mathbf{K}(s,x)||x|^{2\gamma}dx \|\bar g_0\|_{\dot C^{2\gamma}}\\
			&\lesssim s^{\gamma-1}\min\{1,\frac{t-s}{s}\}\|\bar g_0\|_{\dot C^{2\gamma}}\lesssim(t-s)^\alpha s^{\gamma-1-\alpha}\|\bar g_0\|_{\dot C^{2\gamma}}.
		\end{align*}
		Hence 
		\begin{equation}\label{1111}
			\|\partial_x g\|_{X_T^{1-\gamma,\infty}}\lesssim \|\bar g_0\|_{\dot C^{2\gamma}}.
		\end{equation}
		By Lemma \ref{lemheat} and Young's inequality, we have 
		\begin{equation}\label{ini}
			\begin{aligned}
				&t^\frac{1}{2}\|\partial_x g(t)\|_{L^2}+t^\frac{3}{4}\|\partial_x g\|_{L^\infty}\lesssim (t^\frac{1}{2}\|\partial_x \mathbf{K}(t)\|_{L^1}+t^\frac{3}{4}\|\partial_x \mathbf{K}(t)\|_{L^2})\|\partial_x \bar g_0\|_{L^2}\lesssim \|\partial_x \bar g_0\|_{L^2},\\
				&t^\frac{1}{2}\| g(t)\|_{L^2}\lesssim t^\frac{1}{2}\|\partial_x \mathbf{K}(t)\|_{L^1}\|\bar g_0\|_{L^2}\lesssim \|\bar g_0\|_{L^2},
			\end{aligned}	
		\end{equation}and 
		\begin{equation}\label{ini2}
			\begin{aligned}
				&s^{\frac{1}{2}+\alpha}\|\partial_x g(t)-\partial_x g(s)\|_{L^2}+s^{\frac{3}{4}+\alpha}\|\partial_x g(t)-\partial_x g(s)\|_{L^\infty}\\
				&\quad\quad\lesssim (s^{\frac{1}{2}+\alpha}\|\partial_x K(t)-\partial_x K(s)\|_{L^1}+s^{\frac{3}{4}+\alpha}\|\partial_x K(t)-\partial_x K(s)\|_{L^2})\|\partial_x \bar g_0\|_{L^2}\\
				&\quad\quad\lesssim (t-s)^\alpha\|\partial_x \bar g_0\|_{L^2}.
			\end{aligned}
		\end{equation}
		Similarly, it is easy to check that 
		\begin{equation}\label{ini3}
			s^{\frac{1}{2}+\alpha}\|g(t)-g(s)\|_{L^2}\lesssim (t-s)^\alpha\|\bar g_0\|_{L^2}.
		\end{equation}
		By \eqref{1111}-\eqref{ini3}, one has 
		\begin{align}\label{re1}
			\|\partial_x g\|_{X_T^{1-\gamma,\infty}}\lesssim \|\bar g_0\|_{\dot C^{2\gamma}},\ \ \ \ \  \|\partial_x g\|_{X_T^{\frac{1}{2},2}}+\|\partial_x g\|_{X_T^{\frac{3}{4},\infty}}\lesssim \|\partial_x \bar g_0\|_{L^2},\ \ \ \  \|g\|_{X_T^{\frac{1}{2},2}}\lesssim \|\bar g_0\|_{L^2}.
		\end{align}
		It remains to estimate $\|\partial_x g \|_{X_T^{\frac{5}{6},\frac{6}{5}}}$. By H\"{o}lder's inequality and  Lemma \ref{lemheat},
		\begin{align*}
			\|\partial_x g(t)\|_{L^\frac{6}{5}}&\lesssim \int_{\mathbb{R}}|\partial_x^2\mathbf{K}(t,z)|\|\Delta_z\bar f_0\|_{L^\frac{6}{5}}dz\lesssim \left(\int_{\mathbb{R}} |\partial_x^2\mathbf{K}(t,z)|^6|z|^7dz\right)^\frac{1}{6}\left(\int_{\mathbb{R}} \frac{\|\Delta_z\bar f_0\|_{L^{6/5}}^{6/5}}{|z|^{7/5}}dz\right)^{5/6}\\
			&\lesssim t^{-\frac{5}{6}}\|\bar g_0\|_{\dot W^{\frac{1}{3},\frac{6}{5}}}.
		\end{align*}
		Moreover, for any $0<s<t<T$,
		\begin{align*}
			\|\partial_x g(t)-\partial_x g(s)\|_{L^\frac{6}{5}}&\lesssim \int_{\mathbb{R}}|\partial_x^2\mathbf{K}(t,z)-\partial_x^2\mathbf{K}(s,z)|\|\Delta_z\bar f_0\|_{L^\frac{6}{5}}dz\\
			&\lesssim \left(\int_{\mathbb{R}} |\partial_x^2\mathbf{K}(t,z)-\partial_x^2\mathbf{K}(s,z)|^6|z|^7dz\right)^\frac{1}{6}\left(\int_{\mathbb{R}} \frac{\|\Delta_z\bar f_0\|_{L^{6/5}}^{6/5}}{|z|^{7/5}}dz\right)^{5/6}\\
			&\lesssim s^{-\frac{5}{6}-\alpha}(t-s)^\alpha\|\bar g_0\|_{\dot W^{\frac{1}{3},\frac{6}{5}}}.
		\end{align*}
		Then we get 
		$$
		\|\partial_x g\|_{X_T^{\frac{5}{6},\frac{6}{5}}}\lesssim \|\bar g_0\|_{\dot W^{\frac{1}{3},\frac{6}{5}}}.
		$$
		Combining this with \eqref{re1}, we complete the proof.
	\end{proof}
	\begin{remark}\label{remlem3.2}
		The result in Lemma \ref{lemXTl} is still true for $g(t,x)=\int_{\mathbb{R}}\partial_y\mathbf{G}(t,x-y,x)\bar{g}_0(y)dy$, where $\mathbf{G}(t,x,z)$ satisfies that $\int \partial_x^2 \mathbf{G}(t,x,z)dx=0$, and 
		\begin{align*}
			&\max_{l=0,1}\sup_z|\partial_x^m\partial_z^l\mathbf{G}(t,z,x)|\lesssim |\partial_x^m\mathbf{K}(\zeta t,x)|,\\
			&\max_{l=0,1}\sup_z|\partial_x^m\partial_z^l\mathbf{G}(t,z,x)-\partial_x^m\partial_z^l\mathbf{G}(s,z,x)|\lesssim |\partial_x^m\mathbf{K}(\zeta s,x)-\partial_x^m\mathbf{K}(\zeta t,x)|,\ \ \ \ m=1,2,
		\end{align*}
		for a universal constant $\zeta>0$.
	\end{remark}
	The following is the Hardy's inequality.
	\begin{lemma}\label{lemhardy}
		For any non-negative measurable function $h:\mathbb{R}\to \mathbb{R}^+$, there holds for any $p\in(1,\infty)$,
		\begin{align*}
			\int_0^\infty\left(\int_0^\infty h(x)\frac{dx}{x+y}\right)^pdy\lesssim \int_0^\infty h(x)^pdx.
		\end{align*}
	\end{lemma}
	\begin{proof}
		By Hardy's inequality \cite{Hardy}, 
		\begin{align*}
			\int_0^\infty\left(\int_0^\infty h(x)\frac{dx}{x+y}\right)^pdy\leq& \int_0^\infty\left(\int_0^y h(x){dx}\right)^p\frac{dy}{y^p}+\int_0^\infty\left(\int_y^\infty h(x)\frac{dx}{x}\right)^p{dy}\\
			\leq &\left(\frac{p^p}{(p-1)^p}+p^p\right)\int_0^\infty h(x)^p dx.
		\end{align*}
		This completes the proof.
	\end{proof}\vspace{0.5cm}\\
	The following lemma helps us to control boundary terms in \eqref{formubd3}.
	\begin{lemma}\label{lemupg}
		Let $r\in[1,+\infty), p,q\in(1,+\infty)$ be such that $\frac{1}{p}+1=\frac{1}{q}+\frac{1}{r}$. Suppose that  $\tilde G(t,x)$ satisfies
		\begin{align*}
			\int_0^\infty |\tilde G(t,x)|^r dt \lesssim \frac{1}{|x|}.
		\end{align*}
		Then for 	$$
		g(t,x)=\mathbf{1}_{x\geq 0}\int_0^t \int_0^{+\infty} \tilde G(t-\tau,x+y) f(\tau,y)dyd\tau,
		$$ 
		there holds 
		$$
		\|g\|_{L^p_{T}}\lesssim \|f\|_{L^q_{T}}.
		$$
	\end{lemma}
	\begin{proof}
		By Young's inequality,
		\begin{align*}
			\|g\|_{L^p_{T}}&\lesssim \left\|\int_0^\infty\left(	\int_0^t |\tilde G(t,x+y)|^r dt\right)^\frac{1}{r}\|f(y)\|_{L^q_t}dy\right\|_{L^p_x(\mathbb{R}^+)}\\
			&\lesssim \left\|\int_0^\infty \|f(y)\|_{L^q_t}\frac{dy}{|x+y|^{\frac{1}{r}}}\right\|_{L^p_x(\mathbb{R}^+)}.
		\end{align*}
		When $r>1$, then by Young's inequality, we obtain that 
		\begin{align*}
			\left\|\int_0^\infty \|f(y)\|_{L^q_{t,T}}\frac{dy}{|x+y|^{\frac{1}{r}}}\right\|_{L^p_x(\mathbb{R}^+)}\lesssim \|f\|_{L^q_{T}} \||x |^{-\frac{1}{r}}\|_{L^r_w}\lesssim \|f\|_{L^q_{T}} .
		\end{align*}
		When $r=1$, we obtain the result by Lemma \ref{lemhardy}.  
		This completes the proof.
	\end{proof}
	\subsection{Estimates near jump points}
	We first study estimates near jump points. Let $\mathbb{I}_{\varepsilon}^n=[a_n-\varepsilon,a_n+\varepsilon]$.  Without loss of generality, we assume $a_n=0$. Then we can write equation \eqref{eqpara} as 
	\begin{equation}\label{jueq}
		\begin{aligned}
			&\partial_t f -\partial_x(\bar  \phi f_x)=\partial_x \tilde F+R ,\\
			&f(0,x)=f_0(x),
		\end{aligned}
	\end{equation}
	where we denote $\tilde F=F+(\phi-\bar \phi)f_x$ with $\bar \phi=\mathbf{1}_{x\leq 0}c_n^-+\mathbf{1}_{x>0}c_n^+$. For simplicity, we omit  the subscript $n$ in our proof and denote $c^\pm=c^\pm_n$.\\
	By Lemma \ref{lemformula}, the solution of \eqref{jueq} has formula 
	$f=f^+\mathbf{1}_{x\geq 0}+f^-\mathbf{1}_{x< 0}$ , where  
	\begin{align}\label{soforjum}
		f^\pm (t,x)=&f_L^\pm+f_N^\pm+f_R^\pm.
	\end{align}
	Here $f_L^\pm=f_1^\pm+f_{1,B}^\pm$, $f_N^\pm=f_2^\pm+f_{2,B}^\pm$, $f_R^\pm=f_3^\pm+f_{3,B}^\pm$, with 
	\begin{align*}
		&f_1^\pm(t,x)=\int_{\mathbb{R}^\pm}( \mathbf{K}(c_\pm t,x-y)-\mathbf{K}(c_\pm t,x+y))f_0(y) dy,\\
		&f_2^\pm(t,x)=\int_0^t \int_{\mathbb{R}^\pm}( \mathbf{K}(c_\pm (t-\tau),x-y)-\mathbf{K}(c_\pm (t-\tau),x+y)) \partial_x\tilde F(\tau,y) dyd\tau,\\
		&f_3^\pm(t,x)=\int_0^t \int_{\mathbb{R}^\pm}( \mathbf{K}(c_\pm (t-\tau),x-y)-\mathbf{K}(c_\pm (t-\tau),x+y))  R(\tau,y) dyd\tau,\\
		& f_{m,B}^\pm=\frac{-2\sqrt{c_\pm}}{\sqrt{c_+}+\sqrt{c_-}}\int_0^t\mathbf{K}(c_\pm(t-\tau),x)(c_+\partial_xf_m^+(\tau,0)-c_-\partial_xf_m^-(\tau,0))d\tau,\ \ m=1,2,3.
	\end{align*}
	Here $f_{m,B}$ means these are boundary terms. Denote 
	$\beta_m(\tau,y)=c_+f_m^+(\tau,y)+c_-f_m^-(\tau,-y)$, then 
	\begin{align*}
		f_{m,B}^+&=C_+ \int_0^t\int_0^\infty\partial_y\left(\mathbf{K}(c_+(t-\tau),x+ y)\partial_x\beta_m(\tau,y)\right)dyd\tau\\
		&=-C_+\int_0^t\int_0^\infty\partial_x^2\mathbf{K}(c_+(t-\tau),x+ y)\beta_m(\tau,y)dyd\tau\\
		&\quad\quad+C_+\int_0^t\int_0^\infty \mathbf{K}(c_+(t-\tau),x+ y)\partial_x^2\beta_m(\tau,y)dyd\tau,
	\end{align*}
	where we denote $C_\pm=\frac{-2\sqrt{c_\pm}}{\sqrt{c_+}+\sqrt{c_-}}$.
	Note that 
	\begin{align*}
		&\partial_x ^2\beta_1(\tau,y)=\partial_t f_1^+(\tau,y)+\partial_t f_1^-(\tau,-y),\\
		&\partial_x ^2\beta_2(\tau,y)=\partial_t f_2^+(\tau,y)-\partial_x \tilde F(\tau,y)+\partial_t f_2^-(\tau,-y)-\partial_x \tilde F(\tau,-y),\\
		&\partial_x ^2\beta_3(\tau,y)=\partial_t f_3^+(\tau,y)-R (\tau,y)+\partial_t f_3^-(\tau,-y)-R (\tau,-y).
	\end{align*}
	Hence,
	\begin{equation}\label{formubd3}
		\begin{aligned}
			f_{1,B}^+(t,x)=&C_+(c_+-c_-)\int_0^t \int_0^\infty\partial_x^2\mathbf{K}(c_+(t-\tau),x+ y)f_1^-(\tau,-y)dyd\tau\\
			&-C_+\int_0^\infty  \mathbf{K}(c_+t,x+y)(f_0(y)+f_0(-y))dy,\\
			f_{2,B}^+(t,x)=&C_+(c_+-c_-)\int_0^t \int_0^\infty\partial_x^2\mathbf{K}(c_+(t-\tau),x+ y)f_2^-(\tau,-y)dyd\tau\\
			&-C_+\int_0^t \int_0^\infty  \mathbf{K}(c_+(t-\tau),x+y)((\partial_x\tilde F)(\tau,y)+(\partial_x\tilde F)(\tau,-y))dyd\tau,\\
			f_{3,B}^+(t,x)=&C_+(c_+-c_-)\int_0^t \int_0^\infty\partial_x^2\mathbf{K}(c_+(t-\tau),x+ y)f_3^-(\tau,-y)dyd\tau\\
			&-C_+\int_0^t \int_0^\infty  \mathbf{K}(c_+(t-\tau),x+y)(R(\tau,y)+R(\tau,-y)))dy.
		\end{aligned}
	\end{equation}
	\begin{remark}
		Note that we can recover the standard heat kernel from our formula when $c_+=c_-=c$. More precisely, we have 
		\begin{align*}
			&f_L(t,x)=\int_{\mathbb{R}} \mathbf{K}(ct,x-y) f_0(y) dy, \ \ f_N(t,x)=\int_0^t \int_{\mathbb{R}} \mathbf{K}(c(t-\tau),x-y) \partial_x F(\tau,y) dy,\\
			&f_R(t,x)=\int_0^t \int_{\mathbb{R}} \mathbf{K}(c(t-\tau),x-y) R(\tau,y) dy.
		\end{align*}
	\end{remark}
	The following lemma will help us to analyze the regularity of the solution near the jump points.
	\begin{lemma}\label{lemXTjum}
		Let $f$ be a solution to \eqref{eqpara} with initial data $f_0=\partial_x\bar f_0$. 	Then there exists  $T>0$ such that 
		\begin{equation*}
			\|f\|_{X_T^{\frac{1}{2},2}(\mathbb{I}_{\varepsilon})}+\|\partial_x f\|_{X_T^{\frac{5}{6},\frac{6}{5}}(\mathbb{I}_{\varepsilon})}\lesssim \|\bar f_0\|_{\dot W^{\frac{1}{3},\frac{6}{5}}}+\|F\|_{X_T^{\frac{5}{6},\frac{6}{5}}}+T^\frac{1}{3}\|f_x\|_{X_T^{\frac{5}{6},\frac{6}{5}}}+T^\frac{1}{4}(\|R\|_{L^1_{T}}+\|R\|_{X_T^{1,1}}).
		\end{equation*}
		Moreover, if $R=0$,
		\begin{align*}
			&\|\partial_x f\|_{X_T^{1-\gamma,\infty}(\mathbb{I}_{\varepsilon})}\lesssim  \|\bar f_0\|_{\dot C^{2\gamma}}+\|F\|_{X_T^{1-\gamma,\infty}}+ T^{\frac{1}{3}}\|\partial_x f\|_{X_T^{1-\gamma,\infty}},\\
			&\|\partial_x f\|_{X_T^{\frac{1}{2},2}(\mathbb{I}_{\varepsilon})}+\|\partial_x f\|_{X_T^{\frac{3}{4},\infty}(\mathbb{I}_{\varepsilon})}\lesssim  \|\partial_x\bar f_0\|_{L^2}+\|F\|_{X_T^{\frac{3}{4},\infty}}+ T^{\frac{1}{3}}(\|\partial_x f\|_{X_T^{\frac{1}{2},2}}+\|\partial_x f\|_{X_T^{\frac{3}{4},\infty}}).
		\end{align*}
	\end{lemma}
	\begin{proof}Recall that $\mathbb{I}_\varepsilon=\cup_{n=1}^N\mathbb{I}_\varepsilon^n$. It suffices to consider one interval $\mathbb{I}^n_\varepsilon=[a_n-\varepsilon,a_n+\varepsilon]$.  Without loss of generality, we assume $a_n=0$.  The solution has formula \eqref{soforjum}. 
		Using 	integration by parts, one has 
		\begin{align*}
			f_1^+(t,x)&=\int_{\mathbb{R}^+}( \partial_x \mathbf{K}(c_+ t,x-y)+\partial_x \mathbf{K}(c_+ t,x+y))\bar f_0(y) dy\\
			&=\int_{\mathbb{R}}\partial_x \mathbf{K}(c_+ t,x-y)\mathbf{f}_0(y) dy,
		\end{align*}
		where $\mathbf{f}_0(y)$ is the even extension of $\bar f_0\mathbf{1}_{x\geq 0}$
		$$
		\mathbf{f}_0(y)=\bar f_0(y)\mathbf{1}_{y\geq 0}+\bar f_0(-y)\mathbf{1}_{y<0}.
		$$
		It is easy to check that 
		$$
		\|\mathbf{f}_0\|_{\dot C^{2\gamma}}\leq \|\bar f_0\|_{\dot C^{2\gamma}}.
		$$
		Applying Lemma \ref{lemXTl} and Lemma \ref{lemsob} to get
		\begin{equation}\label{f1XT}
			\begin{aligned}
				& \|\partial_x f_1^+\|_{X_T^{1-\gamma,\infty}(\mathbb{R}^+)}\lesssim \|\mathbf{f}_0\|_{\dot C^{2\gamma}}\lesssim \|\bar f_0\|_{\dot C^{2\gamma}},\\
				&\|\partial_x f_1^+\|_{X_T^{\frac{1}{2},2}(\mathbb{R}^+)}+\|\partial_x f_1^+\|_{X_T^{\frac{3}{4},\infty}(\mathbb{R}^+)}\lesssim  \|\partial_x \bar f_0\|_{L^2},\ \ \ \ \ \ \ \|f_1^+\|_{X_T^{\frac{1}{2},2}(\mathbb{R}^+)}\lesssim  \|\bar f_0\|_{L^2},\\
				&\|\partial_x f_1^+\|_{X_T^{\frac{5}{6},\frac{6}{5}}(\mathbb{R}^+)}\lesssim \left(\int_{\mathbb{R}}\int_0^\infty|\mathbf{f}_0(x)-\mathbf{f}_0(x-z)|^\frac{6}{5} \frac{dxdz}{|z|^{7/5}}\right)^{5/6}\lesssim \|\bar f_0\|_{\dot W^{\frac{1}{3},\frac{5}{6}}}.
			\end{aligned}
		\end{equation}
		Recall the formula of boundary term $ f_{1,B}^+$ in \eqref{formubd3}, using  integration by parts one has
		\begin{equation*}
			\begin{aligned}
				f_{1,B}^+=&C_+(c_+-c_-)\int_0^t \int_0^\infty\partial_x^2\mathbf{K}(c_+(t-\tau),x+ y)f_1^-(\tau,-y)dyd\tau\\
				&-C_+\int_0^\infty  \partial_x\mathbf{K}(c_+t,x+y)(\bar f_0(y)-\bar f_0(-y))dy,\\
				:=&I_{1}+I_{2}.
			\end{aligned}
		\end{equation*}
		Applying Lemma \ref{mainlem} to obtain that for $(\sigma,p)\in\{(1-\gamma,\infty),(\frac{1}{2},2),(\frac{3}{4},\infty),(\frac{5}{6},\frac{6}{5})\}$,
		$$\|\partial_x^lI_{1}\|_{X_T^{\sigma,p}}\lesssim  \|\partial_x^l f_1\|_{X_T^{\sigma,p}},\ \ l=1,2.$$
		Moreover, observe that 
		\begin{align*}
			I_{2}=-C_+\int_{\mathbb{R}} \partial_x \mathbf{K}(c_+t,x-y)\bar{\mathbf{f}}_0(y)dy,
		\end{align*}
		where $\bar{\mathbf{f}}_0(y)=(\bar f_0(-y)-\bar f_0(y))\mathbf{1}_{y\leq 0}$. 
		Following the proof of  Lemma \ref{lemXTl} to yield
		\begin{equation*}
			\begin{aligned}
				&\|\partial_xI_{2}\|_{X_T^{1-\gamma,\infty}}\lesssim \|\bar{\mathbf{f}}_0\|_{\dot C^{2\gamma}}\lesssim \|\bar f_0\|_{\dot C^{2\gamma}},\\
				&\|\partial_xI_{2}\|_{X_T^{\frac{1}{2},2}}+\|\partial_xI_{2}\|_{X_T^{\frac{3}{4},\infty}}\lesssim \|\partial_x (\bar f_0(-y)-\bar f_0(y))\mathbf{1}_{y\leq 0}\|_{L^2}\lesssim \|\partial_x \bar f_0\|_{L^2},\\
				&\|I_{2}\|_{X_T^{\frac{1}{2},2}}\lesssim \|\bar f_0\|_{L^2}\lesssim \|\bar f_0\|_{\dot W^{\frac{1}{3},\frac{5}{6}}},
			\end{aligned}
		\end{equation*}
		and 
		\begin{align*}
			&\|\partial_xI_{2}\|_{X_T^{\frac{5}{6},\frac{6}{5}}(\mathbb{R}^+)}\lesssim \left(\int_{\mathbb{R}}\int_{\mathbb{R}^+}|\bar{\mathbf{f}}_0(x)-\bar{\mathbf{f}}_0(x-z)|^\frac{6}{5} \frac{dxdz}{|z|^{7/5}}\right)^{5/6}\lesssim \|\bar f_0\|_{\dot W^{\frac{1}{3},\frac{5}{6}}},
		\end{align*}
		where we applied Lemma \ref{lemsob} in the last inequality. 
		Similar arguments hold for $\partial_x f_{1,B}^-$. Hence 
		\begin{align*}
			&	\|\partial_x f_{1,B}\|_{X_T^{1-\gamma,\infty}}\lesssim \|\partial_x f_1\|_{X_T^{1-\gamma,\infty}}+\|\bar f_0\|_{\dot C^{2\gamma}},\\ &\|\partial_xf_{1,B}\|_{X_T^{\frac{1}{2},2}}+\|\partial_xf_{1,B}\|_{X_T^{\frac{3}{4},\infty}}\lesssim \|\partial_xf_{1}\|_{X_T^{\frac{1}{2},2}}+\|\partial_xf_{1}\|_{X_T^{\frac{3}{4},\infty}}+\|\partial_x \bar f_0\|_{L^2},\\
			&	\| f_{1,B}\|_{X_T^{\frac{1}{2},2}}+\|\partial_x f_{1,B}\|_{X_T^{\frac{5}{6},\frac{6}{5}}}\lesssim \|\bar f_0\|_{\dot W^{\frac{1}{3},\frac{6}{5}}}+	\| f_{1}\|_{X_T^{\frac{1}{2},2}}+\|\partial_x f_{1}\|_{X_T^{\frac{5}{6},\frac{6}{5}}}.
		\end{align*}
		Combining this with \eqref{f1XT}, we obtain that 
		\begin{equation}
			\begin{aligned}\label{fLXT}
				&\|\partial_x f_L\|_{X_T^{1-\gamma,\infty}}\lesssim \|\bar f_0\|_{\dot C^{2\gamma}},\ \ \ \ \ \ \ \|\partial_x f_L\|_{X_T^{\frac{1}{2},2}}+\|\partial_x f_L\|_{X_T^{\frac{3}{4},\infty}}\lesssim \|\partial_x \bar f_0\|_{L^2},\\
				&\| f_L\|_{X_T^{\frac{1}{2},2}}+\|\partial_x f_L\|_{X_T^{\frac{5}{6},\frac{6}{5}}}\lesssim \|\bar f_0\|_{\dot W^{\frac{1}{3},\frac{6}{5}}}.
			\end{aligned}
		\end{equation}
		Then we estimate nonlinar term $f_N=f_2+f_{2,B}$. We first estimate $f_2$. 
		By integration by parts, 
		%\begin{align*}
		%	\partial_x^lf_2^\pm(t,x)=\int_0^t \int(\partial_x^{1+l} K(b_\pm (t-\tau),x-y)+\partial_x^{1+l}K(b_\pm (t-\tau),x+y))\tilde F(\tau,y)\mathbf{1}_{y\in\mathbb{R}^\pm}dyd\tau.
		%\end{align*}
		\begin{equation}\label{f22pa}
			\begin{aligned}
				f_2^\pm(t,x)=&\int_0^t \int_{\mathbb{R}^\pm}(\partial_x \mathbf{K}(c_\pm (t-\tau),x-y)+\partial_x\mathbf{K}(c_\pm (t-\tau),x+y)) F(\tau,y) dyd\tau\\
				&+\int_0^t \int_{\mathbb{R}^\pm}(\partial_x \mathbf{K}(c_\pm (t-\tau),x-y)+\partial_x\mathbf{K}(c_\pm (t-\tau),x+y)) (\phi-\bar\phi)(y)f_x(\tau,y) dyd\tau\\
				=&f_{2,1}^\pm(t,x)+f_{2,2}^\pm(t,x).
			\end{aligned}
		\end{equation}
		By Lemma \ref{mainlem}, we obtain 
		\begin{align}\label{f21}
			\|\partial_x f_{2,1}\|_{X_T^{\sigma,p}}\lesssim \| F\|_{X_T^{\sigma,p}},\ \ \ \ \|\partial_x f_{2,2}\|_{X_T^{\sigma,p}}\lesssim \|\partial_x f\|_{X_T^{\sigma,p}}.
		\end{align}
		And Lemma \ref{lemlow} implies that 
		\begin{align}\label{ccc}
			\|f_{2,1}\|_{X_T^{\frac{1}{2},2}}\lesssim \|F\|_{X_T^{\frac{5}{6},\frac{6}{5}}}.
		\end{align}
		Observe that $\phi(y)-\bar\phi(y)=0$ for $y\in [-\varepsilon,\varepsilon]$. And 
		\begin{align}\label{phiaa}
			|\phi(y)-\bar\phi(y)|=|\phi(y)-\phi(x)|\lesssim C_\phi|x-y|,
		\end{align}
		{for} $x\in (0,\varepsilon), y\in(0,4\varepsilon)$ or $x\in (-\varepsilon,0), y\in(-4\varepsilon,0)$. This implies a better estimate for $ f_{2,2}$ when $x\in(-\varepsilon,\varepsilon)$. For simplicity, we only consider $f_{2,2}^+(t,x)$ with $x\in(0,\varepsilon)$, while $ f_{2,2}^-(t,x)$ with $x\in(-\varepsilon,0)$ can be done similarly.
		We can write 
		\begin{align*}
			\partial_xf_{2,2}^+(t,x)=\int_0^t \int_{\mathbb{R}^+} G(t-\tau,x,y)f_x(\tau,y)dyd\tau,
		\end{align*}
		where
		$$
		G(t,x,y)=(\partial_x^2 \mathbf{K}(c_+ t,x-y)+\partial_x^2\mathbf{K}(c_+ t,x+y)) (\phi-\bar\phi)(y)\mathbf{1}_{x,y\in\mathbb{R}^+}.
		$$ 
		For $x\in (0,\varepsilon)$, $0<s<t<T$,
		\begin{align*}
			&	\int_{\mathbb{R}} |G(t,x,y)|dy\lesssim C_\phi\int _{|y|\leq 4\varepsilon}|\partial_x^2\mathbf{K} (c_+t,y)||y|dy+ \int_{|y|\geq 4\varepsilon} |\partial_x^2\mathbf{K} (c_+t,y)|dy,\\
			&	\int_{\mathbb{R}} |G(t,x,y)-G(s,x,y)|dy\lesssim C_\phi\int _{|y|\leq 4\varepsilon}|\partial_x^2\mathbf{K} (c_+t,y)-\partial_x^2\mathbf{K} (c_+s,y)||y|dy\\
			&\quad\quad\quad\quad\quad\quad\quad\quad\quad\quad\quad\quad\quad\quad\quad\quad+ \int_{|y|\geq 4\varepsilon} |\partial_x^2\mathbf{K} (c_+t,y)-\partial_x^2\mathbf{K} (c_+(s,y))|dy.
		\end{align*}
		By Lemma \ref{lemheat} we obtain that 
		\begin{align*}
			&	\int _{|y|\leq 4\varepsilon}|\partial_x^2\mathbf{K} (c_+(t,y))||y|dy\lesssim t^{-\frac{1}{2}},\ \ \  \int_{|y|\geq 4\varepsilon} |\partial_x^2\mathbf{K} (c_+(t,y))|dy\lesssim \int_{|y|\geq 4\varepsilon}\frac{1}{(|y|+t^\frac{1}{2})^3}dy\lesssim \varepsilon^{-1}t^{-\frac{1}{2}},\\
			&\int _{|y|\leq 4\varepsilon}|\partial_x^2\mathbf{K} (c_+(t,y))-\partial_x^2\mathbf{K} (c_+(s,y))||y|dy\lesssim t^{-\frac{1}{2}}\min\{1,\frac{t-s}{s}\},\\
			&\int_{|y|\geq 4\varepsilon} |\partial_x^2\mathbf{K} (c_+(t,y))-\partial_x^2\mathbf{K} (c_+(s,y))|dy\lesssim \varepsilon^{-1}t^{-\frac{1}{2}}\min\{1,\frac{t-s}{s}\}.
		\end{align*} Hence 
		\begin{equation*}
			\begin{aligned}
				&	\sup_{x\in(0,\varepsilon)}\int_{\mathbb{R}} |G(t,x,y)|dy\lesssim(C_\phi+\varepsilon^{-1}) t^{-\frac{1}{2}},\\
				& \sup_{x\in(0,\varepsilon)}\int_{\mathbb{R}} |G(t,x,y)-G(s,x,y)|dy\lesssim(C_\phi+\varepsilon^{-1}) t^{-\frac{1}{2}}\min\{1,\frac{t-s}{s}\}.
			\end{aligned}
		\end{equation*}
		Then  Lemma \ref{lemlow} implies that 
		\begin{align*}
			&	\|\partial_xf_{2,2}\|_{X_T^{\sigma,p}([-\varepsilon,\varepsilon])}\lesssim  (C_\phi +\varepsilon^{-1}) T^\frac{1}{2}\|f_x\|_{X_T^{\sigma,p}},\\
			&	\|f_{2,2}\|_{X_T^{\frac{1}{2},2}([-\varepsilon,\varepsilon])}\lesssim (C_\phi +\varepsilon^{-1}) T^\frac{1}{2}\|f_x\|_{X_T^{\frac{5}{6},\frac{6}{5}}}.
		\end{align*}
		Combining this with \eqref{f21}, \eqref{ccc} to yield
		\begin{equation}
			\begin{aligned}\label{f2XT}
				&\|\partial_x f_2\|_{X_T^{\sigma,p}([-\varepsilon,\varepsilon])}\lesssim \| F\|_{X_T^{\sigma,p}}+ (C_\phi+\varepsilon^{-1}) T^\frac{1}{2}\|f_x\|_{X_T^{\sigma,p}},\\
				&\|f_2\|_{X_T^{\frac{1}{2},2}([-\varepsilon,\varepsilon])}\lesssim \|F\|_{X_T^{\frac{5}{6},\frac{6}{5}}}+ (C_\phi +\varepsilon^{-1}) T^\frac{1}{2}\|f_x\|_{X_T^{\frac{5}{6},\frac{6}{5}}}.
			\end{aligned}
		\end{equation}
		For boundary terms, we can write 
		\begin{equation}\label{forf2b}
			\begin{aligned}
				\partial_x^lf_{2,B}^+=&	C_+(c_--c_+)\int_0^t\int_0^\infty\partial_x^2\mathbf{K}(c_+(t-\tau),x+ y)(\partial_x^lf_2^-)(\tau,-y)\mathbf{1}_{[0,\varepsilon]}dyd\tau\\
				&+C_+(c_--c_+)\int_0^t\int_0^\infty\partial_x^2\mathbf{K}(c_+(t-\tau),x+ y)(\partial_x^lf_2^-)(\tau,-y)\mathbf{1}_{(\varepsilon,\infty)}dyd\tau\\
				&+C_+\int_0^t\int_0^\infty \partial_x^{1+l}\mathbf{K}(c_+(t-\tau),x+ y)(\tilde F(\tau,y)-\tilde F(\tau,-y))dyd\tau\\
				:=&II_{1,l}+II_{2,l}+II_{3,l}.
			\end{aligned}
		\end{equation}
		Applying Lemma \ref{mainlem} to get
		\begin{align}\label{II1}
			\|II_{1,l}\|_{X_T^{\sigma,p}}\lesssim \|\partial_x^l f_2\|_{X_T^{\sigma,p}([-\varepsilon,\varepsilon])}. 
		\end{align}
		One can check from Lemma \ref{lemheat} that 
		\begin{align}\label{tutu}
			\sup_{x\in\mathbb{R}^+}\int \partial_x^2 \mathbf{K}(c_+t,x+y) \mathbf{1}_{[\varepsilon,\infty)}dy \lesssim  \int _0^\infty \frac{1}{(t^\frac{1}{2}+y)^3}dy\lesssim  \varepsilon^{-1} t^{-\frac{1}{2}}.
		\end{align}
		Then Lemma \ref{lemlow} implies that 
		\begin{equation}
			\begin{aligned}\label{II2}
				&\|II_{2,1}\|_{X_T^{\sigma,p}}\lesssim \varepsilon^{-1} T^\frac{1}{2}\|\partial_x f_{2}\|_{X_T^{\sigma,p}}\overset{\eqref{f21}}\lesssim \varepsilon^{-1} T^\frac{1}{2}(\|F\|_{X_T^{\sigma,p}}+\|\partial_x f\|_{X_T^{\sigma,p}}),\\
				&\|II_{2,0}\|_{X_T^{\frac{1}{2},2}}\lesssim\varepsilon^{-1} T^\frac{1}{2}\| f_{2}\|_{X_T^{\frac{1}{2},2}}\overset{\eqref{f21}}\lesssim\varepsilon^{-1} T^\frac{1}{2}(\|F\|_{X_T^{\frac{5}{6},\frac{6}{5}}}+\|\partial_x f\|_{X_T^{\frac{5}{6},\frac{6}{5}}}).
			\end{aligned}
		\end{equation}
		Finally, following the estimate for $f_2$, we can prove that 
		\begin{equation}	\label{II3}
			\begin{aligned}
				&\|II_{3,0}\|_{X_T^{\frac{1}{2},2}([0,\varepsilon])}\lesssim \| F\|_{X_T^{\frac{5}{6},\frac{6}{5}}}+ (C_\phi+\varepsilon^{-1}) T^\frac{1}{2}\|f_x\|_{X_T^{\frac{5}{6},\frac{6}{5}}},\\
				&	\|II_{3,1}\|_{X_T^{\sigma,p}([0,\varepsilon])}\lesssim \| F\|_{X_T^{\sigma,p}}+ (C_\phi+\varepsilon^{-1}) T^\frac{1}{2}\|f_x\|_{X_T^{\sigma,p}}.
			\end{aligned}
		\end{equation}
		One can estimate $f_{2,B}^-$ similarly.
		By \eqref{II1}, \eqref{II2} and \eqref{II3}, we obtain that 
		\begin{align*}
			&	\|\partial_x f_{2,B}\|_{X_T^{\sigma,p}([-\varepsilon,\varepsilon])}\lesssim \|\partial_x f_2\|_{X_T^{\sigma,p}([-\varepsilon,\varepsilon])}+ (1+\varepsilon^{-1} T^\frac{1}{2})\| F\|_{X_T^{\sigma,p}}+ (C_\phi+\varepsilon^{-1}) T^\frac{1}{2}\|\partial_xf\|_{X_T^{\sigma,p}},\\
			&		\| f_{2,B}\|_{X_T^{\frac{1}{2},2}([-\varepsilon,\varepsilon])}\lesssim \|\partial_x f_2\|_{X_T^{\frac{5}{6},\frac{6}{5}}([-\varepsilon,\varepsilon])}+ (1+\varepsilon^{-1} T^\frac{1}{2})\| F\|_{X_T^{\frac{5}{6},\frac{6}{5}}}+ (C_\phi+\varepsilon^{-1}) T^\frac{1}{2}\|\partial_xf\|_{X_T^{\frac{5}{6},\frac{6}{5}}}.
		\end{align*}
		Combining this with \eqref{f2XT}, we obtain that 
		\begin{equation}
			\begin{aligned}\label{fNXT}
				&\|\partial_x f_N\|_{X_T^{\sigma,p}([-\varepsilon,\varepsilon])}\lesssim (1+\varepsilon^{-1} T^\frac{1}{2}) \| F\|_{X_T^{\sigma,p}}+ (C_\phi+\varepsilon^{-1}) T^\frac{1}{2}\|f_x\|_{X_T^{\sigma,p}},\\
				&\| f_N\|_{X_T^{\frac{1}{2},2}([-\varepsilon,\varepsilon])}\lesssim (1+\varepsilon^{-1} T^\frac{1}{2}) \| F\|_{X_T^{\frac{5}{6},\frac{6}{5}}}+ (C_\phi+\varepsilon^{-1}) T^\frac{1}{2}\|f_x\|_{X_T^{\frac{5}{6},\frac{6}{5}}}.
			\end{aligned}
		\end{equation}
		Finally, we estimate $f_R$. 
		For $f_3$, we can write 
		\begin{align*}
			f_3^\pm(t,x)=\int_0^t \int_{\mathbb{R}}  \mathbf{K}(c_\pm (t-\tau),x-y) R^\pm(\tau,y)dyd\tau,
		\end{align*}
		where $R^\pm$ are odd extensions of $R$:
		$$
		R^+(t,x)=R(t,x)\mathbf{1}_{x\geq 0}-R(-x)\mathbf{1}_{x< 0},\ \ \ \ R^-(t,x)=R(t,x)\mathbf{1}_{x< 0}-R(-x)\mathbf{1}_{x\geq 0}.
		$$
		From  Lemma \ref{rem}, we obtain 
		\begin{align*}
			&\|f_3\|_{X_T^{\frac{1}{2},2}}+\|\partial_x f_3\|_{X_T^{\frac{5}{6},\frac{6}{5}}}\lesssim T^\frac{1}{4}(\|R\|_{L^1_{T}}+\|R\|_{X_T^{1,1}}).
		\end{align*} 
		The proof for boundary terms $f^\pm_{3,B}$ are similar as previous arguments. 	We conclude that 	
		\begin{align}\label{fRXT}
			\| f_R\|_{X_T^{\frac{1}{2},2}}+\|\partial_x f_R\|_{X_T^{\frac{5}{6},\frac{6}{5}}}\lesssim T^\frac{1}{4} (\|R\|_{L^1_{T}}+\|R\|_{X_T^{1,1}}).
		\end{align}
		We conclude from \eqref{fLXT}, \eqref{fNXT} and \eqref{fRXT} that 
		\begin{align*}
			\| f\|_{X_T^{\frac{1}{2},2}([-\varepsilon,\varepsilon])}+\|\partial_x f\|_{X_T^{\frac{5}{6},\frac{6}{5}}([-\varepsilon,\varepsilon])}\lesssim &\|\bar f_0\|_{\dot W^{\frac{1}{3},\frac{6}{5}}}+(1+\varepsilon^{-1} T^\frac{1}{2}) \| F\|_{X_T^{\frac{5}{6},\frac{6}{5}}}+ (C_\phi+\varepsilon^{-1}) T^\frac{1}{2}\|f_x\|_{X_T^{\frac{5}{6},\frac{6}{5}}}\\
			&\ \ +T^\frac{1}{4} (\|R\|_{L^1_{T}}+\|R\|_{X_T^{1,1}}),
		\end{align*}
		and, when $R=0$,
		\begin{align*}
			&\|\partial_x f\|_{X_T^{1-\gamma,\infty}([-\varepsilon,\varepsilon])}\lesssim \|\bar f_0\|_{\dot C^{2\gamma}}+(1+\varepsilon^{-1} T^\frac{1}{2}) \|F\|_{X_T^{1-\gamma,\infty}}+( C_\phi+\varepsilon^{-1} ) T^{\frac{1}{2}}\|\partial_x f\|_{X_T^{1-\gamma,\infty}},\\
			&\|\partial_x f\|_{X_T^{\frac{1}{2},2}([-\varepsilon,\varepsilon])}+\|\partial_x f\|_{X_T^{\frac{3}{4},\infty}([-\varepsilon,\varepsilon])}\lesssim \|\partial_x \bar f_0\|_{L^2}+(1+\varepsilon^{-1} T^\frac{1}{2}) (\| F\|_{X_T^{\frac{1}{2},2}}+\| F\|_{X_T^{\frac{3}{4},\infty}})\\&\quad\quad\quad\quad\quad\quad\quad\quad\quad\quad\quad\quad\quad\quad\quad\quad\quad\quad\quad+ (C_\phi+\varepsilon^{-1}) T^\frac{1}{2}(\|f_x\|_{X_T^{\frac{1}{2},2}}+\|f_x\|_{X_T^{\frac{3}{4},\infty}}).
		\end{align*}
		Fix $T=\frac{1}{(1+C_\phi+\varepsilon^{-1})^{10}}$, we obtain the results. This completes the proof.
	\end{proof}\vspace{0.3cm}\\
	The following lemma gives estimates of  $L^p$ norms near jump points. 
	\begin{lemma}\label{lemLPjum}
		Suppose $f_0=\partial_x\bar f_0$.	There exists $T>0$ such that 
		\begin{align*}
			\|f\|_{L^2_{T}(\mathbb{I}_\varepsilon)}+	\|\partial_x f\|_{L^\frac{6}{5}_{T}(\mathbb{I}_\varepsilon)}\lesssim \|\bar f_0\|_{\dot W^{\frac{1}{3},\frac{6}{5}}}+ \| F\|_{L^\frac{6}{5}_{T}}+T^\frac{1}{3}\|\partial_x f\|_{L^\frac{6}{5}_{T}}+T^{\frac{1}{4}}\|R\|_{L^1_{T}}.
		\end{align*}
		Moreover, if $R=0$, there holds 
		\begin{align*}
			\|f\|_{L^6_{T}(\mathbb{I}_\varepsilon)}+	\|\partial_xf\|_{L^2_{T}(\mathbb{I}_\varepsilon)}\lesssim \|f_0\|_{L^2}+\|F\|_{L^2_{T}}+T^\frac{1}{3}\|\partial_x f\|_{L^2_{T}}.
		\end{align*}
	\end{lemma}
	\begin{proof} ~\\
		\textit{1.Estimate $f_L$.}\\
		For simplicity, we only consider $f_1^+\mathbf{1}_{x\geq 0}$, while $f_1^-\mathbf{1}_{x< 0}$ can be done similarly. Integrate by parts we obtain 
		\begin{align*}
			f_1^+(t,x)
			&=\int _{\mathbb{R}}\partial_x\mathbf{K}(b_+ t,x-y)\mathbf{f}_0(y)dy,\ \ \ \text{where}\ \mathbf{f}_0(y)=\bar f_0(y)\mathbf{1}_{y\geq 0}+\bar f_0(-y)\mathbf{1}_{y< 0}.
		\end{align*}
		By Parseval's identity (see \eqref{paraL2}), we have for any $T>0$,
		\begin{equation}\label{L21}
			\begin{aligned}
				&\|f_1^+\mathbf{1}_{x\geq 0}\|_{L^2_{T}}\sim \|\mathbf{f}_0\|_{L^2}\lesssim \|\bar f_0\|_{L^2}\lesssim \|\bar f_0\|_{\dot W^{\frac{1}{3},\frac{6}{5}}},\\ &\|\partial_x f_1^+\mathbf{1}_{x\geq 0}\|_{L^2_{T}}\sim \|\partial_x\mathbf{f}_0\|_{L^2}\lesssim \|f_0\|_{L^2}.
			\end{aligned}
		\end{equation}
		Then following \eqref{aplinearLp} in Lemma \ref{lemap}, we obtain from Lemma \ref{lemsob} that 
		\begin{equation}\label{L22}
			\begin{aligned}
				& \|	\mathbf{1}_{x\geq 0}\partial_x f_1^+\|_{L^{6/5}_{T}} ^{6/5}\lesssim \int_{\mathbb{R}} \frac{\| \mathbf{ f}_0(\cdot)-\mathbf{ f}_0(\cdot-z)\|_{L^{6/5}(\mathbb{R}^+)}^{6/5}}{|z|^{\frac{7}{5}}}dz\lesssim  \|\bar f_0\|_{\dot W^{\frac{1}{3},\frac{6}{5}}}^{6/5},\\
				&\|	\mathbf{1}_{x\geq 0} f_1^+\|_{L^6_{T}} ^6\lesssim \int_{\mathbb{R}} \frac{\| \mathbf{ f}_0(\cdot)- \mathbf{ f}_0(\cdot-z)\|_{L^6(\mathbb{R}^+)}^6}{|z|^{5}}dz\lesssim  \|\bar f_0\|_{\dot W^{\frac{2}{3},6}}^6\lesssim \|\bar f_0\|_{\dot W^{1,2}}^6=\|f_0\|_{L^2}^6.
			\end{aligned}
		\end{equation}
		By \eqref{L21} and \eqref{L22} we obtain that 
		\begin{align}\label{f1}
			\|f_1\|_{L^2_{T}}+\|\partial_x f_1\|_{L^\frac{6}{5}_{T}}\lesssim \|\bar f_0\|_{\dot W^{\frac{1}{3},\frac{6}{5}}},\ \ \ \ \ \  \|f_1\|_{L^6_{T}}+\|\partial_x f_1\|_{L^2_{T}}\lesssim \|f_0\|_{L^2}.
		\end{align}
		Then we estimate $f_{1,B}$. Recall the formula
		\begin{align*}
			f_{1,B}^+(t,x)=&C_+(c_+-c_-)\int_0^t \int_0^\infty\partial_x^2\mathbf{K}(c_+(t-\tau),x+ y)f_1^-(\tau,-y)dyd\tau\\
			&\quad\quad\quad-C_+\int \partial_x\mathbf{K}(c_+t,x-y)\bar{\mathbf {f}_0}(y)dy\\
			:=&I_{1}(t,x)+I_{2}(t,x).
		\end{align*}
		From Lemma \ref{lemheat} one has $\left(\int_0^\infty|\partial_x\mathbf{K}(c_+\tau,x+y)|^2d\tau\right)^\frac{1}{2}\lesssim \frac{1}{x+y}$.  Applying Lemma \ref{lemupg} yields that 
		\begin{align}\label{A1}
			\|\partial_x^lI_1\|_{L^p_{T}}\lesssim \|\partial_x^lf_1\|_{L^p_{T}},\ \ \ l=0,1,\ p>1.
		\end{align}
		Moreover, recall that $\bar{\mathbf{f}}_0(y)=(\bar f_0(-y)-\bar f_0(y))\mathbf{1}_{y\leq 0}$. We can write
		\begin{align}\nonumber
			\|I_2\|_{L^2_{T}}^2&\lesssim \int_0^\infty\left(\int_0^\infty \|\partial_x\mathbf{K}(c_+\tau,x+y)\|_{L^2_t}(\bar f_0(y)-\bar f_0(-y))dy\right)^2dx\\
			&\lesssim \int_0^\infty\left(\int_0^\infty(\bar f_0(y)-\bar f_0(-y))\frac{dy}{x+y}\right)^2dx\lesssim \|\bar f_0\|_{L^2}^2,\label{A2}
		\end{align}
		where we apply Lemma \ref{lemhardy} in the last inequality. Similarly, we obtain 
		\begin{align}\label{DA2}
			\|\partial_x I_2\|_{L^2_{T}}\lesssim \|f_0\|_{L^2}.
		\end{align}
		It remains to estimate $\|\partial_x I_2\|_{L^\frac{6}{5}_{T}}$ and $\|I_2\|_{L^6_{T}}$. Following \eqref{aplinearLp} in Lemma \ref{lemap}, one gets
		\begin{align*}
			\|\partial_x^lI_2\|_{L^p_{T}(\mathbf{\mathbb{R}^+})}\lesssim 
			\int_{\mathbb{R}}\int_0^\infty \frac{| \bar {\mathbf{f}}_0(x)-\bar {\mathbf{f}}_0(x-z)|^p}{|z|^{(1+l)p-1}}dxdz.
		\end{align*}
		Applying Lemma \ref{lemsob} to obtain that 
		\begin{align*}
			\|\partial_xI_2\|_{L^\frac{6}{5}_{T}}\lesssim \|\bar f_0\|_{\dot W^{\frac{1}{3},\frac{6}{5}}}, \ \ \ \ \|I_2\|_{L^6_{T}}\lesssim \|f_0\|_{L^2}.
		\end{align*}
		Similar argument holds for $f_{1,B}^-$. Combining this with \eqref{A1}, \eqref{A2} and \eqref{DA2}, we obtain that 
		\begin{align}\label{bdterm}
			\|f_{1,B}\|_{L^2_{T}}+\|\partial_x f_{1,B}\|_{L^\frac{6}{5}_{T}}\lesssim \|\bar f_0\|_{\dot W^{\frac{1}{3},\frac{6}{5}}},\ \ \ \ \  \|f_{1,B}\|_{L^6_{T}}+\|\partial_x f_{1,B}\|_{L^2_{T}}\lesssim \|f_0\|_{L^2}.
		\end{align}
		By \eqref{f1} and \eqref{bdterm} we obtain that 
		\begin{equation}\label{jumLpL}
			\begin{aligned}
				\|f_L\|_{L^2_{T}}+\|\partial_x f_L\|_{L^\frac{6}{5}_{T}}\lesssim \|\bar f_0\|_{\dot W^{\frac{1}{3},\frac{6}{5}}},\ \ \ \ \ \ \  \| f_L\|_{L^6_{T}}+\|\partial_x f_L\|_{L^2_{T}}\lesssim \|f_0\|_{L^2}.
			\end{aligned}
		\end{equation}
		\textit{2. Estimate $f_N$.}\\
		We first estimate $f_2^\pm=f_{2,1}^\pm+f_{2,2}^\pm$ as defined in \eqref{f22pa}. For simplicity, we only consider $\partial_x^lf_2^+$, and the other  part $\partial_x^lf_2^-$ can be done parallelly.
		Note that
		\begin{align*}
			f_{2,1}^+(t,x)&=\int_0^t \int_{\mathbb{R}^+}( \partial_x\mathbf{K}(c_+ (t-\tau),x-y)+\partial_x\mathbf{K}(c_+ (t-\tau),x+y))F(\tau,y) dyd\tau\\
			&=\int_0^t \int_{\mathbb{R}} \partial_x\mathbf{K}(c_+(t-\tau),x-y)( F(\tau,y)\mathbf{1}_{y\geq 0}+ F(\tau,-y)\mathbf{1}_{y< 0}) dyd\tau.
		\end{align*}
		%Similarly, we have 
		%\begin{align*}
		%	f_{2,1}^-(t,x)&=\int_0^t \int_{\mathbb{R}^-}( \partial_x\mathbf{K}(c_- (t-\tau),x-y)+\partial_x\mathbf{K}(c_- (t-\tau),x+y))F(\tau,y) dyd\tau\\
		%	&=\int_0^t \int \partial_x\mathbf{K}(c_-(t-\tau),x-y)( F(\tau,y)\mathbf{1}_{y\leq 0}+ F(\tau,-y)\mathbf{1}_{y> 0})dyd\tau.
		%\end{align*}
		Following the proof of Lemma \ref{lemap}, we obtain that  
		\begin{align}\label{f21lp}
			\|\partial_x^lf_{2,1}\|_{L^p_{T}}\lesssim \| F\|_{L^\frac{3p}{3+(1-l)p}_{T}}, \ \text{for}\ l=0,1,\ \ p>1.
		\end{align}
		Then we consider $f_{2,2}$. Recall that 
		\begin{align*}
			\partial_x^lf_{2,2}^+(t,x)=\int_0^t \int_0^\infty \partial_x^lG(t-\tau,x,y)f_x(\tau,y)dyd\tau,\ \ \ l=0,1,
		\end{align*}
		where
		$$
		G(t,x,y)=(\partial_x \mathbf{K}(c_\pm t,x-y)+\partial_x\mathbf{K}(c_\pm t,x+y)) (\phi-\bar\phi)(y).
		$$ 
		Then Lemma \ref{lemap} yields that 
		\begin{align}\label{rrr}
			\|\partial_x^lf_{2,1}\|_{L^p_{T}}\lesssim \| f_x\|_{L^\frac{3p}{3+(1-l)p}_{T}}, \ \text{for}\ l=0,1,\ \ p>1.
		\end{align}
		By \eqref{phiaa} and Lemma \ref{lemheat} we can check that 
		\begin{align*}
			\sup_{x\in[0,\varepsilon]}\left(\int_0^T\int_0^\infty|\partial_x^l G(t,x,y)|^\frac{3}{2+l}dydt\right)^\frac{2+l}{3}\lesssim ( C_\phi +\varepsilon^{-1})T^\frac{1}{2}.
		\end{align*}
		Applying Young's inequality we obtain that 
		\begin{align*}
			\|\partial_x^l f_{2,2}\|_{L^p_{T}( [0,\varepsilon])}&\lesssim \sup_{x\in[0,\varepsilon]}\left(\int_0^T\int_0^\infty|\partial_x^l G(t,x,y)|^\frac{3}{2+l}dydt\right)^\frac{2+l}{3}\|f_x\|_{L^\frac{3p}{3+(1-l)p}_{T}}\\
			&\lesssim( C_\phi +\varepsilon^{-1})T^\frac{1}{2} \|f_x\|_{L^\frac{3p}{3+(1-l)p}_{T}}.
		\end{align*}
		Similarly, we have 
		\begin{align*}
			\|\partial_x^l f_{2,2}\|_{L^p_{T}(\mathbb{I}_{\varepsilon})}\lesssim ( C_\phi +\varepsilon^{-1}) T^\frac{1}{2} \|f_x\|_{L^\frac{3p}{3+(1-l)p}_{T}}.
		\end{align*}
		Combining this with \eqref{f21lp}, we obtain that 
		\begin{align}\label{f2lp}
			\|\partial_x^l f_{2}\|_{L^p_{T}(\mathbb{I}_{\varepsilon})}\lesssim\| F\|_{L^\frac{3p}{3+(1-l)p}_{T}}+( C_\phi +\varepsilon^{-1}) T^\frac{1}{2} \|f_x\|_{L^\frac{3p}{3+(1-l)p}_{T}}.
		\end{align}
		Then we consider the boundary term $f_{2,B}$. Recall the formula $\partial_x^l	f_{2,B}^+=II_{1,l}+II_{2,l}+II_{3,l}$ in \ref{forf2b}.
		%, we have  
		%\begin{align*}
		%	\|\partial_x^l	f_{2,B}^+\|_{L^p_{T}}\lesssim& \left\|\mathbf{1}_{x\geq 0}\int_0^t \int\left|\partial_t \mathbf{K}(c_+(t-\tau),x+y)\right||\partial_x^l f_2^-(\tau,-y)|\mathbf{1}_{[0,\varepsilon]}dyd\tau\right\|_{L^{p}_{T}}\\
		%	&+\left\|\mathbf{1}_{x\geq 0}\int_0^t \int\left|\partial_t \mathbf{K}(c_+(t-\tau),x+y)\right||\partial_x^l f_2^-(\tau,-y)|\mathbf{1}_{[\varepsilon,\infty)}dyd\tau\right\|_{L^{p}_{T}}\\
		%	&+\left\|\mathbf{1}_{x\geq 0}\int_0^t \int_0^\infty \partial_x^{1+l} \mathbf{K}(c_+(t-\tau),x+y)(\tilde F(\tau,y)-\tilde F(\tau,-y))dyd\tau\right\|_{L^p_{T}}\\
		%	:=&\mathcal{B}_{1,l}+\mathcal{B}_{2,l}+\mathcal{B}_{3,l}.
		%\end{align*}
		We apply Lemma \ref{lemupg} to obtain that 
		$$\|II_{1,l}\|_{L^p_{T}}\lesssim \|\partial_x^l f_2\|_{L^p_{T}([0,\varepsilon])}.$$
		For ${II}_{2,l}$, the kernel is not singular in view of the fact that $x+y\geq \varepsilon$. Applying Young's inequality to obtain 
		\begin{align*}
			\|II_{2,l}\|_{L^p_{T}}\lesssim \|\partial_x^l f_2\|_{L^p_{T}}\int_0^t \int_\varepsilon^\infty |\partial_t \mathbf{K}(c_+\tau,y)|dyd\tau\overset{\eqref{tutu}}\lesssim \varepsilon^{-1}T^{\frac{1}{2}}\|\partial_x^l f_2\|_{L^p_{T}}.
		\end{align*}
		One can estimate $II_{3,l}$ similarly as $\partial_x^lf_{2}$. We conclude that 
		\begin{align*}
			&\|\partial_x^l	f_{2,B}^+\|_{L^p_{T}([0,\varepsilon])}\\
			&	\lesssim \|\partial_x^l f_2\|_{L^p_{T}([0,\varepsilon])}+\varepsilon^{-1}T^{\frac{1}{2}} \|\partial_x^lf_2\|_{L^p_{T}}+ \| F\|_{L^\frac{3p}{3+(1-l)p}_{T}}+( C_\phi +\varepsilon^{-1}) T^\frac{1}{2} \|f_x\|_{L^\frac{3p}{3+(1-l)p}_{T}}\\
			&\overset{\eqref{rrr}} \lesssim\|\partial_x^l f_2\|_{L^p_{T}([0,\varepsilon])}+(1+\varepsilon^{-1}T^\frac{1}{2})\| F\|_{L^\frac{3p}{3+(1-l)p}_{T}}+( C_\phi +\varepsilon^{-1}) T^\frac{1}{2} \|f_x\|_{L^\frac{3p}{3+(1-l)p}_{T}}.
		\end{align*}
		Combining this with \eqref{f2lp}, we get 
		\begin{align}\label{jumLpN}
			\|\partial_x^l f_N\|_{L^p_{T}(\mathbb{I}_{\varepsilon})}\lesssim (1+\varepsilon^{-1}T^\frac{1}{2})\| F\|_{L^\frac{3p}{3+(1-l)p}_{T}}+( C_\phi +\varepsilon^{-1}) T^\frac{1}{2} \|f_x\|_{L^\frac{3p}{3+(1-l)p}_{T}}.
		\end{align}
		\textit{3. Estimate $f_{R}$.}\\
		Recall the formula 
		\begin{align*}
			\partial_x^lf_3^\pm(t,x)=\int_0^t \int_{\mathbb{R}^\pm}( \partial_x^l\mathbf{K}(c_\pm (t-\tau),x-y)-\partial_x^l\mathbf{K}(c_\pm (t-\tau),x+y))  R(\tau,y) dyd\tau,
		\end{align*}
		and 
		\begin{align*}
			\partial_x^lf_{3,B}^+(t,x)=&C_+(c_+-c_-)\int_0^t \int_0^\infty\partial_x^2\mathbf{K}(c_+(t-\tau),x+ y)\partial_x^lf_3^-(\tau,-y)dyd\tau\\
			&-C_+\int_0^t \int_0^\infty  \partial_x^l\mathbf{K}(c_+(t-\tau),x+y)(R(\tau,y)+R(\tau,-y)))dy.
		\end{align*}
		It follows from Young's inequality and Lemma \ref{lemupg} that 
		%\begin{align*}
		%\|\partial_x^lf_R\|_{L^p_{T}}\lesssim 	\|	\partial_x^lf_3\|_{L^p_{T}}+\|	\partial_x^lf_{3,B}\|_{L^p_{T}}\lesssim \|\partial_x^l K\|_{L^\frac{3}{2+l}_{T}}\|R\|_{L^\frac{3p}{3+(1-l)p}_{T}}
		%\lesssim  T^{\frac{1+2l}{4+2l}}\|R\|_{L^\frac{3p}{3+(1-l)p}_{T}}.
		%\end{align*}
		\begin{align*}
			\|f_R\|_{L^2_{T}}+\|\partial_xf_R\|_{L^\frac{6}{5}_{T}}\lesssim \| \mathbf{K}\|_{L^2_{T}}\|R\|_{L^1_{T}}+\|\partial_x \mathbf{K}\|_{L^\frac{6}{5}_{T}}\|R\|_{L^1_{T}}\lesssim T^\frac{1}{4}\|R\|_{L^1_{T}}.
		\end{align*}
		Combining this with \eqref{jumLpL} and \eqref{jumLpN}, we obtain the result by taking $T=\frac{1}{(10+C_\phi+\varepsilon^{-1})^{10}}$. 
		This completes the proof.
	\end{proof}
	\subsection{Interior estimates}
	In this part we estimate the solution $f$ on domain $\mathbb{R}\backslash\mathbb{I}_\varepsilon$. The solution to \eqref{eqpara} satisfies (see Lemma \ref{forcont}) 
	\begin{align}
		f(t,x)=&\mathbf{H} (t,\cdot,x)\ast f_0(\cdot)+\int_0^t\partial_{1}\mathbf{H} (t-\tau,\cdot,x)\ast  F(\tau,\cdot)d\tau+\int_0^t\mathbf{H} (t-\tau,\cdot,x)\ast  R(\tau,\cdot)d\tau\nonumber\\
		&-\int_0^t\int \partial_1\mathbf{H} (t-\tau,x-y,x)\left({ \phi(x)}-{ \phi(y)}\right)\partial_x f(\tau,y)dyd\tau\nonumber\\
		:=&f_{L}(t,x)+f_N(t,x)+f_R(t,x)+f_E(t,x),\label{sofor1}
	\end{align}
	where we denote $\mathbf{H} (t,x,y)=\mathbf{K}({ \phi(y)}{t},x)$,  $\partial_1\mathbf{H}(t,x,y)=\partial_x\mathbf{H}(t,x,y)$ and $\partial_2\mathbf{H}(t,x,y)=\partial_y\mathbf{H}(t,x,y)$. We approximate the kernel of equation \eqref{eqpara} by $\mathbf{H}(t,x,y)$ and $f_E$ is the error term.
	
	\begin{lemma} \label{lemXTinte} Let $f$ satisfy \eqref{sofor1} with data $f_0=\partial_x \bar f_0$. 
		There exists $T>0$ such that 
		\begin{align*}
			\| f\|_{X_T^{\frac{1}{2},2}(\mathbb{R}\backslash\mathbb{I}_{\varepsilon})}+\| \partial_x f\|_{X_T^{\frac{5}{6},\frac{6}{5}}(\mathbb{R}\backslash\mathbb{I}_{\varepsilon})}\lesssim \|\bar f_0\|_{\dot W^{\frac{1}{3},\frac{6}{5}}}+\|F\|_{X_T^{\frac{5}{6},\frac{6}{5}}}+T^{\frac{1}{4}}(\|R\|_{X_T^{1,1}}+\|R\|_{L_{T}^1})+T^\frac{1}{3}\|\partial_x f\|_{X_T^{\frac{5}{6},\frac{6}{5}}}.
		\end{align*}
		Moreover, if $R=0$, then
		\begin{align*}
			&	\|\partial_xf\|_{X_T^{\frac{1}{2},2}(\mathbb{R}\backslash\mathbb{I}_{\varepsilon})}+\|\partial_xf\|_{X_T^{\frac{3}{4},\infty}(\mathbb{R}\backslash\mathbb{I}_{\varepsilon})}\lesssim \| f_0\|_{L^2}+\|F\|_{X_T^{\frac{1}{2},2}}+\|F\|_{X_T^{\frac{3}{4},\infty}}+T^{\frac{1}{3}}(\|\partial_xf\|_{X_T^{\frac{1}{2},2}}+\|\partial_xf\|_{X_T^{\frac{3}{4},\infty}}),\\
			&	\|\partial_xf\|_{X_T^{1-\gamma,\infty}(\mathbb{R}\backslash\mathbb{I}_{\varepsilon})}\lesssim \|\bar{f}_0\|_{\dot{C}^{2\gamma}}+\|F\|_{X_T^{1-\gamma,\infty}}+T^{\frac{1}{3}}\|\partial_xf\|_{X_T^{1-\gamma,\infty}}.
		\end{align*}
	\end{lemma}
	\begin{proof}
		% For simplicity, we consider $0<T<\frac{\varepsilon^{10}}{(1+C_\phi)^{10}}$ in our proof. Then we have 
		%	\begin{align*}
		%		|\partial_x (\partial_1H(t,x-y,x))|&\lesssim 	|\partial_1^2H(t,x-y,x)|+| \partial_1\partial_2H(t,x-y,x)|\\
		%		&\lesssim \frac{1+C_\phi T^\frac{1}{2}}{(t^\frac{1}{2}+|x-y|)^{3}}\lesssim \frac{1}{(t^\frac{1}{2}+|x-y|)^{3}}.
		%	\end{align*}
		We first estimate
		\begin{align*}
			f_L(t,x)=\int_{\mathbb{R}} \mathbf{H}(t,x-y,x)f_0(y) dy=\int_{\mathbb{R}} \partial_1  \mathbf{H}(t,x-y,x)\bar f_0(y) dy.
		\end{align*} 
		One has 
		\begin{align*}
			\partial_x f_L(t,x)=\int_{\mathbb{R}} (\partial_1^2  \mathbf{H}(t,x-y,x)+ \partial_1\partial_2\mathbf{H}(t,x-y,x))\bar f_0(y) dy.	
		\end{align*}
		Note that $\mathbf{H}(t,x,y)$ is a time rescaling of heat kernel. We can write
		\begin{align*}
			\partial_1^2  \mathbf{H}(t,z,x)+ \partial_1\partial_2\mathbf{H}(t,z,x)=\partial_x^2 \mathbf{K}({\phi(x)t},z)+t\phi'(x)\partial_t \partial_x\mathbf{K}(\phi(x)t,z):=\bar G(t,z,x).
		\end{align*}
		By Lemma \ref{lemheat}, it is easy to check that $\bar G(t,z,x)$ satisfies the conditions in Remark \ref{remlem3.2}.
		Then Lemma \ref{lemXTl} and Remark \ref{remlem3.2} imply that 
		\begin{equation}\label{I1XT}
			\begin{aligned}
				&\|	 f_L\|_{X_T^{\frac{1}{2},2}(\mathbb{R}\backslash\mathbb{I}_{\varepsilon})}+	\|	\partial_x f_L\|_{X_T^{\frac{5}{6},\frac{6}{5}}(\mathbb{R}\backslash\mathbb{I}_{\varepsilon})}\lesssim \|\bar f_0\|_{\dot W^{\frac{1}{3}, \frac{6}{5}}},\ \ \ \ \ \ \ \|	\partial_x f_L\|_{X_T^{1-\gamma,\infty}(\mathbb{R}\backslash\mathbb{I}_{\varepsilon})}\lesssim \|\bar{f}_0\|_{\dot{C}^{2\gamma}},
				\\
				&	\|	\partial_x f_L\|_{X_T^{\frac{1}{2},2}(\mathbb{R}\backslash\mathbb{I}_{\varepsilon})}+\|	\partial_x f_L\|_{X_T^{\frac{3}{4},\infty}(\mathbb{R}\backslash\mathbb{I}_{\varepsilon})}\lesssim \|\partial_x\bar f_0\|_{L^2}.
			\end{aligned}
		\end{equation}
		Then we estimate the nonlinear term 
		\begin{equation*}
			f_N(t,x)=
			\int_0^t\int_{\mathbb{R}} \partial_{1}\mathbf{H} (t-\tau,x-y,x) F(\tau,y)dyd\tau.
		\end{equation*}
		We have 
		\begin{align*}
			\partial_x f_N(t,x)=\int_0^t \int_{\mathbb{R}} \bar G (t-\tau,x-y,x)F(\tau,y)dyd\tau.
		\end{align*}
		By Remark \ref{rem2.2} we have 
		\begin{align}\label{fNXT1}
			\left\|	\partial_x f_N\right\|_{X_T^{\sigma,p}}\lesssim \|F\|_{X_T^{\sigma,p}},\ \ \ \forall \ \sigma\in(0,1-\alpha), \ p\in[1,\infty].
		\end{align}
		%	Moreover, we have 
		%	\begin{align*}
		%	\sup_{x\in\mathbb{R}\backslash\mathbb{I}_{
		%			\varepsilon}}\int |\partial_1\partial_2\mathbf{H} (t,y,x)|dy \lesssim C_\phi t^{-\frac{1}{2}}.
		%	\end{align*}
		Since
		\begin{align*}
			\sup_{z\in\mathbb{R}\backslash\mathbb{I}_{\varepsilon}} \|\partial_1\mathbf{H}(t,\cdot,z)\|_{L^{\frac{3}{2}}}\lesssim t^{-\frac{2}{3}}.
		\end{align*}
		Then 	Lemma \ref{lemlow} implies that
		$$
		\left\|\int_0^t \int_{\mathbb{R}} \partial_1\mathbf{H} (t-\tau,x-y,x)F(\tau,y)dyd\tau\right\|_{X_T^{\frac{1}{2},2}(\mathbb{R}\backslash\mathbb{I}_{\varepsilon})}\lesssim \|F\|_{X_T^{\frac{5}{6},\frac{6}{5}}}.
		$$
		Combining this with \eqref{fNXT1} to obtain that
		\begin{equation}
			\begin{aligned}\label{DfNXT}
				&	\|f_N\|_{X_T^{\frac{1}{2},2}(\mathbb{R}\backslash\mathbb{I}_{\varepsilon})}+\|\partial_x f_N\|_{X_T^{\frac{5}{6},\frac{6}{5}}(\mathbb{R}\backslash\mathbb{I}_{\varepsilon})}\lesssim\|F\|_{X_T^{\frac{5}{6},\frac{6}{5}}},
				\\
				&\|\partial_xf_N\|_{X_T^{\frac{1}{2},2}(\mathbb{R}\backslash\mathbb{I}_{\varepsilon})}+\|\partial_xf_N\|_{X_T^{\frac{3}{4},\infty}(\mathbb{R}\backslash\mathbb{I}_{\varepsilon})}\lesssim \|F\|_{X_T^{\frac{1}{2},2}}+\|F\|_{X_T^{\frac{3}{4},\infty}},\\
				&	\|\partial_xf_N\|_{X_T^{1-\gamma,\infty}(\mathbb{R}\backslash\mathbb{I}_{\varepsilon})}\lesssim \|F\|_{X_T^{1-\gamma,\infty}}.
			\end{aligned}
		\end{equation}
		Then we estimate  $$f_R=\int_0^t\int_{\mathbb{R}}\mathbf{H} (t-\tau,x-y,x) R(\tau,y)dyd\tau.$$
		By Lemma \ref{rem}, we obtain that 
		\begin{align}\label{DfRXT}
			\|f_R\|_{X_T^{\frac{1}{2},2}(\mathbb{R}\backslash\mathbb{I}_{\varepsilon})}+\|\partial_x f_R\|_{X_T^{\frac{5}{6},\frac{6}{5}}(\mathbb{R}\backslash\mathbb{I}_{\varepsilon})}\lesssim T^{\frac{1}{4}}(\|R\|_{X_T^{1,1}}+\|R\|_{L_{T}^1}).
		\end{align}
		Finally, we estimate $f_E$. One has
		\begin{align*}
			f_E(t,x)=&-\int_0^t\int_{\mathbb{R}} G_1(t-\tau,x,y)\partial_x f(\tau,y)dyd\tau,
		\end{align*}
		where 
		\begin{align*}
			G_1(t-\tau,x,y)=\partial_1\mathbf{H} (t-\tau,x-y,x)\left({ \phi(x)}-{ \phi(y)}\right).
		\end{align*}
		Note that 
		$$
		|\phi(x)-\phi(y)|\lesssim C_\phi|x-y|,\ \ \ \forall x\in \mathbb{R}\backslash\mathbb{I}_{\varepsilon},\ |y-x|\leq \varepsilon.
		$$
		Hence for any $0<s<t<T$,
		\begin{align*}
			&	\sup_{x\in\mathbb{R}\backslash\mathbb{I}_{\varepsilon} }\int_{\mathbb{R}} |\partial_x G_1(t,x,y)|dy\lesssim(\varepsilon^{-1}+C_\phi) t^{-\frac{1}{2}},\\
			&\sup_{x\in\mathbb{R}\backslash\mathbb{I}_{\varepsilon} }\int_{\mathbb{R}} |\partial_x G_1(t,x,y)-\partial_x G_1(s,x,y)|dy\lesssim(\varepsilon^{-1}+C_\phi) s^{-\frac{1}{2}}\min\{1,\frac{t-s}{s}\},\\
			&\sup_{x\in\mathbb{R}\backslash\mathbb{I}_{\varepsilon} }\| G_1(t,x,\cdot)\|_{L^\frac{3}{2}}dy\lesssim C_{\phi}t^{-\frac{1}{6}},\\
			&\sup_{x\in\mathbb{R}\backslash\mathbb{I}_{\varepsilon} }\ |G_1(t,x,\cdot)- G_1(s,x,\cdot)\|_{L^\frac{3}{2}}\lesssim C_{\phi}t^{-\frac{1}{6}}\min\{1,\frac{t-s}{s}\}.
		\end{align*}
		Combining this with  Lemma \ref{lemlow} to obtain that for $(\sigma,p)\in\{(1-\gamma,\infty),(\frac{1}{2},2),(\frac{3}{4},\infty),(\frac{5}{6},\frac{6}{5})\}$, 
		\begin{align}\label{DfEXT}
			\|\partial_x f_{E}\|_{X_T^{\sigma,p}}\lesssim T^\frac{1}{2}(C_\phi+\varepsilon^{-1})\|\partial_x f\|_{X_T^{\sigma,p}},
		\end{align}
		and 
		\begin{align}\label{DfEXTf}
			\|f_{E}\|_{X_T^{\frac{1}{2},2}}\lesssim T^{\frac{1}{2}}C_\phi\|\partial_x f\|_{X_T^{\frac{5}{6},\frac{6}{5}}}.
		\end{align}
		Take  $T=\frac{1}{(1+C_\phi+\varepsilon^{-1})^{10}}$. Then  the result follows from  \eqref{I1XT}, \eqref{DfNXT}, \eqref{DfRXT}, \eqref{DfEXT} and \eqref{DfEXTf}.
	\end{proof}
	
	\begin{lemma}\label{lemLPint}
		Suppose $f_0=\partial_x\bar f_0$. There exists $T>0$ such that 
		\begin{align*}
			\|f\|_{L^2_{T}(\mathbb{R}\backslash\mathbb{I}_{\varepsilon})}+	\|\partial_x f\|_{L^\frac{6}{5}_{T}(\mathbb{R}\backslash\mathbb{I}_{\varepsilon})}\leq& C_1(\|\bar f_0\|_{\dot W^{\frac{1}{3},\frac{6}{5}}}+ \| F\|_{L^\frac{6}{5}_{T}}+T^{\frac{1}{4}}\|R\|_{L^1_{T}})\\
			&+\frac{1}{10}(\|f\|_{L^2_{T}}+\|\partial_x f\|_{L^\frac{6}{5}_{T}}).
		\end{align*}
		Moreover, if $R=0$, there holds 
		$$
		\|f\|_{L^6_{T}(\mathbb{R}\backslash\mathbb{I}_{\varepsilon})}+\|\partial_xf\|_{L^2_{T}(\mathbb{R}\backslash\mathbb{I}_{\varepsilon})}\leq C_1 (\|f_0\|_{L^2}+\|F\|_{L^2_{T}})+\frac{1}{10}(\|f\|_{L^6_{T}}+\|\partial_x f\|_{L^2_{T}}).
		$$
	\end{lemma}
	\begin{proof}
		There exists a sequence $\{z_n\}_{n\in\mathbb{Z}}\subset\mathbb{R}\backslash\mathbb{I}_\varepsilon$ such that  $\sum_n\chi_n^\ell(x)=1,\  \forall x\in \mathbb{R}\backslash\mathbb{I}_\varepsilon$ and  $\operatorname{supp}(\chi_j^\ell)\cap \operatorname{supp}(\chi_k^\ell)=\emptyset$ if $|j-k|\geq 2$, where $\chi_n^\ell:\mathbb{R}\to [0,1]$ is a smooth cutoff function  satisfying $\mathbf{1}_{[z_n-\ell/2,z_n+\ell/2]}\leq \chi_n^\ell\leq \mathbf{1}_{[z_n-\ell,z_n+\ell]}$ with $0<\ell<\varepsilon$. 
		We rewrite the equation \eqref{eqpara} as 
		\begin{align*}
			&	\partial_t (f\chi_n^\ell)-\phi(z_n)\partial_x ^2(f\chi_n^\ell)=\partial_x F_1+R_1,\\
			&	(f\chi_n^\ell)_{t=0}=\partial_x(\bar f_0\chi_n^\ell)-\bar f_0\partial_x\chi_n^\ell,
		\end{align*}
		where 
		\begin{align*}
			F_1=F\chi_n^\ell+(\phi(x)-\phi(z_n))\partial_x f\chi_n^\ell+\phi(z_n)f\partial_x \chi_n^\ell,\ \ \ \ R_1=\phi(x)\partial_x f\partial_x \chi_n^\ell
			+R\chi_n^\ell.
		\end{align*}
		Applying Lemma \ref{lemap} to obtain that 
		\begin{align*}
			&\|f\chi_n^\ell\|_{L^2_{T}}+\|\partial_x(f\chi_n^\ell)\|_{L^\frac{6}{5}_{T}}\lesssim \|\bar f_0\chi_n^\ell\|_{\dot W^{\frac{1}{3},\frac{6}{5}}}+T^\frac{1}{3}\|\bar f_0\partial_x\chi_n^\ell\|_{L^\frac{6}{5}}+ \|F_1\|_{L^\frac{6}{5}_{T}}+T^\frac{1}{4}\|R_1\|_{L^1_{T}},
		\end{align*}
		and 
		\begin{align*}
			&\|f\chi_n^\ell\|_{L^6_{T}}+\|\partial_x(f\chi_n^\ell)\|_{L^2_{T}}\lesssim \|f_0\chi_n^\ell\|_{L^2}+\|F_1\|_{L^2_{T}}+T^\frac{1}{4}\ell^{-1}\|R_1\|_{L^2_{T}},\ \ \ \ \text{if}\ R=0.
		\end{align*}
		Note that 
		\begin{align*}
			\|F_1\|_{L^\frac{6}{5}_{T}}&\lesssim \|F\chi_n^\ell\|_{L^\frac{6}{5}_{T}}+\|\phi(x)-\phi(z_n)\|_{L^\infty (B_\ell)}\|\partial_x f\|_{L^\frac{6}{5}_{T}(B_{\ell})}+\|f\|_{L^2_{t,x, T}(B_\ell)}\|\partial_x \chi_n^\ell\|_{L^3_{T}}\\
			&\lesssim  \|F\|_{L^\frac{6}{5}_{T}(B_{\ell})}+\ell C_\phi\|\partial_x f\|_{L^\frac{6}{5}_{T}(B_{\ell})}+\ell^{-\frac{2}{3}}T^\frac{1}{3}\|f\|_{L^2_{T}(B_{\ell})},\\
			\|R_1\|_{L^1_{T}}&\lesssim \|\partial_x f\|_{L^\frac{6}{5}_{T}(B_{\ell})}\|\partial_x \chi_n^\ell\|_{L^4_{T}}+\|R\|_{L^1_{T}(B_\ell)}\\
			&\lesssim T^{\frac{1}{4}}\ell^{-\frac{3}{4}}\|\partial_x f\|_{L^\frac{6}{5}_{T}(B_{\ell})}+\|R\|_{L^1_{T}(B_\ell)},
		\end{align*}
		where we denote $B_\ell=[z_n-\ell,z_n+\ell]$. 
		Moreover, we can check that 
		\begin{align*}
			\sum_n\|\bar f_0\chi_n^\ell\|_{\dot W^{\frac{1}{3},\frac{6}{5}}}\lesssim \|\bar f_0\|_{\dot W^{\frac{1}{3},\frac{6}{5}}}+\ell^{-1}\|\bar f_0\|_{L^2}\lesssim (1+\ell^{-1})\|\bar f_0\|_{\dot W^{\frac{1}{3},\frac{6}{5}}}.
		\end{align*}
		Taking sum of $n$, we obtain that 
		\begin{align*}
			\|f\|_{L^2_{T}(\mathbb{R}\backslash\mathbb{I}_\varepsilon)}+\|\partial_xf\|_{L^\frac{6}{5}_{T}(\mathbb{R}\backslash\mathbb{I}_\varepsilon)}\leq&\tilde C_1(1+\ell^{-1})\|\bar f_0\|_{\dot W^{\frac{1}{3},\frac{6}{5}}}+\tilde C_1\|F\|_{L^\frac{6}{5}_{T}}+\tilde C_1T^\frac{1}{4}\|R\|_{L^1_{T}}\\
			&+\tilde C_1(\ell C_\phi+T^\frac{1}{3}\ell^{-\frac{3}{4}})(\|\partial_x f\|_{L^\frac{6}{5}_{T}}+\|f\|_{L^2_{T}}) .
		\end{align*}
		Similarly, when $R=0$, we have 
		\begin{align*}
			\|f\|_{L^6_{T}(\mathbb{R}\backslash\mathbb{I}_\varepsilon)}+\|\partial_xf\|_{L^2_{T}(\mathbb{R}\backslash\mathbb{I}_\varepsilon)}\leq&\tilde C_1(1+\ell^{-1})\|f_0\|_{L^2}+\tilde C_1\|F\|_{L^2_{T}}\\
			&+\tilde C_1(\ell C_\phi+\ell^{-1}T^\frac{1}{4})(\|\partial_x f\|_{L^2_{T}}+\|f\|_{L^6_{T}}).
		\end{align*}
		Then we obtain the result by taking $\ell ={(10+\tilde C_1+C_\phi)^{-10}}$ and $T=\ell^{10}$. 
		This completes the proof.
	\end{proof}
	\begin{remark}
		We have two ways to prove Lemma \ref{lemLPint}. One is to approximate the kernel (see \eqref{sofor1}), and the other is to cut off in space. We choose the latter one for simplicity. It is equivalent to approximate the kernel piece wisely, namely, approximate the kernel by $\sum_n \mathbf{H}(t,x,z_n)\chi_n^l(x)$.
	\end{remark}
	\section{Well-posedness for the isentropic Navier–Stokes equations}\label{jum}
	In this section, we consider the Cauchy problem of \eqref{inns} with initial data 
	$
	(v(0,x),u(0,x))=(v_0,u_0)
	$ satisfying 
	\begin{align}\label{inicon}
		v_0\in L^\infty,\ \  \inf_x v_0\geq \lambda_0>0,\ \  u_0=\partial_x\bar u_0\ \text{with}\ \bar u_0\in\dot C^{2\gamma}.
	\end{align}
	The initial data $v_0$ is allowed to have finite jump. More precisely, we assume there exists a sequence  $\{a_n\}_{n=1}^N$ satisfying 
	\begin{align}\label{conccc}
		&\min _{j\neq k} |a_j-a_k|\geq \sigma >0,\ \ \ \
		\|v_0-b_{\varepsilon,\eta}\|_{L^\infty}\leq \varepsilon_0,
	\end{align}
	for some constants $0<\varepsilon,\eta\ll \sigma$. Here 
	\begin{align}\label{defb}
		& b_{\varepsilon,\eta}(x)=\sum _{n=1}^{N+1}\left(v_0\ast\rho_\eta\mathbf{1}_{(a_{n-1}+\varepsilon,a_{n}-\varepsilon)}+(v_0\ast\rho_\eta)(a_n+\varepsilon)\mathbf{1}_{[a_n,a_n+\varepsilon]}+v_0\ast \rho_\eta(a_n-\varepsilon)\mathbf{1}_{[a_n-\varepsilon,a_n]}\right),
	\end{align}
	where we take  $a_{0}=-\infty$ and $a_{N+1}=+\infty$.\\
	The main result is the following 
	\begin{thm}\label{thmisen}
		There exists $\varepsilon_0>0$ such that if  initial data $(v_0,u_0)$ satisfies \eqref{inicon} and  \eqref{conccc}, then the system \eqref{cpns} admits a unique local solution $(v,u)$ in $[0,T]$ satisfying 
		\begin{align*}
			\inf_{t\in[0,T]}\inf_xv(t,x)\geq \frac{\lambda_0}{2},\ \ \ \|v\|_{Y_T}\leq 2\|v_0\|_{L^\infty},\ \ \ \ \|\partial_xu\|_{X_T^{1-\gamma,\infty}}\leq M.	
		\end{align*}
		for some $T,M>0$.
	\end{thm}
	Clearly, the theorem \eqref{maininns} is as a consequece of  Theorem \ref{thmisen}. \\
	Define the space 
	\begin{align*}
		\mathcal{E}_{T,M}:=\{w:w(0,x)=u_0(x),\|\partial_x w\|_{X_T^{1-\gamma,\infty}}\leq M\}.
	\end{align*}
	Consider $w\in\mathcal{E}_{T,M}$, we define a map $\mathcal{T}w=u$, where $u$ is a solution to the equation 
	\begin{align}\label{eqjum}
		\partial_tu-\mu(\frac{\partial_x u}{b_{\varepsilon,\eta}})_x&=-(p(v))_x+\mu((\frac{1}{v}-\frac{1}{b_{\varepsilon,\eta}})\partial_x w)_x=:\tilde F(v,w)_x,
	\end{align}
	where $v$ is defined by 
	\begin{align}\label{defv}
		v(t,x)=v_0(x)+\int_0^t \partial_x w(s,x)ds.
	\end{align}
	The following is a key lemma to prove Theorem \ref{thmisen}.
	\begin{lemma}\label{lemrejum}
		For any initial data $$v_0\in L^\infty,\ \  \inf_x v_0\geq \lambda_0>0,\ \  u_0=\partial_x\bar u_0\ \text{with}\ \bar u_0\in\dot C^{2\gamma},$$	there exists $\varepsilon_0>0$ such that if \eqref{conccc} holds,  then 
		\begin{align*}
			& \|\partial_x(\mathcal{T}w)\|_{X_T^{1-\gamma,\infty}}\leq M,\ \ \forall w\in \mathcal{E}_{T,M},\\
			&\|\partial_x(\mathcal{T}w_1-\mathcal{T}w_2)\|_{X_T^{1-\gamma,\infty}}\leq \frac{1}{2}\|\partial_x(w_1-w_2)\|_{X_T^{1-\gamma,\infty}}, \ \ \forall w_1,w_2\in \mathcal{E}_{T,M},
		\end{align*}
		for some $T,M>0$ depending on $\varepsilon_0,\lambda_0, \|v_0\|_{L^\infty}$ and $\|\bar u_0\|_{\dot C^{2\gamma}}$.
	\end{lemma}
	\begin{proof}
		We first show the estimates of $v$. By definition \eqref{defv}, it is easy to check that 
		\begin{align*}
			&\sup_{s\in[0,T]}\|v-b_{\varepsilon,\eta}\|_{L^\infty}\lesssim \|v_0-b_{\varepsilon,\eta}\|_{L^\infty}+\int_0^T\|w_{x}(\tau)\|_{L^\infty}d\tau\lesssim \|v_0-b_{\varepsilon,\eta}\|_{L^\infty}+T^{\gamma}\|\partial_x w\|_{X_T^{1-\gamma,\infty}},\\
			&\sup_{0<s<t<T}s^\alpha\frac{\|v(t)-v(s)\|_{L^\infty}}{(t-s)^\alpha}\lesssim \sup_{0<s<t<T}s^\alpha\frac{\int_s^t\|\partial_x w\|_{L^\infty}d\tau}{(t-s)^\alpha}\\
			&\quad\ \quad\quad\quad\quad\quad\quad\quad \ \ \ \ \ \ \ \ \ \ \ \lesssim \|\partial_x w\|_{X_T^{1-\gamma,\infty}}\sup_{s<t<T}s^\alpha \frac{t^\gamma-s^\gamma}{(t-s)^\alpha}\lesssim T^\gamma\|\partial_x w\|_{X_T^{1-\gamma,\infty}}.
		\end{align*}
		Hence,
		\begin{align*}
			\|v-b_{\varepsilon,\eta}\|_{Y_T}\leq \|v_0-b_{\varepsilon,\eta}\|_{L^\infty}+T^{\gamma}\|\partial_x w\|_{X_T^{1-\gamma,\infty}}\overset{\eqref{conccc}}\leq \varepsilon_0+T^{\gamma}M.
		\end{align*}
		Moreover, we have 
		\begin{align*}
			\inf_{s\in[0,T]}	\inf_x v(s,x)\geq \inf_{x}v_0(x)-\sup_{s\in[0,T]}\|v(s)-v_0\|_{L^\infty}\geq \lambda_0-T^\gamma\|\partial_x w\|_{X_T^{1-\gamma,\infty}}\geq \lambda_0-T^\gamma M.
		\end{align*}
		We can take $T$ small enough such that $T^\gamma M \leq \frac{1}{10}\min\{\varepsilon_0,\lambda_0\}$, then 
		\begin{align}\label{estofv}
			\|v-b_{\varepsilon,\eta}\|_{Y_T}\leq 2\varepsilon_0,\ \ \ \ \ \inf_{s\in[0,T]}	\inf_x v(s,x)\geq \frac{\lambda_0}{2}.
		\end{align}	
		Applying  Lemma \ref{lemma} to \eqref{eqjum}, there exists $T_1>0$ such that 
		\begin{align}\label{haha}
			\|\partial_x (\mathcal{T}w)\|_{X_T^{1-\gamma,\infty}}\lesssim \|\bar u_0\|_{\dot C^{2\gamma}}+\|\tilde F(v,w)\|_{X_T^{1-\gamma,\infty}}, \ \ \ \forall 0<T<T_1.
		\end{align}
		We have 
		\begin{align*}
			\|\tilde F(v,w)\|_{X_T^{1-\gamma,\infty}}\lesssim \|p(v)\|_{X_T^{1-\gamma,\infty}}+\left\|(\frac{1}{v}-\frac{1}{b_{\varepsilon,\eta}})\partial_x w\right\|_{X_T^{1-\gamma,\infty}}.
		\end{align*}
		Note that $p\in W^{2,\infty}$, hence 
		\begin{align*}
			&\sup_{t\in[0,T]}t^{1-\gamma}\|p(v)(t)\|_{L^\infty}\lesssim 	T^{1-\gamma},\\
			&\frac{\|p(v)(t)-p(v)(s)\|_{L^\infty}}{(t-s)^\alpha}\lesssim \frac{\|v(t)-v(s)\|_{L^\infty}}{(t-s)^\alpha}\lesssim s^{-\alpha}\|v-b_{\varepsilon}\|_{Y_T},\ \ \forall 0<s<t<T.
		\end{align*}
		Hence,
		\begin{align}\label{PXT}
			\|p(v)\|_{X_T^{1-\gamma,\infty}}\lesssim 	T^{1-\gamma}(1+\|v-b_{\varepsilon}\|_{Y_T}).
		\end{align}
		Denote $q(t,x)=(\frac{1}{v(t,x)}-\frac{1}{b_{\varepsilon,\eta}(x)})\partial_x w(t,x)$, we have 
		\begin{align}\label{qinf}
			\left\|q(t)\right\|_{L^\infty}\lesssim \|v(t)-b_{\varepsilon,\eta}\|_{L^\infty}\|\partial_x w(t)\|_{L^\infty}\lesssim t^{\gamma-1}\|v-b_{\varepsilon,\eta}\|_{Y_T}\|\partial_x w\|_{X_T^{1-\gamma,\infty}}.
		\end{align}
		Moreover, for $0<s<t<T$, we have 
		\begin{align*}
			q(t,x)-q(s,x)=(\frac{1}{v(t)}-\frac{1}{v(s)})\partial_x w(t)+(\frac{1}{v(s)}-\frac{1}{b_{\varepsilon,\eta}})(\partial_x w(t)-\partial_x w(s)).
		\end{align*}
		Note that 
		\begin{align*}
			&	\left\|(\frac{1}{v(t)}-\frac{1}{v(s)})\partial_x w(t)\right\|_{L^\infty}\lesssim \|v(t)-v(s)\|_{L^\infty}\|\partial_x w(t)\|_{L^\infty}\\
			&\quad\quad\quad\quad\quad\quad\quad\quad\quad\quad\quad\ \lesssim (t-s)^{\alpha} s^{\gamma-1-\alpha}\|v-b_{\varepsilon,\eta}\|_{Y_T}\|\partial_x w\|_{X_T^{1-\gamma,\infty}},\\
			&	\left\|(\frac{1}{v(s)}-\frac{1}{b_{\varepsilon,\eta}})(\partial_x w(t)-\partial_x w(s))\right\|_{L^\infty}\lesssim \|v(s)-b_{\varepsilon,\eta}\|_{L^\infty}\|w(t)-w(s)\|_{L^\infty}\\
			&\quad\quad\ \quad\quad\quad\quad\quad\quad\quad\quad\quad\quad\quad\quad\quad\lesssim (t-s)^\alpha s^{\gamma-1-\alpha}\|v-b_{\varepsilon,\eta}\|_{Y_T}\|\partial_x w\|_{X_T^{1-\gamma,\infty}}.
		\end{align*}
		Then one gets
		\begin{align}\label{qdif}
			\sup_{0<s<t<T}	s^{1-\gamma+\alpha}\frac{\left\|	q(t)-q(s)\right\|_{L^\infty}}{(t-s)^\alpha}\lesssim \|v-b_{\varepsilon,\eta}\|_{Y_T}\|\partial_x w\|_{X_T^{1-\gamma,\infty}}.
		\end{align}
		Hence we get from \eqref{qinf} and \eqref{qdif} that 
		\begin{align*}
			\|q\|_{X_T^{1-\gamma,\infty}}\lesssim \|v-b_{\varepsilon,\eta}\|_{Y_T}\|\partial_x w\|_{X_T^{1-\gamma,\infty}}.
		\end{align*}
		Combining this with \eqref{PXT} to obtain that 
		\begin{align}\label{tilF}
			\|\tilde F(v,w)\|_{X_T^{1-\gamma,\infty}}&\leq C_1 T^{1-\gamma}(1+\|v-b_{\varepsilon}\|_{Y_T})+ C_1\|v-b_{\varepsilon,\eta}\|_{Y_T}\|\partial_x w\|_{X_T^{1-\gamma,\infty}}.
		\end{align}
		We conclude from  \eqref{estofv}, \eqref{haha} and \eqref{tilF}  that 
		\begin{align*}
			\|\partial_x (\mathcal{T}w)\|_{X_T^{1-\gamma,\infty}}&\leq C_1 \|\bar u_0\|_{\dot C^{2\gamma}}+C_1 T^{1-\gamma}(1+2\varepsilon_0)+2C_1 \varepsilon_0M, \ \ \ \ \ \forall w\in \mathcal{E}_{T,M}.
		\end{align*}
		Then we obtain   $\|\partial_x (\mathcal{T}w)\|_{X_T^{1-\gamma,\infty}}\leq M$  by taking $M=2C_1 \|\bar u_0\|_{\dot C^{2\gamma}}+1$, $\varepsilon_0<\frac{1}{(1+C_1)^{10}}$ and $T<\min\{\frac{1}{(1+C_1)^{10}},T_1\}$. This implies that $\mathcal{T}$ maps $\mathcal{E}_{T,M}$ to itself.\vspace{0.3cm}\\
		It remains to prove that $\mathcal{T}$  is contraction. Consider $w_1,w_2\in 	\mathcal{E}_{T,M}$, let $v_m(t,x)=v_0(x)+\int_0^t \partial_x w_m(s,x)ds$, $m=1,2$. Then it is easy to check that both $v_1$ and $v_2$ satisfy \eqref{estofv}, and 
		\begin{align}\label{difv}
			\|v_1-v_2\|_{Y_T}\lesssim T^\gamma \|\partial_x (w_1-w_2)\|_{X_T^{1-\gamma,\infty}}.
		\end{align}
		Thanks to Lemma \ref{lemma}, there exists $T_2>0$ such that 
		\begin{align}\label{contra}
			\|\partial_x(\mathcal{T}w_1-	\mathcal{T}w_2)\|_{X_T^{1-\gamma,\infty}}\lesssim \|\tilde F(v_1,w_1)-\tilde F(v_2,w_2)\|_{X_T^{1-\gamma,\infty}},\ \ \forall 0<T<T_2.
		\end{align}
		Note that 
		\begin{align*}
			\tilde	F(v_1,w_1)-\tilde F(v_2,w_2)=p(v_2)-p(v_1)+\mu \left(\left(\frac{1}{v_1}-\frac{1}{b_{\varepsilon,\eta}}\right)\partial_xw_1- \left(\frac{1}{v_2}-\frac{1}{b_{\varepsilon,\eta}}\right)\partial_xw_2\right).
		\end{align*}
		We first estimate $p(v_2)-p(v_1)$. Note that $p\in W^{2,\infty}$, hence 
		\begin{align*}
			&\|p(v_2(t))-p(v_1(t))\|_{L^\infty}\lesssim \|v_1-v_2\|_{L^\infty}\lesssim \|v_1-v_2\|_{Y_T},
		\end{align*}
		and \begin{align*}
			&\|p(v_2(t))-p(v_1(t))-(p(v_2(s))-p(v_1(s)))\|_{L^\infty}\\
			&\lesssim\|(v_2-v_1)(t)-(v_2-v_1)(s)\|_{L^\infty}+ \|v_2(t)-v_1(t)\|_{L^\infty}(\|v_1(t)-v_1(s)\|_{L^\infty}+\|v_2(t)-v_2(s)\|_{L^\infty})\\
			&\lesssim (t-s)^\alpha s^{-\alpha}\|v_1-v_2\|_{Y_T}(1+\|v_1-b_{\varepsilon,\eta}\|_{Y_T}+\|v_2-b_{\varepsilon,\eta}\|_{Y_T}).
		\end{align*}
		Hence 
		\begin{align}
			\|p(v_2)-p(v_1)\|_{X_T^{1-\gamma,\infty}}&\lesssim T^{1-\gamma} \|v_1-v_2\|_{Y_T}(1+\|v_1-b_{\varepsilon,\eta}\|_{Y_T}+\|v_2-b_{\varepsilon,\eta}\|_{Y_T})\nonumber\\
			&\overset{\eqref{difv}}\lesssim T\|\partial_x (w_1-w_2)\|_{X_T^{1-\gamma,\infty}}.\label{P}
		\end{align}
		Then we estimate $\mathcal{Q}=\left(\frac{1}{v_1}-\frac{1}{b_{\varepsilon,\eta}}\right)\partial_xw_1- \left(\frac{1}{v_2}-\frac{1}{b_{\varepsilon,\eta}}\right)\partial_xw_2$. We have 
		\begin{align*}
			\mathcal{Q}=\left(\frac{1}{v_1}-\frac{1}{v_2}\right)\partial_xw_1+\left(\frac{1}{v_2}-\frac{1}{b_{\varepsilon,\eta}}\right)\partial_x(w_1-w_2).
		\end{align*}
		For any $0<t<T$, one has 
		\begin{align}\label{111}
			\|\mathcal{Q}(t)\|_{L^\infty}\lesssim \|(v_1-v_2)(t)\|_{L^\infty}\|\partial_x w_1(t)\|_{L^\infty}\lesssim t^{\gamma-1}\|v_1-v_2\|_{Y_T}\|\partial_x w_1\|_{X_T^{1-\gamma,\infty}}.
		\end{align}
		For any $0<s<t<T$,
		\begin{align}
			&	\|\mathcal{Q}(t)-\mathcal{Q}(s)\|_{L^\infty}\nonumber\\
			&\lesssim\|(v_1-v_2)(t)-(v_1-v_2)(s)\|_{L^\infty}\|\partial_x w_1(t)\|_{L^\infty}\nonumber \\
			&\ \ \ \ +\|(v_1-v_2)(t)\|_{L^\infty}(\|v_1(t)-v_1(s)\|_{L^\infty}+\|v_2(t)-v_2(s)\|_{L^\infty})\|\partial_x w_1(t)\|_{L^\infty}\nonumber \\
			&\ \ \ \ +\|(v_1-v_2)(s)\|_{L^\infty}\|\partial_x (w_1(t)-w_1(s))\|_{L^\infty}+\|v_2(t)-v_2(s)\|_{L^\infty}\|\partial_x (w_1-w_2)(t)\|_{L^\infty}\nonumber \\
			&\ \ \ \ +\|v_2(s)-b_{\varepsilon,\eta}\|_{L^\infty} \|\partial_x (w_1-w_2)(t)-\partial_x (w_1-w_2)(s)\|_{L^\infty}\nonumber \\
			&\lesssim (t-s)^\alpha s^{\gamma-1-\alpha}\|v_1-v_2\|_{Y_T}\|\partial_x w_1\|_{X_T^{1-\gamma,\infty}}(1+\|v_1-b_{\varepsilon,\eta}\|_{Y_T}+\|v_2-b_{\varepsilon,\eta}\|_{Y_T})\nonumber \\
			&\quad\quad\quad\quad+(t-s)^\alpha s^{\gamma-1-\alpha}\|v_2-b_{\varepsilon,\eta}\|_{Y_T}\|\partial_x(w_1-w_2)\|_{X_T^{1-\gamma,\infty}}.\label{222}
		\end{align}
		Combining \eqref{difv}, \eqref{111}, \eqref{222} with the fact that $\|v_1-b_{\varepsilon,\eta}\|_{Y_T}+\|v_2-b_{\varepsilon,\eta}\|_{Y_T}\lesssim \varepsilon_0$ and  $\|\partial_x w_1\|_{X_T^{1-\gamma,\infty}}\leq M$, we obtain 
		\begin{align}\label{Q}
			\|\mathcal{Q}\|_{X_T^{1-\gamma,\infty}}\lesssim M T^\gamma \|\partial_x (w_1-w_2)\|_{X_T^{1-\gamma,\infty}}.
		\end{align} 
		We conclude from \eqref{contra}, \eqref{P} and \eqref{Q} that 
		$$
		\|\partial_x(\mathcal{T}w_1-	\mathcal{T}w_2)\|_{X_T^{1-\gamma,\infty}}\leq C_2( M T^\gamma+T)\|\partial_x(w_1-w_2)\|_{X_T^{1-\gamma,\infty}}\leq \frac{1}{2}\|\partial_x(w_1-w_2)\|_{X_T^{1-\gamma,\infty}}.
		$$
		Here we fix $T=\min\left\{\frac{1}{(1+C_1+C_2)^{10}},T_1,T_2\right\}$. This completes the proof.
	\end{proof}
	
	\section{Well-posedness for the full Navier–Stokes system }
	In this section, we consider the Cauchy problem of \eqref{cpns} with initial data 
	$(v,u,\theta)(0,x)=(v_0,u_0,\theta_0)(x)$ satisfying
	\begin{align}\label{conini2}
		\inf_xv_0(x)\geq \lambda_0>0,\ \ \	\|v_0\|_{L^\infty}<\infty,\ \ \|u_0\|_{L^2}<\infty,\ \ \|\theta_0\|_{\dot W^{-\frac{2}{3},\frac{6}{5}}}<\infty.
	\end{align}
	Moreover, we  suppose that there exist $0<\varepsilon,\eta\ll 1 $ such that $v_0$ satisfies the condition \eqref{conccc} for some $\varepsilon_0>0$ that will be fixed later in our proof.
	%
	%For simplicity, we denote
	%\begin{align*}
	%	\|(w,\vartheta)\|_{Z_T}=&\|\partial_x w\|_{L^2_{T}}+\|\partial_x w\|_{X_T^{\frac{1}{2},2}}+\|\partial_x w\|_{X_T^{\frac{3}{4},\infty}}\\
	%	&\ \ \ \ \ \ +\|\vartheta\|_{L^2_{T}}+\|\partial_x \vartheta\|_{L^\frac{6}{5}_{T}}+\|\vartheta\|_{X_T^{\frac{1}{2},2}}+\|\partial_x \vartheta\|_{X_T^{\frac{5}{6},\frac{6}{5}}}.
	%\end{align*}
	\begin{thm}\label{thmfull}
		There exists $\varepsilon_0>0$ such that if  initial data $(v_0,u_0,\theta_0)$ satisfies  \eqref{conini2} and \eqref{conccc}, then  the system \eqref{cpns} admits a unique local solution in $[0,T]$ satisfying 
		\begin{align*}
			\inf_{t\in[0,T]}\inf_xv(t,x)\geq \frac{\lambda_0}{2},\ \ \ \|v\|_{Y_T}\leq 2\|v_0\|_{L^\infty},\ \ \ \ \|(u,\theta)\|_{ Z_T}\leq B.	
		\end{align*}
		for some $T,B>0.$
	\end{thm}
	Clearly, the theorem \eqref{maincp} is as a consequece of  Theorem \ref{thmfull}. \\
	Define the space
	\begin{align*}
		\mathcal{X}_{T,B}=\{(w,\vartheta):(w,\vartheta)|_{t=0}=(u_0,\theta_0), \|(w,\vartheta)\|_{Z_T}\leq B\}.
	\end{align*}
	Let $(w,\vartheta)\in\mathcal{X}_{T,B}$, we define a map $\mathcal{M}(w,\vartheta)=(u,\theta)$, where $u,\theta$ solve the equations
	\begin{align*}
		&	u_t-\mu(\frac{ u_x}{b_{\varepsilon,\eta}})_x+(p(v,\theta))_x=\mu((\frac{1}{v}-\frac{1}{b_{\varepsilon,\eta}})\partial_x w)_x,\ \ \ \ 
		u(0,x)=u_0(x),
		\\
		&		\theta_{t}-\frac{\kappa}{\mathbf{c}}\left(\frac{\theta_x}{b_{\varepsilon,\eta}}\right)_x=-\frac{p(v,\vartheta)}{\mathbf{c}} w_{x}+\frac{\mu}{\mathbf{c} v}\left(w_{x}\right)^{2}+\frac{\kappa}{\mathbf{c}}((\frac{1}{v}-\frac{1}{b_{\varepsilon,\eta}})\partial_x \vartheta)_x,\ \ \ \theta(0,x)=\theta_0(x),
	\end{align*}
	where $v$ is fixed by 
	\begin{align*}
		v(t,x)=v_0(x)+\int_0^t \partial_x w(s,x)ds,
	\end{align*}
	and the function $b_{\varepsilon,\eta}$ is defined in \eqref{defb}.\\
	The following is the key lemma to prove Theorem \ref{thmfull}.
	\begin{lemma}\label{lemre2}
		There exists $\varepsilon_0>0$ such that if  initial data $(v_0,u_0,\theta_0)$ satisfies  \eqref{conini2} and \eqref{conccc}, then we have 
		\begin{align*}
			& \|\mathcal{M}(w,\vartheta)\|_{Z_T}\leq B,\ \ \forall (w,\vartheta)\in \mathcal{X}_{T,B},\\
			&\|\mathcal{M}(w_1,\vartheta_1)-\mathcal{M}(w_1,\vartheta_2)\|_{Z_T}\leq \frac{1}{2}\|(w_1-w_2,\vartheta_1-\vartheta_2)\|_{Z_T}, \ \ \forall (w_k,\vartheta_k)\in \mathcal{X}_{T,B},\  k=1,2.
		\end{align*}
		for some $T,B>0$.
	\end{lemma}
	\begin{proof}~\\
		1.\textit{Estimate $v$}\\
		We have 
		\begin{align*}
			&\|v(t)-b_{\varepsilon,\eta}\|_{L^\infty}\lesssim \|v_0-b_{\varepsilon,\eta}\|_{L^\infty}+\int_0^t \|w_x(\tau)\|_{L^\infty}d\tau,\\
			& \inf_{t\in[0,T]}\inf_x v(t,x)\geq \inf_{x}v_0(x)-\int_0^t \|w_x(\tau)\|_{L^\infty}d\tau.
		\end{align*}
		By H\"{o}lder's inequality,
		\begin{align*}
			\int_0^T \|w_x(\tau)\|_{L^\infty}d\tau \leq C T^\frac{1}{4}\|w_x\|_{X_T^{\frac{3}{4},\infty}}\leq C T^\frac{1}{4}B.
		\end{align*}
		We take $T$ small enough such that $CT^\frac{1}{4}B\leq \frac{1}{100}\min\{\lambda_0,\varepsilon_0\} $, and  
		\begin{align}\label{es2v}
			\|v-b_{\varepsilon,\eta}\|_{Y_T}\leq 2\varepsilon_0,\ \ \ \ \ \inf_{s\in[0,T]}	\inf_x v(s,x)\geq \frac{\lambda_0}{2}.
		\end{align}
		2.\textit{Estimate $\theta$}\\
		\begin{align*}
			\theta_{t}-\frac{\kappa}{\mathbf{c}}\left(\frac{\theta_x}{b_{\varepsilon,\eta}}\right)_x=-\frac{p(v,\vartheta)}{\mathbf{c}} w_{x}+\frac{\mu}{\mathbf{c} v}\left(w_{x}\right)^{2}+\frac{\kappa}{\mathbf{c}}((\frac{1}{v}-\frac{1}{b_{\varepsilon,\eta}})\partial_x \vartheta)_x:=R(v,w,\vartheta)+\partial_x \tilde F(v,\vartheta).
		\end{align*}
		Applying Lemma \ref{lemma} with the equation above for $\theta$ to obtain
		\begin{align*}
			&\sum_{\star\in\{L^2_{T},X_T^{\frac{1}{2},2}\}}\|\theta\|_{\star}+	\sum_{\star\in\{L^\frac{6}{5}_{T},X_T^{\frac{5}{6},\frac{6}{5}}\}}\|\theta_x\|_{\star}\lesssim \|\bar \theta_0\|_{\dot W^{\frac{1}{3},\frac{6}{5}}}+\sum_{\star\in\{L_{T}^\frac{6}{5},X_T^{\frac{5}{6},\frac{6}{5}}\}}\|\tilde F\|_{\star}+T^\frac{1}{4}\sum_{\star\in\{L^1_{T},X_T^{1,1}\}}\|R\|_{\star}.
		\end{align*}
		It is easy to check that 
		\begin{align*}
			&\sum_{\star\in\{L_{T}^\frac{6}{5},X_T^{\frac{5}{6},\frac{6}{5}}\}}\|\tilde F\|_{\star}\lesssim \|v-b_{\varepsilon,\eta}\|_{L^\infty_{T}}(\|\partial_x\vartheta\|_{L_{T}^\frac{6}{5}}+\|\partial_x\vartheta\|_{X_T^{\frac{5}{6},\frac{6}{5}}})\lesssim \varepsilon_0B,\\
			& \sum_{\star\in\{L^1_{T},X_T^{1,1}\}}\|R\|_{\star}\lesssim \|w_x\|_{L^2_T}^2+\|\vartheta\|_{L^2_T}^2+\|w_x\|_{X_T^{\frac{1}{2},2}}^2+\|\vartheta\|_{X_T^{\frac{1}{2},2}}^2\lesssim B^2.
		\end{align*}
		Hence we obtain 
		\begin{align}
			&\|\theta\|_{L^2_{T}}+	\|\theta_x\|_{L^\frac{6}{5}_{T}}+\|\theta\|_{X_T^{\frac{1}{2},2}}+\|\theta_x\|_{X_T^{\frac{5}{6},\frac{6}{5}}}\nonumber\\
			&\quad\quad\quad\ \ \ \ \leq \tilde C_0\|\bar \theta_0\|_{\dot W^{\frac{1}{3},\frac{6}{5}}}+\tilde C_0\varepsilon_0B+\tilde C_0T^\frac{1}{4}B^2\leq \frac{B}{3},\label{retheta}
		\end{align}
		provided $ \|\bar \theta_0\|_{\dot W^{\frac{1}{3},\frac{6}{5}}}\leq \frac{B}{\tilde C_0}$, $\varepsilon_0\leq \frac{1}{10(1+\tilde C_0)^{10}}$ and $T\leq \frac{1}{(1+\tilde C_0+B)^{10}}$.\\
		3.\textit{ Estimate $u$}\\
		Recall the equation
		\begin{align*}
			u_t-\mu\left(\frac{u_x}{b_{\varepsilon}}\right)_x=-(p(v,\theta))_x+\mu((\frac{1}{v}-\frac{1}{b_{\varepsilon}}) w_x)_x=:\partial_xF(v,w,\theta).
		\end{align*}
		Applying Lemma \ref{lemma} with the equation above for $u$, we obtain that 
		\begin{align*}
			%	\int_0^t\int|u_x(\tau,y)|^2dyd\tau\lesssim \int |u_0(y)|^2dy+\int_0^t\int|F(\tau,y)|^2dyd\tau
			\|u_x\|_{L^2_{T}}+	\|u_x\|_{X_T^{\frac{1}{2},2}}+\|u_x\|_{X_T^{\frac{3}{4},\infty}}\lesssim \|u_0\|_{L^2}+\|F\|_{L^2_{T}}+\|F\|_{X_T^{\frac{1}{2},2}}+\|F\|_{X_T^{\frac{3}{4},\infty}}.
		\end{align*}
		It is easy to check that 
		\begin{align*}
			&\|F\|_{L^2_{T}}\lesssim \sum_{\star\in\{L^2_{T},X_T^{\frac{1}{2},2},X_T^{\frac{1}{2},2}\}}\|\theta\|_{\star}+\|v-b_{\varepsilon}\|_{L^\infty}\sum_{\star\in\{L^2_{T},X_T^{\frac{1}{2},2}\}}\|w_x\|_{\star},\\
			&\|F\|_{X_T^{\frac{3}{4},\infty}}\lesssim \|\theta\|_{X_T^{\frac{3}{4},\infty}}+\|v-b_{\varepsilon}\|_{L^\infty}\|w_x\|_{X_T^{\frac{3}{4},\infty}}.
		\end{align*}
		By the interpolation inequality, we have $\|\theta\|_{X_T^{\frac{3}{4},\infty}}\lesssim \|\theta\|_{X_T^{\frac{1}{2},2}}+\|\theta_x\|_{X_T^{\frac{5}{6},\frac{6}{5}}}$.
		Hence 
		\begin{align*}
			&\sum_{\star\in\{L^2_{T},X_T^{\frac{1}{2},2},X_T^{\frac{3}{4},\infty}\}}\|u_x\|_{\star}\\
			&\ \ \leq \tilde C_1 (\|u_0\|_{L^2}+\|\theta\|_{L^2_{T}}+\|\theta\|_{X_T^{\frac{1}{2},2}}+\|\theta_x\|_{X_T^{\frac{5}{6},\frac{6}{5}}}+\|v-b_{\varepsilon,\eta}\|_{L^\infty}(\|w_x\|_{L^2_{T}}+\|w_x\|_{X_T^{\frac{1}{2},2}}))\\
			&\ \  \overset{\eqref{es2v}, \eqref{retheta}}\leq \tilde C_1(\|u_0\|_{L^2}+\tilde C_0\|\bar \theta_0\|_{\dot W^{\frac{1}{3},\frac{6}{5}}}+\tilde C_0\varepsilon_0B+\tilde C_0T^\frac{1}{4}B^2+\varepsilon_0 B)\\
			&\ \ \ \ \ \leq \frac{B}{3},
		\end{align*}
		provided $B\geq (10+\tilde C_0+\tilde C_1+\tilde C_2)(\|u_0\|_{L^2}+\|\bar \theta_0\|_{\dot W^{\frac{1}{3},\frac{6}{5}}})$, $\varepsilon_0\leq \frac{1}{100(1+\tilde C_0+\tilde C_1+\tilde C_2)^{10}}$ and $T\leq {\varepsilon_0^{10}}B^{-10}$.
		Combining this with \eqref{retheta}, we obtain 
		\begin{align*}
			\|(u,\theta)\|_{Z_T}\leq B.
		\end{align*}
		This implies that $\mathcal{M}$ maps $\mathcal{X}_{T,B}$ to itself. It remains to prove that $\mathcal{M}$ is contraction. Consider $(w_1,\vartheta_1), (w_2,\vartheta_2)\in \mathcal{X}_{T,B}$, let $v_m(t,x)=v_0(x)+\int_0^t \partial_x w_m(s,x)ds, m=1,2$. Then we have 
		\begin{equation}\label{v1v2}
			\begin{aligned}
				&\|v_m-b_{\varepsilon,\eta}\|_{L^\infty_{T}}\leq 2\varepsilon_0,\ \ \ \ \ \inf_{s\in[0,T]}	\inf_x v_m(s,x)\geq \frac{\lambda_0}{2},\\
				&\|v_1-v_2\|_{L^\infty_{T}}\leq \int_0^T \|\partial_x(w_1-w_2)(s)\|_{L^\infty}ds\lesssim T^\frac{1}{4} \|\partial_x (w_1-w_2)\|_{X_T^{\frac{3}{4},\infty}}.
			\end{aligned}
		\end{equation}
		Denote $(u_m,\theta_m)=\mathcal{M}(w_m,\vartheta_m), m=1,2$. We have 
		\begin{align*}
			&	\partial_t (u_1-u_2)-\kappa\left(\frac{\partial_x (u_1-u_2)}{b_{\varepsilon,\eta}}\right)_x=\partial_x(F(v_1,w_1,\theta_1)- F(v_2,w_2,\theta_2)),\\
			&	\partial_{t}(\theta_1-\theta_2)-\frac{\kappa}{\mathbf{c}}\left(\frac{\partial_x(\theta_1-\theta_2)}{b_{\varepsilon,\eta}}\right)_x=R(v_1,w_1,\vartheta_1)-R(v_2,w_2,\vartheta_2)+\partial_x(\tilde F(v_1,\vartheta_1)-\tilde F(v_2,\vartheta_2)).
		\end{align*}
		Applying Lemma \ref{lemma}  with $(f,\phi,C_\phi)=(\theta_1-\theta_2,b_{\varepsilon,\eta}^{-1},\eta^{-1}\|v_0\|_{L^\infty})$ to get
		\begin{align}
			&\sum_{\star\in\{L^2_{T},X_T^{\frac{1}{2},2}\}}\|\theta_1-\theta_2\|_{\star}+	\sum_{\star\in\{L^\frac{6}{5}_{T},X_T^{\frac{5}{6},\frac{6}{5}}\}}\|\partial_x(\theta_1-\theta_2)\|_{\star}\nonumber\\
			&\lesssim \sum_{\star\in\{L_{T}^\frac{6}{5},X_T^{\frac{5}{6},\frac{6}{5}}\}}\|\tilde F(v_1,\vartheta_1)-\tilde F(v_2,\vartheta_2)\|_{\star}+T^\frac{1}{4}\sum_{\star\in\{L^1_{T},X_T^{1,1}\}}\|R(v_1,w_1,\vartheta_1)-R(v_2,w_2,\vartheta_2)\|_{\star} .\label{diffthet}
		\end{align}
		Then \eqref{v1v2} yields
		\begin{align}
			&\sum_{\star\in\{L_{T}^\frac{6}{5},X_T^{\frac{5}{6},\frac{6}{5}}\}}\|(\tilde F(v_1,\vartheta_1)-\tilde F(v_2,\vartheta_2))\|_{\star}\nonumber\\
			&\lesssim \|{v_1}-{v_2}\|_{L^\infty}\sum_{\star\in\{L_{T}^\frac{6}{5},X_T^{\frac{5}{6},\frac{6}{5}}\}}\left\|\partial_x \vartheta_1\right\|_{\star} +\|v_1-b_{\varepsilon,\eta}\|_{L^\infty}\sum_{\star\in\{L_T^\frac{6}{5},X_T^{\frac{5}{6},\frac{6}{5}}\}}\left\|\partial_x (\vartheta_1-\vartheta_2)\right\|_{\star}\nonumber\\
			&\lesssim (T^\frac{1}{5}B+\varepsilon_0)\|(w_1,\vartheta_1)-(w_2,\vartheta_2)\|_{Z_T}.\label{diffF}
		\end{align}
		Moreover, we have 
		\begin{align*}
			|R(v_1,w_1,\vartheta_1)-R(v_2,w_2,\vartheta_2)|&\lesssim |\vartheta_1-\vartheta_2||\partial_xw_1|+|\vartheta_2||\partial_x(w_1-w_2)|+|\vartheta_2||\partial_xw_2||v_1-v_2|\\
			&\quad\quad+|\partial_x(w_1-w_2)|(|\partial_xw_1|+|\partial_x w_2|)+||\partial_xw_2|^2|v_1-v_2|.
		\end{align*}
		By H\"{o}lder's inequality, it is easy to check that 
		\begin{align}
			&\|R(v_1,w_1,\vartheta_1)-R(v_2,w_2,\vartheta_2)\|_{L^1_{T}}\nonumber\\
			&\lesssim \|\vartheta_1-\vartheta_2\|_{L^{2}_{T}}\|\partial_x w_1\|_{L^{2}_{T}}+(\|\vartheta_2\|_{L^{2}_{T}}+\|\partial_x w_1\|_{L^{2}_{T}}+\|\partial_x w_2\|_{L^{2}_{T}})\|\partial_x (w_1-w_2)\|_{L^{2}_{T}}\nonumber\\
			&\quad\quad\quad+(\|\vartheta_2\|_{L^{2}_{T}}^2+ \|\partial_x w_2\|_{L^{2}_{T}}^2)\|v_1-v_2\|_{L^\infty_{T}}\nonumber\\
			&\lesssim (B+B^2T^\frac{1}{5})\|(w_1,\vartheta_1)-(w_2,\vartheta_2)\|_{Z_T}.\label{diffR}
		\end{align}
		Similarly, one gets
		\begin{align*}
			&	\sum_{\star\in\{L^1_{T},X_T^{1,1}\}}\|R(v_1,w_1,\vartheta_1)-R(v_2,w_2,\vartheta_2)\|_{\star}\lesssim (B+B^2T^\frac{1}{5})\|(w_1,\vartheta_1)-(w_2,\vartheta_2)\|_{Z_T}.
		\end{align*}
		Combining this with \eqref{diffthet}, \eqref{diffF} and \eqref{diffR} to obtain 
		\begin{align}
			&\sum_{\star\in\{L^2_{T},X_T^{\frac{1}{2},2}\}}\|\theta_1-\theta_2\|_{\star}+	\sum_{\star\in\{L^\frac{6}{5}_{T},X_T^{\frac{5}{6},\frac{6}{5}}\}}\|\partial_x(\theta_1-\theta_2)\|_{\star}\nonumber\\
			&\leq \tilde C_3 (T^\frac{1}{5}B+\varepsilon_0+T^\frac{1}{4}(B+B^2T^\frac{1}{5}))\|(w_1,\vartheta_1)-(w_2,\vartheta_2)\|_{Z_T}\nonumber\\
			&\leq 2\tilde C_3\varepsilon_0\|(w_1,\vartheta_1)-(w_2,\vartheta_2)\|_{Z_T}.\label{diffrethe}
		\end{align}
		Then we estimate $u_1-u_2$. By Lemma \ref{lemma}, we get 
		\begin{align*}
			&\sum_{\star\in\{L^2_{T},\  X_T^{\frac{1}{2},2},\  X_T^{\frac{3}{4},\infty}\}}\|\partial_x (u_1-u_2)\|_{\star}\lesssim \sum_{\star\in\{L^2_{T},\  X_T^{\frac{1}{2},2},\ X_T^{\frac{3}{4},\infty}\}}\|F(v_1,w_1,\theta_1)- F(v_2,w_2,\theta_2)\|_{\star}.
		\end{align*}
		Note that 
		\begin{align*}
			|F(v_1,w_1,\theta_1)- F(v_2,w_2,\theta_2)|\lesssim |\theta_1-\theta_2|+|\theta_2||v_1-v_2|+|v_1-b_{\varepsilon,\eta}||\partial_x(w_1-w_2)|+|v_1-v_2||\partial_x w_2|.
		\end{align*}
		Using the interpolation $\|f\|_{X_T^{\frac{3}{4},\infty}}\lesssim \|f\|_{X_T^{\frac{1}{2},2}}\|\partial_x f\|_{ X_T^{\frac{5}{6},\frac{6}{5}}}$, we get
		\begin{align*}
			&	\sum_{\star\in\{L^2_{T},\  X_T^{\frac{1}{2},2},\ X_T^{\frac{3}{4},\infty}\}}\|F(v_1,w_1,\theta_1)- F(v_2,w_2,\theta_2)\|_{\star}\\
			&\ \ \ \lesssim \sum_{\star\in\{L^2_{T},\  X_T^{\frac{1}{2},2}\}}\|\theta_1-\theta_2\|_{\star}+\|\partial_x(\theta_1-\theta_2)\|_{ X_T^{\frac{5}{6},\frac{6}{5}}}+\|v_1-b_{\varepsilon,\eta}\|_{L^\infty}\sum_{\star\in\{L^2_{T},\  X_T^{\frac{1}{2},2},\ X_T^{\frac{3}{4},\infty}\}}\|\partial_x(w_1-w_2)\|_{\star}\\
			&\ \ \ \ \ \ +\Big(\sum_{\star\in\{L^2_{T},\  X_T^{\frac{1}{2},2}\}}\|\theta_2\|_{\star}+\|\partial_x\theta_2\|_{ X_T^{\frac{5}{6},\frac{6}{5}}}+\sum_{\star\in\{L^2_{T},\  X_T^{\frac{1}{2},2},\ X_T^{\frac{3}{4},\infty}\}}\|\partial_x w_2\|_{\star}\Big)\|v_1-v_2\|_{L^\infty}\\
			&\ \ \ \lesssim(\tilde C_3\varepsilon_0+BT^\frac{1}{5})\|(w_1,\vartheta_1)-(w_2,\vartheta_2)\|_{Z_T}.
		\end{align*}
		Then we get
		\begin{align*}
			&\sum_{\star\in\{L^2_{T},\  X_T^{\frac{1}{2},2},\  X_T^{\frac{3}{4},\infty}\}}\|\partial_x (u_1-u_2)\|_{\star}\\
			&\quad\quad\quad\leq \tilde C_4 (\tilde C_3\varepsilon_0+BT^\frac{1}{5})\|(w_1,\vartheta_1)-(w_2,\vartheta_2)\|_{Z_T}\leq \frac{1}{5}\|(w_1,\vartheta_1)-(w_2,\vartheta_2)\|_{Z_T},
		\end{align*}
		by taking  $B= (10+\sum_{m=1}^4\tilde C_m)(\|u_0\|_{L^2}+\|\bar \theta_0\|_{\dot W^{\frac{1}{3},\frac{6}{5}}})$, $\varepsilon_0= {100(1+\sum_{m=1}^4\tilde C_m)^{-10}}$ and $T= {\varepsilon_0^{10}}B^{-10}$. Combining this with \eqref{diffrethe} yields
		\begin{align*}
			\|\mathcal{M}(w_1,\vartheta_1)-\mathcal{M}(w_1,\vartheta_2)\|_{Z_T}\leq \frac{1}{2}\|(w_1-w_2,\vartheta_1-\vartheta_2)\|_{Z_T}.
		\end{align*}
		This completes the proof.
	\end{proof}\vspace{0.5cm}\\
	In the rest part of this section, we give a local version of  Theorem \ref{thmfull}. 
	Define the local  norms
	\begin{align*}
		\|h\|_{\tilde L^p}:=\sup_{z}\|h\|_{L^p([z,z+1])},\ \ \ 
		\|h\|_{\tilde W ^{s,p} }:=\sup_{z}\|h\chi_z\|_{\dot W ^{s,p}},\ \ \ \|f\|_{\tilde X^{\sigma,p}_T}=\sup_z\|f\chi_z\|_{\tilde X^{\sigma,p}_T},
	\end{align*}
	for $h:\mathbb{R}\to \mathbb{R}$, $f:\mathbb{R}^+\times\mathbb{R}\to\mathbb{R}$.  Here $\chi_z$ is a smooth cutoff function satisfying $\mathbf{1}_{[z-1,z+1]}\leq \chi_z\leq \mathbf{1}_{[z-2,z+2]}$. 
	Moreover, we denote $\|f\|_{\tilde L^p_{T}}=\|f\|_{L^p_t\tilde L^p_x}$.
	We introduce the following lemma.
	\begin{lemma}\label{lemlocal}
		Let $$
		g(t,x)=\int_0^t \int_{|y|\geq \rho} \partial_t K(t-\tau,y)F(\tau,x-y)dyd\tau,
		$$
		for some $\rho\in(0,1)$.	Then for any $0<T<1$, $p\geq 1$, $\sigma\in(0,1-\alpha)$,
		\begin{align*}
			\|g\|_{\tilde L^p_{T}}\lesssim_\rho \|F\|_{\tilde L^p_{T}},\ \ \ \ \|g\|_{\tilde X^{\sigma,p}_T}\lesssim _\rho \|F\|_{\tilde X^{\sigma,p}_T}.
		\end{align*}
	\end{lemma}
	\begin{proof}
		By Lemma \ref{lemheat} and H\"{o}lder's inequality, 
		\begin{align*}
			|g(t,x)|\lesssim &\int_0^t \int_{|y|\geq \rho} \frac{1}{((t-\tau)^\frac{1}{2}+|y|)^3}|F(\tau,x-y)|dyd\tau
			\\
			\lesssim &\sum_{n=0}^\infty\int_0^t \int_{|y|\in[\rho+n,\rho+n+1]}\frac{1}{((t-\tau)^\frac{1}{2}+|y|)^3}|F(\tau,x-y)|dyd\tau\\
			\lesssim& \sum_{n=0}^\infty \frac{1}{(\rho+n)^3}\min\{\|F\|_{\tilde L^p_{T}},\|F\|_{\tilde X_T^{\sigma,p}}\}\lesssim \rho^{-3}\min\{\|F\|_{\tilde L^p_{T}},\|F\|_{\tilde X_T^{\sigma,p}}\},
		\end{align*}
		for any $ x\in\mathbb{R}, \ t\in(0,T)$.
		
		Hence for any $z\in\mathbb{R}$, 
		\begin{align*}
			\|g\|_{L^p_{T}([z,z+1])}+\|g\|_{X_T^{\sigma,p}([z,z+1])}\lesssim \|g\|_{L^\infty_{T}([z,z+1])}\lesssim \min\{\|F\|_{\tilde L^p_{T}},\|F\|_{\tilde X_T^{\sigma,p}}\}.
		\end{align*}Similarly, one can check that $\sup_{0<s<t<T}s^{\sigma+\alpha}\frac{\|g(t)-g(s)\|_{\tilde L^p}}{(t-s)^\alpha}\lesssim \|F\|_{\tilde X_T^{\sigma,p}}$.
		This completes the proof.
	\end{proof}
	
	The following is a local version of Lemma \ref{lemma}.
	\begin{lemma}
		Let $f$ be a solution to \eqref{eqpara} with initial data $f_0=\partial_x \bar f_0$. There exists $T>0$ such that 
		$$
		\sum_{\star\in\{\tilde L^2_{T},\tilde X_T^{\frac{1}{2},2}\}}\|f\|_{\star}+	\sum_{\star\in\{\tilde L^\frac{6}{5}_{T},\tilde X_T^{\frac{5}{6},\frac{6}{5}}\}}\|\partial_x f\|_{\star}\lesssim \|\bar f_0\|_{\tilde  W^{\frac{1}{3},\frac{6}{5}}}+ \sum_{\star\in\{\tilde L^\frac{6}{5}_{T},\tilde X_T^{\frac{5}{6},\frac{6}{5}}\}}\| F\|_{\star}+T^{\frac{1}{4}}\sum_{\star\in\{{\tilde L^1_{T}},{\tilde X_T^{1,1}}\}}\|R\|_{\star}.
		$$
		Moreover, if $R=0$, we have 
		$$\sum_{\star\in\{\tilde L^2_{T},\tilde X_T^{\frac{1}{2},2},\tilde X_T^{\frac{3}{4},\infty}\}}\|\partial_xf\|_{\star}\lesssim \|f_0\|_{\tilde L^2}+\sum_{\star\in\{\tilde L^2_{T},\tilde X_T^{\frac{1}{2},2},\tilde X_T^{\frac{3}{4},\infty}\}}\|F\|_{\star}.
		$$
	\end{lemma} 
	\begin{proof} Denote $\mathbb{I}_{\varepsilon}=\cup_{n=1}^\infty[a_n-\varepsilon,a_n+\varepsilon]$. 
		Fix $r>100$ such that $\mathbb{I}_{\varepsilon}\subset B_{r/10}=[-r/10,r/10]$. 
		We consider a compact interval $\mathbb{K}$, which is either $B_{r/2}$ or any $[z-\frac{1}{2},z+\frac{1}{2}]\subset \mathbb{R}\backslash B_{r/2}$.
		We divide the solution into two parts  $f=f_{1}+f_{2}$, such that 
		\begin{align*}
			\begin{cases}
				\partial_x f_1-\partial_x(\phi(x)\partial_x f_1)=\partial_x (F\chi)+R\chi,\\
				f_1(0,x)=\partial_x (\bar f_0(x)\chi(x)),
			\end{cases}	\ \ \ \ 
			\begin{cases}
				\partial_x f_2-\partial_x(\phi(x)\partial_x f_2)=\partial_x (F(1-\chi))+R(1-\chi),\\
				f_2(0,x)=\partial_x (\bar f_0(x)(1-\chi)(x)).
			\end{cases}
		\end{align*}
		Here $\chi$ is a smooth cutoff function satisfying $$\begin{cases}
			&\mathbf{1}_{B_r}\leq \chi\leq \mathbf{1}_{B_{2r}},\ \ \text{if}\ \mathbb{K}=B_{r/2},\\
			&\mathbf{1}_{[z-1,z+1]}\leq \chi\leq \mathbf{1}_{[z-2,z+2]},\ \ \text{if}\ \mathbb{K}=[z-\frac{1}{2},z+\frac{1}{2}]\subset \mathbb{R}\backslash B_{r/2}.
		\end{cases}$$ We first consider  the main term $f_1$. Applying  Lemma \ref{lemma} to  obtain that 
		\begin{align*}
			\sum_{\star\in\{ L^2_{T},X_T^{\frac{1}{2},2}\}}\|f_1\|_{\star}+	\sum_{\star\in\{ L^\frac{6}{5}_{T},X_T^{\frac{5}{6},\frac{6}{5}}\}}\|\partial_x f_1\|_{\star}\lesssim \|\bar  f_0\chi\|_{\dot  W^{\frac{1}{3},\frac{6}{5}}}+ \sum_{\star\in\{ L^\frac{6}{5}_{T},X_T^{\frac{5}{6},\frac{6}{5}}\}}\| F\chi\|_{\star}+T^{\frac{1}{4}}\sum_{\star\in\{ L^1_{T},X_T^{1,1}\}}\|R\chi\|_{\star}.
		\end{align*}
		Note that $\chi$ is compactly supported. Hence $\|g\chi\|_{L^p}\lesssim \|g\|_{\tilde L^p}$. Then we obtain 
		\begin{align}\label{f1local}	&		\
			\sum_{\star\in\{ L^2_{T},X_T^{\frac{1}{2},2}\}}\|f_1\|_{\star}+	\sum_{\star\in\{ L^\frac{6}{5}_{T},X_T^{\frac{5}{6},\frac{6}{5}}\}}\|\partial_x f_1\|_{\star}\lesssim \|\bar  f_0\|_{\tilde  W^{\frac{1}{3},\frac{6}{5}}}+ \sum_{\star\in\{ L^\frac{6}{5}_{T},\tilde X_T^{\frac{5}{6},\frac{6}{5}}\}}\| F\|_{\star}+T^{\frac{1}{4}}\sum_{\star\in\{\tilde  L^1_{T},\tilde X_T^{1,1}\}}\|R\|_{\star}.
		\end{align}
		Similarly, when $R=0$ we have 
		\begin{align*}
			\sum_{\star\in\{\tilde L^2_{T},\tilde X_T^{\frac{1}{2},2},\tilde X_T^{\frac{3}{4},\infty}\}}\|\partial_xf_1\|_{\star}\lesssim \|f_0\|_{\tilde L^2}+\sum_{\star\in\{\tilde L^2_{T},\tilde X_T^{\frac{1}{2},2},\tilde X_T^{\frac{3}{4},\infty}\}}\|F\|_{\star}.
		\end{align*}
		Then we estimate the perturbation part $f_2$. Note that for any $y\in\operatorname{supp}(1-\chi)$ and any  $x\in \mathbb{K}$, there holds $|x-y|\geq \frac{1}{2}$, which removes the singularity of kernel. The good decay property of the kernel helps us to control local norms (see Lemma \ref{lemlocal}). We consider $\partial_x f_{2,N}=\int_0^t \int \partial_1^2\mathbf{H} (t-\tau,x-y,x)F(\tau,y)(1-\chi(y))dyd\tau$ for an example. Other terms can be done similarly.
		Applying Lemma \ref{lemlocal} we obtain that 
		\begin{align*}
			\sum_{\star\in\{\tilde L^p_{T}(\mathbb{K}),\tilde X_T^{\sigma,p}\}}\|\partial_x f_{2,N}\|_{\star}
			\lesssim \sum_{\star\in\{\tilde L^p_{T},\tilde X_T^{\sigma,p}\}}\|F\|_{\star},\ \ \ p\geq 1.
		\end{align*}
		Similarly, we obtain that 
		\begin{align*}
			\sum_{\star\in\{\tilde{L}^2_{T},\tilde{X}_T^{\frac{1}{2},2}\}}\|f_2\|_{ \star}+\sum_{\star\in\{ \tilde{L}^\frac{6}{5}_{T},\tilde{X}_T^{\frac{5}{6},\frac{6}{5}}\}}	\|\partial_x f_2\|_{\star}\lesssim \|\bar  f_0\|_{\tilde  W^{\frac{1}{3},\frac{6}{5}}}+\sum_{\star\in\{\tilde  L^\frac{6}{5}_{T},\tilde X_T^{\frac{5}{6},\frac{6}{5}}\}} \| F\|_{\star}+T^{\frac{1}{4}}\sum_{\star\in\{\tilde  L^1_{T},\tilde X_T^{1,1}\}}\|R\|_{\star},
		\end{align*}
		and 
		\begin{align*}
			\sum_{\star\in\{\tilde L^2_{T},\tilde X_T^{\frac{1}{2},2},\tilde X_T^{\frac{3}{4},\infty}\}}\|\partial_xf_2\|_{\star}\lesssim \|f_0\|_{\tilde L^2}+\sum_{\star\in\{\tilde L^2_{T},\tilde X_T^{\frac{1}{2},2},\tilde X_T^{\frac{3}{4},\infty}\}}\|F\|_{\star},\ \ \ \text{if}\ R=0.
		\end{align*}
		%\begin{align*}
		%	&	\|\partial_x^l f_2\|_{\tilde L^p_{T}(\mathbb{K})}\lesssim \|\bar f_0\|_{\tilde  W^{\frac{p-2}{p}+l,p}}+ \| F\|_{\tilde  L^\frac{3p}{3+(1-l)p}_{T}}+T^{\frac{3-(1+l)p}{2p}}\|\tilde R \|_{\tilde  L^1_{T}}+T^\frac{1}{3}\|\partial_xf\|_{\tilde L^\frac{3p}{3+(1-l)p}_{T}}.
		%\end{align*}
		Combining this with \eqref{f1local}, we obtain the result.
	\end{proof}\vspace{0.5cm}\\
	Denote 
	\begin{align*}	\|(w,\vartheta)\|_{\tilde Z_T}:=\sum_{\star\in\{\tilde{L}_T^2,\tilde{X}_T^{\frac{1}{2},2},\tilde{X}_T^{\frac{3}{4},\infty}\}}\|\partial_x w\|_{\star}+\sum_{\star\in\{\tilde{L}_T^2,\tilde{X}_T^{\frac{1}{2},2}\}}\|\vartheta\|_{\star}+\sum_{\star\in\{\tilde{L}_T^{\frac{6}{5}},\tilde{X}_T^{\frac{5}{6},\frac{6}{5}}\}}\|\partial_x \vartheta\|_{\star}.
	\end{align*}
	Following the proof of Theorem \ref{thmfull}, we obtain the following result.
	\begin{thm}\label{thmloc}There exists $\varepsilon_0>0$ such that if initial data $(v_0,u_0,\theta_0)$ satisfies \eqref{loccon} and \eqref{conccc}, then the system \eqref{cpns} admits a unique local solution $(v,u,\theta)$ in $[0,T]$ satisfying 
		\begin{align*}
			\inf_{t\in[0,T]}\inf_xv(t,x)\geq \frac{\lambda_0}{2},\ \ \ \|v\|_{Y_T}\leq 2\|v_0\|_{L^\infty},\ \ \ \ \|(u,\theta)\|_{\tilde Z_T}\leq B.	
		\end{align*}
		for some $T,B>0$.
	\end{thm}
	
	\section{Appendix}
	Consider the heat equation
	\begin{align}\label{eqparacons}
		\partial_t f-\partial_x^2 f=\partial_x F+R,\ \ \ \ \ \ \ f(0,x)=\partial_x \bar f_0(x)+\tilde f_0(x).
	\end{align}
	We introduce the following lemma.
	\begin{lemma}\label{lemap}
		Let $f$ be a solution to \eqref{eqparacons}, for any $T\in(0,1)$, there holds 
		\begin{align*}
			&	\|f\|_{L^2_{T}}+\|\partial_x f\|_{L^\frac{6}{5}_{T}}\lesssim \|\bar f_0\|_{\dot W^{\frac{1}{3},\frac{6}{5}}}+T^\frac{1}{3}\|\tilde f_0\|_{L^\frac{6}{5}}+\|F\|_{L^\frac{6}{5}_{T}}+T^\frac{1}{4}\|R\|_{L^1_{T}},\\
			&	\|f\|_{L^6_{T}}+\|\partial_x f\|_{L^2_{T}}\lesssim\|\partial_x \bar f_0\|_{L^2}+\|\tilde f_0\|_{L^2}+\|F\|_{L^2_{T}}+T^\frac{1}{2}\|R\|_{L^2_{T}}.
		\end{align*}
	\end{lemma}
	\begin{proof}
		The solution has formula 
		\begin{align*}
			f(t,x)=&\int_{\mathbb{R}} \partial_x\mathbf{K} (t,x-y) \bar f_0(y) dy+\int_{\mathbb{R}} \mathbf{K} (t,x-y) \tilde  f_0(y) dy\\
			&+\int_0^t \int_{\mathbb{R}} \partial_x \mathbf{K} (t-\tau,x-y)F(\tau,y) dyd\tau+ \int_0^t \int_{\mathbb{R}}  \mathbf{K} (t-\tau,x-y)R(\tau,y) dyd\tau\\
			:=&f_{L1}(t,x)+f_{L2}(t,x)+f_N(t,x)+f_R(t,x).
		\end{align*}
		We first estimate $f_{L1}$. By  Parseval’s identity  and Sobolev inequality, it is easy to check that  
		\begin{equation}\label{paraL2}
			\begin{aligned}
				&	\|f_{L1} \|_{L^2_{\infty}}^2=\int_0^\infty \int_{\mathbb{R}} |\xi|^2|e^{-|\xi|^2\tau}\mathcal{F}(\bar f_0)(\xi)|^2d\xi d\tau=\frac{1}{2}\int_{\mathbb{R}}| \mathcal{F}(\bar f_0)(\xi)|^2d\xi=\frac{1}{2}\|\bar f_0\|_{L^2}^2\lesssim \|\bar f_0\|_{\dot W^{\frac{1}{3},\frac{6}{5}}},\\
				&	\|\partial_x f_{L1} \|_{L^2_{\infty}}^2=\int_0^\infty \int_{\mathbb{R}} |\xi|^4|e^{-|\xi|^2\tau}\mathcal{F}(\bar f_0)(\xi)|^2d\xi d\tau=\frac{1}{2}\int_{\mathbb{R}}|\xi|^2| \mathcal{F}(\bar f_0)(\xi)|^2d\xi=\frac{1}{2}\|\partial_x\bar f_0\|_{L^2}^2.
			\end{aligned}
		\end{equation}
		Let $(l,p)=(1,\frac{6}{5})\ \text{or}\ (0,6)$.
		Note that 
		\begin{align*}
			\partial_x^lf_{L1}(t,x)=\int_{\mathbb{R}} \partial_x^{1+l}\mathbf{K} (t,x-y) (\bar f_0(y)-\bar f_0(x)) dy.
		\end{align*}
		Set $K_0(t,x)=|\partial_x^{1+l}\mathbf{K}(t,x)|$. Denote $\Delta_z g(x)=g(x)-g(x-z)$. Then 
		by a change of variable,
		\begin{align*}
			|\partial_x^lf_{L1}(t,x)|\lesssim \int_{\mathbb{R}} K_0(t,z) |\Delta_z\bar f_0(x)|dz.
		\end{align*}
		Then we have 
		\begin{align*}
			\|	\partial_x^lf_{L1}\|_{L^p_{T}}^p&=\int _0^\infty \int_{\mathbb{R}} \left(\int_{\mathbb{R}} K_0(t,z) |\Delta_z\bar f_0(x)|dz\right)^pdxdt\\
			&\lesssim \int _0^\infty \left(\int_{\mathbb{R}} K_0(t,z) \|\Delta_z \bar f_0\|_{L^p_x}dz\right)^pdt\\
			&\lesssim \int_0^\infty\left(\int_{\mathbb{R}} K_0(t,z)^\frac{p}{p-1}(1+|z|^2t^{-1})^\frac{\gamma p}{p-1}dz\right)^{p-1}\left(\int_{\mathbb{R}} \frac{\|\Delta_z\bar f_0\|_{L^p}^p}{(1+|z|^2t^{-1})^{\gamma p}}dz\right)dt,
		\end{align*}
		where we applied Minkowski's inequality in the first inequality, and H\"{o}lder's inequality in the second inequality.
		It is easy to check that 
		\begin{align*}
			\left(\int_{\mathbb{R}} K_0(t,z)^\frac{p}{p-1}(1+|z|^2t^{-1})^\frac{\gamma p}{p-1}dz\right)^{p-1}\lesssim t^{\frac{-(l+1)p-1}{2}},
		\end{align*}
		where we fix $\gamma=\frac{1+l}{2}$.
		Hence,
		\begin{align}
			\|	\partial_x ^lf_L(t,x)\|_{L^p_{T}} ^p&\lesssim \int_0^\infty t^{\frac{-(l+1)p-1}{2}}\left(\int_{\mathbb{R}} \frac{\|\Delta_z \bar f_0(x)\|_{L^p}^p}{(1+|z|^2t^{-1})^{\gamma p}}dz\right)dt\nonumber\\&\lesssim \int_{\mathbb{R}} \frac{\|\Delta_z \bar f_0(x)\|_{L^p}^p}{|z|^{(1+l)p-1}}dz=c \|\bar f_0\|_{\dot W^{\frac{p-2}{p}+l,p}}^p. \label{aplinearLp}
		\end{align}
		Hence we obtain that 
		\begin{equation}\label{e222}
			\begin{aligned}
				\| f_{L1}\|_{L^6_{T}}\lesssim  \|\bar f_0\|_{\dot W^{\frac{2}{3},6}}\lesssim \|\partial_x \bar f_0\|_{L^2},\ \ \ \ \ \ 
				\| \partial_xf_{L1}\|_{L^\frac{6}{5}_{\infty}}\lesssim & \|\bar f_0\|_{\dot W^{\frac{1}{3},\frac{6}{5}}}.
			\end{aligned}
		\end{equation}
		Combining this with \eqref{paraL2} and \eqref{e222}, one has 
		\begin{equation}\label{fL1}
			\begin{aligned}
				&\|f_{L1}\|_{L^2_{T}}+\|\partial_x f_{L1}\|_{L^\frac{6}{5}_{T}}\lesssim \|\bar f_0\|_{\dot W^{\frac{1}{3},\frac{6}{5}}},\\
				&	\|f_{L1}\|_{L^6_{T}}+\|\partial_x f_{L1}\|_{L^2_{T}}\lesssim \|\partial_x \bar f_0\|_{L^2}.
			\end{aligned}
		\end{equation}
		Then we estimate  $f_{L2}$. By Parseval’s identity, we obtain that 
		\begin{align*}
			&	\|\partial_x f_{L2}\|_{L^2_{T}}\lesssim \int_0^\infty \int_{\mathbb{R}} |\xi|^2|e^{-|\xi|^2\tau}\mathcal{F}(\tilde  f_0)(\xi)|^2d\xi d\tau=\frac{1}{2}\int_{\mathbb{R}}| \mathcal{F}(\tilde  f_0)(\xi)|^2d\xi=\frac{1}{2}\|\tilde  f_0\|_{L^2}^2.
		\end{align*}
		Moreover, by H\"{o}lder's inequality and Young's inequality,
		\begin{align*}
			&\|f_{L2}\|_{L^{2}_{T}}\lesssim \|\mathbf{K}\|_{L^2_{t,T}L^\frac{3}{2}_x}\|\tilde f_0\|_{L^\frac{6}{5}}\lesssim T^\frac{1}{3}\|\tilde f_0\|_{L^\frac{6}{5}},\\
			&\|f_{L2}\|_{L^{6}_{T}}\lesssim\left\|\int_{\mathbb{R}} \|\mathbf{K}(x-y)\|_{L^6_{t,T}}|\tilde f_0(y)|dy\right\|_{L^6_x}\lesssim \left\|\int_{\mathbb{R}} |\tilde f_0(y)|\frac{dy}{|x-y|^\frac{2}{3}}\right\|_{L^6_x}\lesssim  \|\tilde f_0\|_{L^2},\\
			&\|\partial_x f_{L2}\|_{L^\frac{6}{5}_{T}}\lesssim\|\partial_x \mathbf{K}\|_{L^\frac{6}{5}_{t,T}L^1_x}\|\tilde f_0\|_{L^\frac{6}{5}}\lesssim T^\frac{1}{3}\|\tilde f_0\|_{L^\frac{6}{5}}.
		\end{align*}
		Hence we obtain that 
		\begin{equation}\label{fL2}
			\begin{aligned}
				&\|f_{L2}\|_{L^2_{T}}+\|\partial_x f_{L2}\|_{L^\frac{6}{5}_{T}}\lesssim T^\frac{1}{3}\|\tilde  f_0\|_{L^\frac{6}{5}},\\	&\|f_{L2}\|_{L^6_{T}}+\|\partial_x f_{L2}\|_{L^2_{T}}\lesssim \|\tilde  f_0\|_{L^2}.
			\end{aligned}
		\end{equation}
		Then we estimate the forced term $f_N$. We can write
		$$
		\partial_t f_N-\partial_x^2 f_N =\partial_x F_0,\ \ \ \text{in}\ \mathbb{R}\times \mathbb{R}, \ \ \ \ \text{where}\ \ F_0=F\mathbf{1}_{t>0}.
		$$
		We take Fourier transform in $(t,x)$ and obtain that 
		\begin{align*}
			f_N=\mathcal{F}^{-1}_{t,x}\left(\frac{\xi\mathcal{F}_{T} ( F_0)(\tau,\xi)}{i\tau+\xi^2}\right),
		\end{align*}
		where we denote $\mathcal{F}_{t,x}$ the Fourier transform with respect to $(t,x)$, and $\mathcal{F}^{-1}_{T}$ its inverse. It  follows from the parabolic Calderon-Zygmund theory (see \cite{Stein}) that 
		$$
		\left\|\partial_x f_N\right\|_{L^p_{T}}\lesssim \|F_0\|_{L^p_{T}}\lesssim \|F\|_{L^p_{T}},\ \ \ \ p\in(1,\infty).
		$$
		On the other hand, we have 
		$
		\left|\partial_x\mathbf{K}( t,x)\right|\lesssim {(t^\frac{1}{2}+|x|)^{-2}}.
		$
		Then by Young's inequality,
		\begin{align*}
			&\left\|	\int_0^t \int_{\mathbb{R}}\partial_x\mathbf{K}( t-\tau,x-y)F(\tau,y) dyd\tau\right\|_{L^p_{T}}\\
			&\lesssim \left\|\int_{\mathbb{R}} \|\partial_x\mathbf{K}(\cdot ,x-y)\|_{L^\frac{3}{2}_{t,T}}\| F(y)\|_{L^\frac{3p}{3+p}_{t,T}}dy\right\|_{L^p_x}\lesssim \left\|\int_{\mathbb{R}}\| F(y)\|_{L^\frac{3p}{3+p}_{t,T}}\frac{dy}{|x-y|^\frac{2}{3}}\right\|_{L^p_x}\\
			&\lesssim \| F\|_{L^\frac{3p}{3+p}_{T}}.
		\end{align*} 
		Hence we obtain that for $l=0,1$,
		\begin{align}\label{N}
			\|\partial_x^l f_N\|_{L^p_{T}}=\|\int_0^t \int_{\mathbb{R}}\partial_x^{1+l}\mathbf{K}( t-\tau,x-y)F(\tau,y) dyd\tau\|_{L^p_{T}}\lesssim \| F\|_{L^\frac{3p}{3+(1-l)p}_{T}}.
		\end{align}
		Finally, we estimate $f_R$. By H\"{o}lder's inequality we have 
		$$
		\|\partial_x^l f_R\|_{L^p_{T}}\lesssim \|\partial_x ^l \mathbf{K}\|_{L^p_{T}}\|R\|_{L^1_{T}}\lesssim T^\frac{1}{4}\|R\|_{L^1_{T}},\ \ (l,p)\in\{(0,2),(1,6/5)\},$$
		and $$	\|\partial_x^l f_R\|_{L^p_{T}}\lesssim \|\partial_x ^l \mathbf{K}\|_{L^\frac{2p}{2+p}_{T}}\|R\|_{L^1_{T}}\lesssim T^\frac{1}{2} \|R\|_{L^2_{T}},\ \ (l,p)\in\{(0,6),(1,2)\}.
		$$
		Combining this with \eqref{fL1}, \eqref{fL2} and \eqref{N}, we finish the proof.
	\end{proof}\vspace{0.5cm}\\
	Consider the parabolic equation with jump coefficients,
	\begin{equation}\label{jueqle}
		\begin{aligned}
			&\partial_t f^\pm - c_\pm f_{xx}^\pm=F ,\ \ \ \text{in}\ (t,x)\in(0,T)\times \mathbb{R}^\pm,\\
			& f^+(t,0)=f^-(t,0),\quad\quad \ \ c_+\partial_x f^+(t,0)-c_-\partial_x f^-(t,0)=0,\ \ t\in(0,T),\\
			&f(0,x)=f_0(x),\ \ \ \ x\in\mathbb{R}.
		\end{aligned}
	\end{equation}
	\begin{lemma}\label{lemformula}
		Solution to \eqref{jueqle} has formula 
		$f=f^+\mathbf{1}_{x\geq 0}+f^-\mathbf{1}_{x< 0}$ , where  
		$$		f^\pm (t,x)=f_M^\pm(t,x)+f_B^\pm(t,x).
		$$
		with 
		\begin{align}
			f_M^\pm(t,x)=&\int_{\mathbb{R}^\pm}( \mathbf{K}(c_\pm t,x-y)-\mathbf{K}(c_\pm t,x+y))f_0(y) dy\nonumber\\
			&+\int_0^t \int_{\mathbb{R}^\pm}( \mathbf{K}(c_\pm (t-\tau),x-y)-\mathbf{K}(c_\pm (t-\tau),x+y)) F(\tau,y) dyd\tau,\label{forM}\\
			f_{B}^\pm(t,x)=&\frac{-2\sqrt{c_\pm}}{\sqrt{c_+}+\sqrt{c_-}}\int_0^t\mathbf{K}(c_\pm(t-\tau),x)(c_+\partial_xf_M^+(\tau,0)-c_-\partial_xf_M^-(\tau,0))d\tau.\label{forB}
		\end{align}
	\end{lemma}
	\begin{proof}
		By superposition principle, we have $	f^\pm (t,x)=f_M^\pm(t,x)+f_B^\pm(t,x)$, where 
		$f_M^\pm$ solves the following equation on half space 
		\begin{align*}
			\begin{cases}
				&\partial_t f_M^+-c_+\partial_{x}^2f_M^+=F,\ \ \ \text{in}\ (0,T)\times \mathbb{R}^+,\\
				&f_M^+(0,x)=f_0(x),\ \ \ x\in\mathbb{R}^+,\\
				&f_M^+(t,0)=0,\ \ \ t\in(0,T).
			\end{cases}
			\begin{cases}
				&\partial_t f_M^--c_-\partial_{x}^2f_M^-=F,\ \ \ \text{in}\ (0,T)\times \mathbb{R}^-,\\
				&f_M^-(0,x)=f_0(x),\ \ \  x\in\mathbb{R}^-,\\
				&f_M^-(t,0)=0,\ \ \ t\in(0,T).
			\end{cases}
		\end{align*}
		And $f_B^\pm$ solves the system
		\begin{equation}\label{eqB}
			\begin{aligned}
				&\partial_t f_B^\pm-c_\pm\partial_{x}^2f_B^\pm=0,\ \ \ \text{in}\ (0,T)\times \mathbb{R}^\pm,\\
				&f_B^+(0,x)=0\ \text{in}\ \mathbb{R}^+,\ \ f_B^-(0,x)=0\ \text{in}\ \mathbb{R}^-,\\
				&f_B^+(t,0)=f_B^-(t,0),\ \ \ t\in(0,T),\\
				& c_+\partial_x f_B^+(t,0)-c_-\partial_x f_B^-(t,0)=-(c_+\partial_x f_M^+(t,0)-c_-\partial_x f_M^-(t,0)),\ \ \ t\in(0,T).
			\end{aligned}
		\end{equation}
		Indeed, integrate by parts we have 
		\begin{align*}
			&\int_0^t \int_{\mathbb{R}^+}(\mathbf{K}(c_+(t-\tau),x-y)-\mathbf{K}(c_+(t-\tau),x+y))F(\tau,y) dyd\tau\\
			&=\int_0^t \int_{\mathbb{R}^+}(\mathbf{K}(c_+(t-\tau),x-y)-\mathbf{K}(c_+(t-\tau),x+y))(\partial_t f_M^+-c_+\partial_x ^2 f_M^+)(\tau,y) dyd\tau\\
			&=f_M^+(t,x)-\int_{\mathbb{R}^+}(\mathbf{K}(c_+t,x-y)-\mathbf{K}(c_+t,x+y))f_0(y)dy\\
			&\quad\quad+\int_0^t \int_{\mathbb{R}^+}\partial_t(\mathbf{K}(c_+(t-\tau),x-y)-\mathbf{K}(c_+(t-\tau),x+y))f_M^+(\tau,y)dyd\tau\\
			&\quad\quad-c_+\int_0^t \int_{\mathbb{R}^+}\partial_x^2(\mathbf{K}(c_+(t-\tau),x-y)-\mathbf{K}(c_+(t-\tau),x+y))f_M^+(\tau,y)dyd\tau\\
			&=f_M^+(t,x)-\int_{\mathbb{R}^+}(\mathbf{K}(c_+t,x-y)-\mathbf{K}(c_+t,x+y))f_0(y)dy,
		\end{align*}
		where we use the fact that 
		$$
		\partial_t \mathbf{K}(c_+t,x)-c_+\partial_x^2\mathbf{K}(c_+t,x)=0,\ \ \forall \ t>0,x\in\mathbb{R}.
		$$
		We also get similar estimate for $f^-_M$. Hence we obtain formula \eqref{forM}. \\
		It remains to check that  formula \eqref{forB} satisfies equation \eqref{eqB}. For simplicity, we only check the Neumann boundary condition 
		\begin{align}\label{neubdhh}
			c_+\partial_x f_B^+(t,0)-c_-\partial_x f_B^-(t,0)=-(c_+\partial_x f_M^+(t,0)-c_-\partial_x f_M^-(t,0)):=h(t),\ \ \ t\in(0,T).	
		\end{align}
		And other conditions are trivial.
		Note that 
		\begin{equation}\label{neubd}
			\begin{aligned}
				c_+\partial_x f_B^+(t,0)-c_-\partial_x f_B^-(t,0)=&\frac{2c_+^\frac{3}{2}}{\sqrt{c_+}+\sqrt{c_-}}\lim_{x\to 0^+}\int_0^t\partial_x\mathbf{K}(c_+(t-\tau),x)h(\tau)d\tau\\
				&-\frac{2c_-^\frac{3}{2}}{\sqrt{c_+}+\sqrt{c_-}}\lim_{x\to 0^-}\int_0^t\partial_x\mathbf{K}(c_-(t-\tau),x)h(\tau)d\tau.
			\end{aligned}
		\end{equation}
		Note that the heat kernel has scaling $\mathbf{K}(t,x)=t^{-\frac{1}{2}}\mathbf{K}(1,\frac{x}{\sqrt{t}})$.
		\begin{align*}
			\int_0^t \partial_x  \mathbf{K}(c_+(t-\tau),x)h(\tau)d\tau&=-\int_0^t \frac{x}{2c_+\tau}K(c_+\tau,x)h(t-\tau)d\tau\\
			&=-\int_0^t \frac{x}{2(c_+\tau)^\frac{3}{2}}K(1,\frac{x}{\sqrt{c_+\tau}})h(t-\tau)d\tau.
		\end{align*}
		Let $\tilde \tau=\frac{x}{\sqrt{c_+\tau}}$, we obtain that 
		$$
		\int_0^t \partial_x  \mathbf{K}(c_+(t-\tau),x)h(\tau)d\tau=c_+^{-1}\int_{\frac{x}{\sqrt{c_+t}}}^\infty\mathbf{K}(1,\tilde \tau)h(t-\frac{x^2}{c_+\tilde\tau^2})d\tilde \tau.
		$$
		We can write 
		\begin{align*}
			\int_{\tilde t}^\infty\mathbf{K}(1,\tilde \tau)h(t-\frac{x^2}{c_+\tilde\tau^2})d\tilde \tau=\left(\int_{\tilde t^{1/2}}^\infty+\int_{\tilde t}^{\tilde t^\frac{1}{2}}\right)\mathbf{K}(1,\tilde \tau)h(t-\frac{x^2}{c_+\tilde\tau^2})d\tilde \tau,
		\end{align*}
		where we denote $\tilde t=\frac{x}{\sqrt{c_+t}}$.
		We have  $\tilde t\to 0$ when $x\to 0$, and the integral on $[\tilde t,\tilde t^\frac{1}{2}]$ tends to 0. Hence,
		$$
		\lim_{x\to 0^+}\int_0^t\partial_x\mathbf{K}(c_+(t-\tau),x)h(\tau)d\tau=c_+^{-1}\int_0^\infty \mathbf{K}(1,\tilde \tau)d\tilde \tau h(t)=\frac{1}{2c_+}h(t).
		$$
		Similarly, one has 
		\begin{align}\label{neubd2}
			\lim_{x\to 0^-}\int_0^t\partial_x\mathbf{K}(c_-(t-\tau),x)h(\tau)d\tau=-\frac{1}{2c_-}h(t).
		\end{align}
		Then we obtain \eqref{neubdhh} from \eqref{neubd}-\eqref{neubd2}.
		This completes the proof.
	\end{proof} \vspace{0.5cm}\\
	Consider the following parabolic equation with variable coefficient,
	\begin{equation}\label{conticoeff}
		\begin{aligned}
			&\partial_t f-\partial_x (\phi(x)f_x)=F,\quad\quad \ \ (t,x)\in (0,T)\times \mathbb{R},\\
			&	f(0,x)=f_0(x).
		\end{aligned}
	\end{equation}
	\begin{lemma}\label{forcont} the
		Solution $f$ to \eqref{conticoeff} satisfies
		\begin{align*}
			f(t,x)=&\int_{\mathbb{R}}\mathbf{H} (t,x-y,x)f_0(y)dy+\int_0^t\int_{\mathbb{R}} \mathbf{H} (t-\tau,x-y,x)F(\tau,y)dyd\tau\\
			&-\int_0^t \int_{\mathbb{R}} \partial_1 \mathbf{H}(t-\tau,x-y,x)(\phi(x)-\phi(y)) f_x(\tau,y)dyd\tau,
		\end{align*}
		where we denote $\mathbf{H} (t,x,y)=\mathbf{K}({ \phi(y)}{t},x)$,  $\partial_1\mathbf{H}(t,x,y)=\partial_x\mathbf{H}(t,x,y)$.
	\end{lemma}
	\begin{proof}
		Integrate by parts we obtain 
		\begin{align*}
			&\int_0^t\int_{\mathbb{R}} \mathbf{H} (t-\tau,x-y,x)F(\tau,y)dyd\tau\\
			&=	\int_0^t\int_{\mathbb{R}} \mathbf{H} (t-\tau,x-y,x)(\partial_t f-\partial_x(\phi f_x))(\tau,y)dyd\tau\\
			&=f(t,x)-\int_{\mathbb{R}} \mathbf{H}(t,x-y,x)f_0(y) dy +\int_0^t \int_{\mathbb{R}} \partial_t \mathbf{H} (t-\tau,x-y,x) f(\tau,y)dyd\tau\\
			&\ \ \ +\int_0^t \int_{\mathbb{R}} \partial_1 \mathbf{H}(t-\tau,x-y,x)(\phi(x)-\phi(y)) f_x(\tau,y)dyd\tau\\
			&\ \ \ -\phi(x)\int_0^t \int_{\mathbb{R}} \partial_1^2 \mathbf{H} (t-\tau,x-y,x) f(\tau,y)dyd\tau.
		\end{align*}
		It is easy to check that 
		$$
		\partial_t \mathbf{H} (t-\tau,x-y,x)-\phi(x)\partial_1^2 \mathbf{H} (t-\tau,x-y,x)=0.
		$$
		Then we obtain the result.
	\end{proof}
	\vspace{0.5cm}\\
	\textbf{Acknowledgements:} 
	Q.H.N.  is supported by the Academy of Mathematics and Systems Science, Chinese Academy of Sciences startup fund, and the National Natural Science Foundation of China (No. 12050410257 and No. 12288201) and  the National Key R$\&$D Program of China under grant 2021YFA1000800. 
	%	\textbf{Data availability}: Data will be made available on reasonable request.

\end{document}